\documentclass[a4paper,11pt,reqno]{amsart}
\pdfoutput = 1
\usepackage[utf8]{inputenc}
\usepackage[T1]{fontenc}
\usepackage{lmodern}
\usepackage[draft=false,kerning=true]{microtype}
\usepackage{a4wide}
\usepackage{ifthen}

\usepackage{tikz}
\usetikzlibrary{calc}
\usetikzlibrary{decorations.pathreplacing}
\usepackage{pgfplots}

\usepackage[%
pdftex=true,      
colorlinks=false,   
bookmarks=true,          
bookmarksopen=true,     
bookmarksnumbered=false, 
pdfpagemode=UseOutlines,         
final
]{hyperref}  

\usepackage{stmaryrd} 

\usepackage{enumerate}
\usepackage{dcolumn}

\usepackage{colonequals}
\newcommand{\bfW}{\mathbf{W}}
\newcommand{\C}{\mathbb{C}}

\newcommand{\calC}{\mathcal{C}}
\newcommand{\calD}{\mathcal{D}}
\newcommand{\calF}{\mathcal{F}}
\newcommand{\calT}{\mathcal{T}}
\newcommand{\length}{\mathsf{length}}
\newcommand{\E}{\mathbb{E}}
\newcommand{\Var}{\mathbb{V}}
\newcommand{\lambdamax}{\lambda_{\mathrm{max}}}
\renewcommand{\P}{\mathbb{P}}

\newcommand{\N}{\mathbb{N}}
\DeclareMathOperator{\Res}{Res}

\newcommand{\Z}{\mathbb{Z}}
\newcommand{\R}{\mathbb{R}}
\newcommand{\TODO}[1]%
{\par\fbox{\begin{minipage}{0.9\linewidth}\textbf{TODO:} #1\end{minipage}}\par}
\newcommand{\abs}[2][]{\parentheses[#1]{\lvert}{\rvert}{#2}}

\newtheorem{theorem}{Theorem}

\newtheorem{lemma}{Lemma}[section]
\newtheorem{proposition}[lemma]{Proposition}
\theoremstyle{remark}
\newtheorem{remark}[lemma]{Remark}

\numberwithin{equation}{section}
\numberwithin{figure}{section}

\newcounter{constant}
\newcommand{\newc}{\refstepcounter{constant}c_{\arabic{constant}}}
\newcommand{\newinvisiblec}{\refstepcounter{constant}}
\newcounter{epsilon}
\newcommand{\neweps}{\refstepcounter{epsilon}\varepsilon_{\arabic{epsilon}}}

\newcommand{\hp}[1]{\llbracket{#1}\rrbracket}
\renewcommand{\MR}[1]{}

\RequirePackage{suffix}

\newcommand{\parentheses}[4][]%
{\ifthenelse{\equal{#1}{}}{\left#2}{\csname#1\endcsname#2}%
    {#4}\ifthenelse{\equal{#1}{}}{\right#3}{\csname#1\endcsname#3}}

\newcommand{\foperator}[1]{\ensuremath{%
    \mathop{{#1}\thinspace\negthinspace}
    \mathchoice{\negthinspace}{\negthinspace}{}{}}}

\newcommand{\f}[3][]{\ensuremath{\foperator{#2}\parentheses[#1]{(}{)}{#3}}}
\WithSuffix\newcommand{\f}*[2]{\ensuremath{
    \mathop{{#1}\thinspace\negthinspace}{#2}}}

\newcommand{\dd}{\ensuremath{\operatorname{d}\negthinspace{}}}

\newcommand{\alignpheqop}[1]{\mathrel{\phantom{=}} \mathop{#1}}  

\newif\ifproc\procfalse

\begin{document}

\title[Canonical Trees, Compact Prefix-free Codes and Sums of Unit Fractions]{Canonical Trees,\\ Compact Prefix-free Codes and \\Sums of Unit Fractions:\\ A Probabilistic Analysis}


\author{Clemens Heuberger}
\address{Institut f\"ur Mathematik\\Alpen-Adria-Universit\"at Klagenfurt\\Austria}
\email{\href{mailto:clemens.heuberger@aau.at}{clemens.heuberger@aau.at}}

\author{Daniel Krenn}
\address{Institute of Analysis and Computational Number Theory (Math A)\\TU
  Graz\\Austria}
\email{\href{mailto:math@danielkrenn.at}{math@danielkrenn.at} \textit{or}
  \href{mailto:krenn@math.tugraz.at}{krenn@math.tugraz.at}}

\author{Stephan Wagner}
\address{Department of Mathematical Sciences\\Stellenbosch University\\South Africa}
\email{\href{mailto:swagner@sun.ac.za}{swagner@sun.ac.za}}

\thanks{C.\ Heuberger and D.\ Krenn are supported by the Austrian Science Fund
  (FWF): W1230, Doctoral Program ``Discrete Mathematics'' and the Austrian
  Science Fund (FWF): P24644-N26.}

\thanks{S.\ Wagner is supported by the National Research Foundation of South
  Africa, grant number 70560.}

\thanks{An extended abstract with less general results without proofs
appeared as~\cite{Heuberger-Krenn-Wagner:2013:analy-param}.}

\begin{abstract}
  For fixed $t\ge 2$, we consider the class of representations of $1$ as sum
  of unit fractions whose denominators are powers of $t$ or equivalently the
  class of canonical compact $t$-ary Huffman codes or equivalently rooted
  $t$-ary plane ``canonical'' trees.

  We study the probabilistic behaviour of the height (limit distribution is
  shown to be normal), the
  number of distinct summands (normal distribution), the path length (normal
  distribution), the width (main term of the expectation and concentration property) and the number of
  leaves at maximum distance from the root (discrete distribution).
\end{abstract}

\maketitle







\newcommand{\tableheight}{{\small{
\begin{table}[htbp]
  \centering
  \begin{equation*}
  \begin{array}{r|l|l}
\multicolumn{1}{c|}{t} &
\multicolumn{1}{c|}{\mu_h} &
\multicolumn{1}{c}{\sigma_h^2} \\
\hline
2 & 0.5517980333242771 & 0.3191028720021838 \\
3 & 0.5330219170893142 & 0.2640876574238174 \\
4 & 0.5216130806307567 & 0.2465933142213578 \\
5 & 0.5137644952434437 & 0.2404182939877220 \\
6 & 0.5084950082062925 & 0.2396633993742431 \\
7 & 0.5051047365215813 & 0.2411570855092153 \\
8 & 0.5030001253275540 & 0.2432575483836212 \\
9 & 0.5017308605343554 & 0.2452173961787762 \\
10 & 0.5009832278618640 & 0.2467757623911673 \\
\end{array}
\end{equation*}

\caption{Numerical values of the constants in mean and variance of the 
  height for small values of $t$, cf.\ Theorem~\ref{thm:height}.
  See also Remark~\ref{rem:numerical-calc}.
  For the accuracy of these numerical results see the note at the end of
  the introduction.}
\label{tab:special-values-height}
\end{table}
}}}

\newcommand{\tabledepths}{{\small{
\begin{table}[htbp]
  \centering
  \begin{equation*}
  \begin{array}{r|l|l}
\multicolumn{1}{c|}{t} &
\multicolumn{1}{c|}{\mu_d} &
\multicolumn{1}{c}{\sigma_d^2} \\
\hline
2 & 0.4151957394337730 & 0.2449371766120133 \\
3 & 0.4869093777539261 & 0.2893609775712220 \\
4 & 0.5024588321518999 & 0.2741197923680785 \\
5 & 0.5050331956677906 & 0.2607084483093273 \\
6 & 0.5043408269340902 & 0.2530808413006747 \\
7 & 0.5030838633817897 & 0.2495578056054622 \\
8 & 0.5020050053196332 & 0.2483362931739359 \\
9 & 0.5012375070905982 & 0.2482103208441571 \\
10 & 0.5007377066674932 & 0.2485046286268308 \\
\end{array}
\end{equation*}

\caption{Values of the constants in mean and variance of the number of distinct
  depths of leaves for small values of $t$, cf.\ 
  Theorem~\ref{thm:distinct-depths}.
  See also Remark~\ref{rem:numerical-calc}.
  For the accuracy of these numerical results see the note at the end of
  the introduction.}
\label{tab:special-values-distinct-depths}
\end{table}
}}}

\newcommand{\tablepathlength}{{\small{
\begin{table}[htbp]
  \centering
  \begin{equation*}
  \begin{array}{r|l|D{.}{.}{2.12}}
\multicolumn{1}{c|}{t} &
\multicolumn{1}{c|}{\mu_{\mathit{tpl}}} &
\multicolumn{1}{c}{\sigma_{\mathit{tpl}}^2} \\
\hline
2 & 0.5517980333242771 & 0.4254704960029117 \\
3 & 0.7995328756339714 & 0.7922629722714524 \\
4 & 1.0432261612615134 & 1.3151643425139087 \\
5 & 1.2844112381086093 & 2.0034857832310170 \\
6 & 1.5254850246188775 & 2.8759607924909180 \\
7 & 1.7678665778255347 & 3.9388990633171834 \\
8 & 2.0120005013102160 & 5.1894943655172528 \\
9 & 2.2577888724045994 & 6.6208696968269586 \\
10 & 2.5049161393093200 & 8.2258587463722461 \\
\end{array}
\end{equation*}
\caption{Values of the constants in mean and variance of the 
  total path length for small values of $t$, cf.\ 
  Theorem~\ref{thm:pathlength}.
  See also Remark~\ref{rem:numerical-calc}.
  For the accuracy of these numerical results see the note at the end of
  the introduction.}
\label{tab:special-values-pathlength}
\end{table}
}}}

\section{Introduction}

We consider three combinatorial classes, which all turn out to be equivalent: partitions of $1$ into powers of $t$, canonical compact $t$-ary Huffman codes, and ``canonical'' $t$-ary trees,  see the precise discussion below. In this paper, we are interested in the structure of these objects under a uniform random model, and we study the distribution of various structural parameters, for which we obtain rather precise limit theorems. Let us first define all three classes precisely and explain the connections between them. Throughout the paper, $t \geq 2$ will be a fixed positive integer. Figure~\ref{fig:equivalence} shows examples in the case $t  = 2$.
\begin{enumerate}
\item Partitions of $1$ into powers of $t$ (representations of $1$ as sum of
  unit fractions whose denominators are powers of $t$) are formally defined as follows:
  \ifproc
  \begin{multline*}
    \calC_{\mathit{Partition}}=\Bigl\{(x_1,\ldots,x_\tau)\in \Z^\tau \Bigm\vert
    \tau\ge 0, \\ 0\le x_1\le x_2\le\cdots\le x_\tau,\ 
    \sum_{i=1}^\tau \frac1{t^{x_i}}=1 \Bigr\}.
  \end{multline*}
  \else
  \begin{equation*}
    \calC_{\mathit{Partition}}=\Bigl\{(x_1,\ldots,x_\tau)\in \Z^\tau \Bigm\vert \text{$\tau\ge 0$, $0\le x_1\le x_2\le\cdots\le x_\tau$,
    $\sum_{i=1}^\tau \frac1{t^{x_i}}=1$} \Bigr\}.
  \end{equation*}
  \fi
  The \emph{external size} $\abs{(x_1,\ldots,x_\tau)}$ of such a representation $(x_1,\ldots,x_\tau)$ is
  defined to be the number~$\tau$ of summands. 

\item Secondly, we consider canonical compact $t$-ary Huffman codes:
  \ifproc
  \begin{multline*}
    \calC_{\mathit{Code}}=\{ C\subseteq \{1,\ldots,t\}^* \mid \text{$C$ is }
    \text{prefix-free,}\\
    \text{compact and canonical}\}.
  \end{multline*}
  \else
  \begin{equation*}
    \calC_{\mathit{Code}}=\{ C\subseteq \{1,\ldots,t\}^* \mid C \text{ is
      prefix-free, compact and canonical}\}.
  \end{equation*}
  \fi
  Here, we use the following notions.
  \begin{itemize}
  \item $\{1,\ldots,t\}^*$ denotes the set of finite words over the alphabet
    $\{1,\ldots,t\}$.
  \item A code $C$ is said to be \emph{prefix-free} if no word in $C$ is a
    proper prefix of any other word in $C$.
  \item A code $C$ is said to be \emph{compact} if the following property
    holds: if $w$ is a proper prefix of a word in $C$, then for every letter
    $a\in\{1,\ldots,t\}$, $wa$ is a prefix of a word in $C$.
  \item A code $C$ is said to be \emph{canonical} if the lexicographic ordering
    of its words corresponds to a non-decreasing ordering of the word
    lengths. This condition corresponds to taking equivalence classes with
    respect to permutations of the alphabet (at each position in the words).
  \end{itemize}
  The \emph{external size} $\abs{C}$ of a code $C$ is defined to be the cardinality
  of $C$.

  If $C\in\calC_{\mathit{Code}}$ with $C=\{w_1,\ldots,w_\tau\}$ and the property
  that $\length(w_i)\le\length(w_{i+1})$ holds for all $i$, then
  $(\length(w_1),\ldots,\length(w_\tau))\in\calC_{\mathit{Partition}}$. This
  is a bijection between $\calC_{\mathit{Code}}$ and
  $\calC_{\mathit{Partition}}$ preserving the external size. This connection
  can be explained by the \emph{Kraft--McMillan
    inequality~\cite{Kraft:1949:thesis, McMillan:1956:inequalities}},
  which states that for any prefix-free code $C=\{w_1,\ldots,w_\tau\}$, one must have
  \begin{equation*}
    \sum_{i=1}^{\tau} t^{-\length(w_i)} \leq 1,
  \end{equation*}
and compact codes are precisely those for which equality holds (meaning that they are optimal in an information-theoretic sense).

\item Finally, both partitions and codes are related to so-called canonical rooted $t$-ary trees:
  \ifproc
  \begin{multline*}
    \calC_{\mathit{Tree}}=\{ T \text{ rooted $t$-ary plane tree} \mid \\T\text{ is
      canonical} \}.
  \end{multline*}
  \else
  \begin{equation*}
    \calC_{\mathit{Tree}}=\{ T \text{ rooted $t$-ary plane tree} \mid T\text{ is
      canonical} \}.
  \end{equation*}
  \fi
  Here, we use the following notions.
  \begin{itemize}
  \item \emph{$t$-ary} means that each vertex has no or $t$ children.
  \item \emph{Plane tree} means that an ordering ``from left to right'' of the
    children of each vertex is specified.
  \item \emph{Canonical} means that the following holds for all $k$: if the
    vertices of depth (i.e., distance to the root) $k$ are denoted by $v_1$,
    \ldots, $v_K$ from left to right, then $\deg(v_i)\le\deg(v_{i+1})$ holds
    for all $i$.
  \end{itemize}
  The \emph{external size} $\abs{T}$ of a tree is given by the number of its
  leaves, i.e., the number of vertices of degree $1$.

  If $C\in\calC_{\mathit{Code}}$, then a tree $T\in\calC_{\mathit{Tree}}$ can
  be constructed such that the vertices of $T$ are given by the prefixes of the
  words in $C$, the root is the vertex corresponding to the empty word, and the
  children of a proper prefix $w$ of a code word are given from left to right
  by $wa$ for $a=1$, $\ldots$, $t$. This is a bijection between
  $\calC_{\mathit{Code}}$ to $\calC_{\mathit{Tree}}$ preserving the external
  size.
\end{enumerate}
Further formulations, details and remarks can be found in the recent paper of Elsholtz, Heuberger
and Prodinger~\cite{Elsholtz-Heuberger-Prodinger:2013:huffm}. We will simply
speak of an element in the class $\calC$ when the particular interpretation as
an element of $\calC_{\mathit{Partition}}$, $\calC_{\mathit{Code}}$ or
$\calC_{\mathit{Tree}}$ is not relevant. Our proofs will use the tree model,
therefore $\calC_{\mathit{Tree}}$ is abbreviated as $\calT$.

\newcommand{\figureA}{
\begin{tikzpicture}[vertex/.style={circle,draw=black,fill=black,
          inner sep=0pt,minimum size=1mm},
        edge/.style={draw=black}, xscale=0.18, yscale=\ifproc0.7\else0.7\fi]
    \draw (0,0) node  (R) [vertex] {};
    \draw (-8,-1) node (R0) [vertex, label=below:$0$] {};
    \draw (8,-1) node (R1) [vertex] {};
    \draw[edge] (R)--(R0);
    \draw[edge] (R)--(R1);
    \draw (4,-2) node (R10) [vertex, label=below:$10$] {};
    \draw (12,-2) node (R11) [vertex] {};
    \draw[edge] (R1)--(R10);
    \draw[edge] (R1)--(R11);
    \draw (10,-3) node (R110) [vertex, label=below:$110$] {};
    \draw (14,-3) node (R111) [vertex] {};
    \draw[edge] (R11)--(R110);
    \draw[edge] (R11)--(R111);
    \draw (13,-4) node (R1110) [vertex, label=below left:$1110$] {};
    \draw (15,-4) node (R1111) [vertex, label=below right:$1111$] {};
    \draw[edge] (R111)--(R1110);
    \draw[edge] (R111)--(R1111);
  \end{tikzpicture}
}
\newcommand{\figureB}{
\begin{tikzpicture}[vertex/.style={circle,draw=black,fill=black,
          inner sep=0pt,minimum size=1mm},
        edge/.style={draw=black}, xscale=0.18, yscale=\ifproc0.7\else0.7\fi]
    \draw (0,0) node  (R) [vertex] {};
    \draw (-8,-1) node (R0) [vertex, label=below:$0$] {};
    \draw (8,-1) node (R1) [vertex] {};
    \draw[edge] (R)--(R0);
    \draw[edge] (R)--(R1);
    \draw (4,-2) node (R10) [vertex] {};
    \draw (12,-2) node (R11) [vertex] {};
    \draw[edge] (R1)--(R10);
    \draw[edge] (R1)--(R11);
    \draw (2,-3) node (R100) [vertex, label=below:$100$] {};
    \draw (6,-3) node (R101) [vertex, label=below:$101$] {};
    \draw (10,-3) node (R110) [vertex, label=below:$110$] {};
    \draw (14,-3) node (R111) [vertex, label=below:$111$] {};
    \draw[edge] (R10)--(R100);
    \draw[edge] (R10)--(R101);
    \draw[edge] (R11)--(R110);
    \draw[edge] (R11)--(R111);
  \end{tikzpicture}
}
\newcommand{\figureC}{
\begin{tikzpicture}[vertex/.style={circle,draw=black,fill=black,
          inner sep=0pt,minimum size=1mm},
        edge/.style={draw=black}, xscale=0.09, yscale=\ifproc0.7\else0.7\fi]
    \draw (0,0) node  (R) [vertex] {};
    \draw (-8,-1) node (R0) [vertex] {};
    \draw (8,-1) node (R1) [vertex] {};
    \draw[edge] (R)--(R0);
    \draw[edge] (R)--(R1);
    \draw (-12,-2) node (R00) [vertex, label=below:$00$] {};
    \draw (-4,-2) node (R01) [vertex, label=below:$01$] {};
    \draw (4,-2) node (R10) [vertex, label=below:$10$] {};
    \draw (12,-2) node (R11) [vertex] {};
    \draw[edge] (R0)--(R00);
    \draw[edge] (R0)--(R01);
    \draw[edge] (R1)--(R10);
    \draw[edge] (R1)--(R11);
    \draw (10,-3) node (R110) [vertex, label=below left:$110$] {};
    \draw (14,-3) node (R111) [vertex, label=below right:$111$] {};
    \draw[edge] (R11)--(R110);
    \draw[edge] (R11)--(R111);
  \end{tikzpicture}
}

\begin{figure}
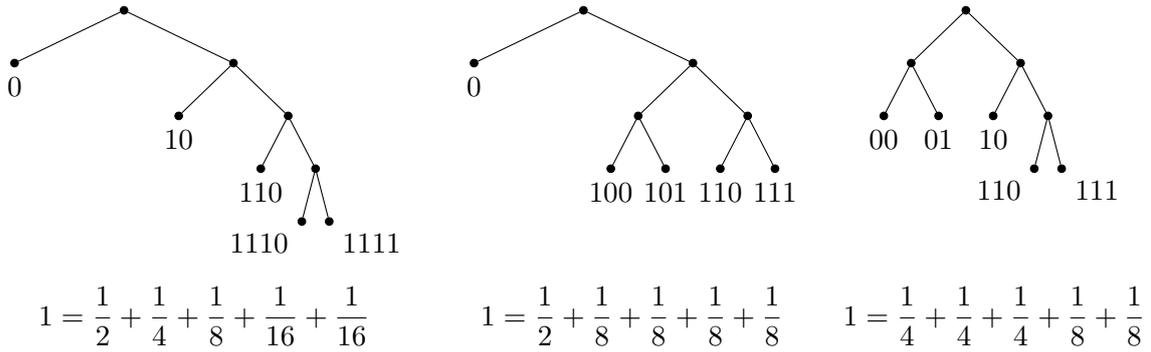

  \ifproc
  \centering
  \figureA
  \begin{equation*}
    1=\frac12+\frac14+\frac18+\frac1{16}+\frac1{16}
  \end{equation*}
  \\[2em]
  \figureB
  \begin{equation*}
    \displaystyle 1=\frac12+\frac18+\frac18+\frac18+\frac18
  \end{equation*}
  \\[2em]
  \figureC
  \begin{equation*}
    \displaystyle 1=\frac14+\frac14+\frac14+\frac1{8}+\frac1{8}
  \end{equation*}
  \else
  \begin{tabular}[t]{@{}ccc@{}}
        \raisebox{-\height}{\figureA}&
        \raisebox{-\height}{\figureB}&
          \raisebox{-\height}{\figureC}\\
\rule{0ex}{5ex}$\displaystyle 1=\frac12+\frac14+\frac18+\frac1{16}+\frac1{16}$&
$\displaystyle 1=\frac12+\frac18+\frac18+\frac18+\frac18$&
  $\displaystyle 1=\frac14+\frac14+\frac14+\frac1{8}+\frac1{8}$
  \end{tabular}
  \fi
  \caption[All elements of external size $5$]{All elements of external size $5$
    (and internal size $4$, respectively) in $\calC_{\mathit{Tree}}$,
    $\calC_{\mathit{Code}}$ and $\calC_{\mathit{Partition}}$ for $t=2$.}
  \label{fig:equivalence}
\end{figure}

The external size of an element in $\calC$ is always congruent to $1$ modulo $t-1$. This can
easily be seen in the tree model, where the number of leaves $\tau$ and the number of internal vertices
$n$ are connected by the identity
\begin{equation*}
  \tau=1+n(t-1).
\end{equation*}

Therefore, we will from now on consider the \emph{internal size:} for a tree
$T\in\calC_{\mathit{Tree}}$ the internal size of $T$ is the number $n(T)$ of
internal vertices, for a code $C\in\calC_{\mathit{Code}}$ the internal size is
the number of proper prefixes of words of $C$, and for a partition
$(x_1,\ldots,x_\tau)\in\calC_{\mathit{Partition}}$ the internal size is defined to
be $(\tau-1)/(t-1)$. We will omit the word ``internal'' and will always use the
variable $n$ (or $n(T)$ for a specific element $T \in \calC$) to denote the size.

The asymptotics of the number of elements in $\calC$ of size $n$ has been
studied by various authors, see the historical overview in \cite{Elsholtz-Heuberger-Prodinger:2013:huffm}. Special cases and weaker versions (without explicit error terms) of the following result, which is given in \cite{Elsholtz-Heuberger-Prodinger:2013:huffm} (building upon the generating function approach by Flajolet and Prodinger~\cite{Flajolet-Prodinger:1987:level}), were obtained earlier and independently by different authors (Boyd \cite{Boyd:1975}, Komlos, Moser and Nemetz \cite{Komlos-Moser-Nemetz:1984}, Flajolet and Prodinger \cite{Flajolet-Prodinger:1987:level} and Tangora \cite{tangora1991level}).

\begin{theorem}[\cite{Elsholtz-Heuberger-Prodinger:2013:huffm}]\label{theorem:asymptotics}
  For $t\ge 2$, the number of elements of size $n$ in $\calC$ is
  (in Bachmann--Landau notation) given by
  \begin{equation*}
    R\rho^{n+1}+\Theta(\rho_2^{n}),
  \end{equation*}
  where $\rho> \rho_2$ and $R$ are positive real constants depending on $t$
  with asymptotic expansions (as $t \to \infty$) \ifproc
  \begin{equation*}
    \rho=2-\frac1{2^{t+1}} +O\!\left(\frac{t}{2^{2t}}\right),
  \end{equation*}
  \begin{equation*}
    \rho_2=1+\frac{\log 2}{t} + O\!\left(\frac1{t^2}\right),
  \end{equation*}
  \begin{equation*}
    R=\frac{1}{8}+\frac{t - 2}{2^{t+5}}+ O\!\left(\frac{t^2}{2^{2t}}\right).
  \end{equation*}
  \else
  \begin{align*}
    \rho&=2-\frac1{2^{t+1}} +O\!\left(\frac{t}{2^{2t}}\right),&
    \rho_2&=1+\frac{\log 2}{t} + O\!\left(\frac1{t^2}\right),&
    R&=\frac{1}{8}+\frac{t - 2}{2^{t+5}}+ O\!\left(\frac{t^2}{2^{2t}}\right).
  \end{align*}
  \fi
\end{theorem}
In fact, all $O$-constants can be made explicit and more terms of the
asymptotic expansions in $t$ of $\rho$, $\rho_2$ and $R$ can be given.

In spite of the fact that the counting problem has been studied independently by many different authors, to the best of our knowledge the structure of random elements has not been considered before. Thus the purpose of this contribution is to study the probabilistic behaviour of various parameters of a random element in $\calC$ of size $n$.  We always use the uniform random model: whenever a random tree (equivalently, partition or code) of a given order~$n$ is chosen, all elements are considered to be equally likely.
\begin{enumerate}
\item The \emph{height} $h(T)$ of a tree $T\in\calC_{\mathit{Tree}}$
  is defined to be the maximum distance of a leaf from the root. In
  the interpretation as a code, this is the maximum length of a code
  word. In a representation of $1$ as a sum of unit fractions, this
  corresponds to the largest denominator used (more precisely, to the
  largest exponent of the denominator).

  The height is discussed in Section~\ref{sec:height}. It is
  asymptotically normally distributed with mean $\sim \mu_h n$ and variance
  $\sim \sigma^2_h n$, where
  \begin{equation*}
    \mu_h = \frac{1}{2} + \frac{t-2}{2^{t+3}}
    + O\!\left(\frac{t^2}{2^{2t}}\right)
    \ifproc\end{equation*}and\begin{equation*}\else\quad\text{and}\quad\fi
    \sigma^2_h=  \frac{1}{4} + \frac{-t^2+5t-2}{2^{t+4}} 
    + O\!\left(\frac{t^3}{2^{2t}}\right),
  \end{equation*}
  cf.\ Theorem~\ref{thm:height}. Moreover, we prove a local limit theorem.

\item The \emph{number of distinct summands} of a representation
  $(x_1,\ldots,x_\tau)$ of $1$ as sum of unit fractions is denoted by
  $d(x_1,\ldots,x_\tau)$. In the tree model, this corresponds to the cardinality
  $d(T)$ of the set of depths of leaves in a tree
  $T\in\calC_{\mathit{Tree}}$. In the code model, this is the number of
  distinct lengths of code words.

  The number $d(T)$ is studied in Section~\ref{sec:distinct-depths}. It is
  asymptotically normally distributed with mean $\sim \mu_d n$ and variance
  $\sim \sigma^2_d n$, where
  \begin{equation*}
    \mu_d = \frac{1}{2} + \frac{t-4}{2^{t+3}} 
    + O\!\left(\frac{t^2}{2^{2t}}\right)
    \ifproc\end{equation*}and\begin{equation*}\else\quad\text{and}\quad\fi
    \sigma^2_d = \frac{1}{4} + \frac{-t^2+9t-14}{2^{t+4}} 
    + O\!\left(\frac{t^2}{2^{2t}}\right),
  \end{equation*}
  cf.\ Theorem~\ref{thm:distinct-depths}. Moreover, a local limit theorem is proved again.

\item The \emph{maximum number of equal summands} of a representation
  $(x_1,\ldots,x_\tau)$ of $1$ as sum of unit fractions is denoted by
  $w(x_1,\ldots,x_\tau)$. In the code model, this is the maximum number of code
  words of equal length. In the tree model, this is the ``leaf-width'' $w(T)$,
  i.e., the maximum number of leaves on the same level.

  The number $w(T)$ is studied in Section~\ref{section:width}. We prove that
  $\E(w(T))=\mu_w\log n+O(\log\log n)$ with $\mu_w=1/(t\log 2)+O(1/t^2)$ and a
  concentration property, cf.\ Theorem~\ref{theorem:width}.

\item The \emph{(total) path length} $\ell(T)$ of a tree
  $T\in\calC_{\mathit{Tree}}$ is defined to be the sum of the depths of all
  vertices of the tree. In our context, it is perhaps most natural to consider
  the \emph{external path length} $\ell_{\mathit{external}}(T)$, though, which
  is the sum of depths over all leaves of the tree, as this parameter
  corresponds to the sum of lengths of code words in a code
  $C\in\calC_{\mathit{Code}}$. Likewise, the \emph{internal path length}
  $\ell_{\mathit{internal}}(T)$ is the sum of depths over all
  non-leaves. Clearly, we have
  $\ell_{\mathit{external}}(T)+\ell_{\mathit{internal}}(T)=\ell(T)$, and the
  relations
  \begin{equation*}
    \ell_{\mathit{external}}(T) = \frac{t-1}{t} \ell(T) + n(T)
    \ifproc\end{equation*}and\begin{equation*}\else\quad\text{and}\quad\fi
    \ell_{\mathit{internal}}(T) = \frac{1}{t} \ell(T) - n(T)
  \end{equation*}
  for $t$-ary trees are easily proven. Therefore, all distributional results
  for any one of those parameters immediately cover all three. The total path
  length turns out to be asymptotically normally distributed as well (see
  Theorem~\ref{thm:pathlength}), with mean $\sim \mu_{\mathit{tpl}} n^2$ and
  variance $\sim \sigma_{\mathit{tpl}}^2 n^3$. The coefficients have asymptotic
  expansions
  \begin{equation*}
    \mu_{\mathit{tpl}} = \frac{t}{2} \cdot \mu_h 
    =\frac{t}{4} + \frac{t(t-2)}{2^{t+4}} +
    O\!\left(\frac{t^3}{2^{2t}}\right)
    \ifproc\end{equation*}and\begin{equation*}\else\quad\text{and}\quad\fi
    \sigma_{\mathit{tpl}} = \frac{t^2}{12} + \frac{-t^4+5t^3-2t^2}{3\cdot 2^{t+4}} 
    + O\!\left(\frac{t^5}{2^{2t}}\right).
  \end{equation*}
  The path length is studied in Section~\ref{sec:path-length}. Its analysis is
  based on a generating function approach for the moments, combined with
  probabilistic arguments to obtain the central limit theorem.

\item The \emph{number of leaves on the last level} (i.e., maximum distance from the root) of a tree
  $T\in\calC_{\mathit{Tree}}$ is denoted by $m(T)$. This corresponds to the
  number of code words of maximum length and to the number of smallest summands
  in a representation of $1$ as a sum of unit fractions. 

  This parameter may appear to be the least interesting of the parameters we
  study. However, it is a natural technical parameter when constructing
  generating functions for the other parameters. From these generating
  functions the probabilistic behaviour of $m(T)$ can be read off without too
  much effort, so we do include these results in Section~\ref{sec:number-leaves-maximum-depth}.

  The limit distribution of $m(T)$ is a discrete distribution with mean
  $2t+o(1)$ and variance $2t^2+o(1)$, cf.\ Theorem~\ref{theorem:distribution-m}.
\end{enumerate}

A noteworthy feature of the results listed above is the fact that the
distributions we observe are quite different from those that one obtains for
other probabilistic random tree models. Specifically, the parameters differ
from the ones of Galton--Watson trees (which include, amongst others, uniformly random
$t$-ary trees), but also from the ones of recursive trees and general families
of increasing trees. See~\cite{drmota2009random} for a general reference. In
particular,
\begin{itemize}
\item the asymptotic order of the height of a random Galton--Watson tree of
  order $n$ is only $\sqrt{n}$, and it is known that the limiting distribution
  (which is sometimes called a Theta distribution) coincides with the
  distribution of the maximum of a Brownian
  excursion~\cite{flajolet1993distribution}. The height of random recursive
  trees (or other families of increasing trees) is even only of order $\log n$,
  and heavily concentrated around its mean, see~\cite{drmota2009height}.
\item The path length of random Galton--Watson trees is of order $n^{3/2}$, and
  it follows an Airy distribution (like the area under a Brownian excursion) in
  the limit~\cite{takacs1992total}. For recursive trees, the path length is of
  order $n \log n$ with a rather unusual limiting
  distribution~\cite{mahmoud1991limiting}.
\item While the height of our canonical trees is greater than that of
  Galton--Watson trees, precisely the opposite holds for the width (as one
  would expect): it is of order $\sqrt{n}$ for Galton--Watson trees
  \cite{drmota1997profile,takacs1993limit}, with the same limiting distribution
  as the height, as opposed to only $\log n$ in our setting. For recursive
  trees, the width is even of order $n/\sqrt{\log n}$,
  see~\cite{drmota2005profiles}.
\end{itemize}

Indeed, the structure of our canonical $t$-ary trees is comparable to that of
\emph{compositions}: Counting the number of internal vertices on each level
from the root, we obtain a restricted composition, in which each
summand is at most $t$ times the previous one. In the limit $t \to \infty$ one
obtains compositions of $n$ starting with a $1$ in this way. The recent series of papers
by Bender and Canfield~\cite{bender2005locally,bender2009locally,bender2010locally} and Bender, Canfield and Gao \cite{bender2012locally}
is concerned with compositions with various local restrictions. In fact it would be possible to derive the central limit theorems for the height and the number of distinct summands from Theorem~4 in \cite{bender2009locally}, but in a less explicit fashion (without precise constants, and further work would still be required for a local limit theorem). A parameter related to the ``leaf width'' (the largest part of a composition) is also studied in \cite{bender2012locally}, but in addition to the fact that the parameters are not quite identical, it also seems that the technical conditions required for the main result of \cite{bender2012locally} are not satisfied here.

Lastly, a remark on numerics and notation. Throughout the paper, various constants occur in all our major results,
and we provide numerical values for small $t$ as well as asymptotic formul\ae{} for these constants in terms of $t$.  The error terms that occur in these formul\ae{} have an
explicit $O$-constant, which is indicated by error functions $\f{\varepsilon_j}{\ldots}$. These functions have the property that
$\abs{\f{\varepsilon_j}{\ldots}} \leq 1$ for all values of the indicated parameters. All results were calculated
with the free open-source mathematics software system
SageMath~\cite{Stein-others:2015:sage-mathem-6.5} and are available
online\footnote{The worksheets containing the calculations can be found at \url{http://www.danielkrenn.at/unit-frac-parameters-full}.}
The numerical expressions were
obtained by using interval arithmetic, therefore they are reliable results. Each
numerical value of this paper is given in such a way that its error is at most the magnitude of the last
indicated digit. It would be possible to calculate the
values with higher accuracy. Determining accurate numerical values and asymptotic formul\ae{} is not just interesting in its own right,
it is also important for some of our theorems: specifically, for all Gaussian limit laws it is crucial to ensure that the growth constants associated
with the variance are nonzero. We will therefore comment repeatedly on how reliable numerical values can be obtained.

\section{The Generating Function}
\label{sec:generating-function}

In this section, we derive the generating function which will be used
throughout the article.

The analysis of the path length (Section~\ref{sec:path-length}) also requires
results on canonical forests. For $r\ge 1$, we consider the set $\calF_r$ of
canonical forests with $r$ roots. These $r$ roots are all on the same level and
ordered from left to right. The notion ``canonical'' introduced for trees here is
meant to hold over all connected components of the forest. This means that
a forest may not be seen as a collection of trees, but rather as the subgraph
of a canonical tree induced by its vertices of depths $\ge d$ for some $d$. In
fact, this is also the interpretation for which we will need results on
forests.  We will phrase the generating function in terms of forests, but most
other results will be formulated for trees only.

The height $h(T)$, the cardinality $d(T)$ of the set of different depths of
leaves, and the number $m(T)$ of leaves on the last level of a
forest\footnote{We use the symbol $T$ (instead of $F$) for a canonical forest
in $\calF_r$ since we usually look at the special case $r=1$, where $T$
  is a tree.}
$T\in\calF_r$ of size $n=n(T)$ can be analysed by studying a multivariate
generating function $H(q,u,v,w)$, where $q$ labels the size~$n(T)$, $u$ labels
the number~$m(T)$ of leaves on the last level, $v$ labels the
cardinality~$d(T)$ of the set of depths of leaves and $w$ labels the
height~$h(T)$.

\begin{theorem}\label{theorem:generating-function-height-and-others}
  The generating function
  \begin{equation*}
    H(q,u,v,w) \colonequals \sum_{T\in\calF_r} q^{n(T)}u^{m(T)}v^{d(T)}w^{h(T)}
  \end{equation*}
  can be expressed as 
  \ifproc
  \begin{equation}\label{eq:generating-function-height-and-others}
    \begin{split}
      H(q,u,v,w)&=a(q,u,v,w) \\
      &\phantom{=}+b(q,u,v,w)\frac{a(q,1,v,w)}{1-b(q,1,v,w)}
    \end{split}
  \end{equation}
  \else
  \begin{equation}\label{eq:generating-function-height-and-others}
    H(q,u,v,w)=a(q,u,v,w)+b(q,u,v,w)\frac{a(q,1,v,w)}{1-b(q,1,v,w)}
  \end{equation}
  \fi
  with
  \begin{align}
    a(q,u,v,w)&=\sum_{j=0}^{\infty} vq^{r\hp{j}}u^{rt^j}w^j
    \prod_{i=1}^{j}\frac{1-v-q^{\hp{i}}u^{t^{i}}}{1-q^{\hp{i}}u^{t^{i}}},\notag\\
    b(q,u,v,w)&=\sum_{j=1}^{\infty} 
    \frac{vq^{\hp{j}}u^{t^{j}}w^{j}}{1-q^{\hp{j}}u^{t^{j}}}
    \prod_{i=1}^{j-1}\frac{1-v-q^{\hp{i}}u^{t^{i}}}{1-q^{\hp{i}}u^{t^{i}}},
    \label{eq:b-q-u-v-w-general-formula}
  \end{align}
  where 
   $\hp{j} \colonequals 1+t+\cdots+t^{j-1}$.

   The functions $a(q,u,v,w)$ and $b(q,u,v,w)$ are analytic in $(q,u,v,w)$ when
   \begin{equation*}
     \abs{q} < \frac{1}{\abs{u}^{t-1}}.
   \end{equation*}
\end{theorem}

When $u=1$, the generating function can be simplified to
\begin{equation}\label{eq:generating-function-simplified}
  H(q,1,v,w) = \frac{a(q,1,v,w)}{1-b(q,1,v,w)}.
\end{equation}

The proof of Theorem~\ref{theorem:generating-function-height-and-others}
depends on solving a functional equation for the generating function. As we
will encounter similar functional equations for related generating functions in
Section~\ref{sec:path-length}, we formulate the relevant result in the
following lemma.

\begin{lemma}\label{lemma:functional-equation}
  Let $\calD\subseteq \C$ be the closed unit disc and $q\in\C$ with $\abs{q}<1$.
  Let $P$, $R$, $S$, $f$ be bounded
  functions on $\calD$ and $s$ be a constant such that $\abs{S(u)}\le s<1$ for all
  $u\in\calD$.
  
  If 
  \begin{equation}\label{eq:general-functional-equation}
    f(u)=P(u)+R(qu^t)f(1)+S(qu^t)f(qu^t)
  \end{equation}
  holds for all $u\in\calD$, then
  \begin{equation}\label{eq:general-functional-equation-solution}
    f(u)=a(u)+b(u)\frac{a(1)}{1-b(1)}
  \end{equation}
  holds with
  \begin{equation}\label{eq:a_u_b_u_definition}
    \begin{aligned}
      a(u)&=\sum_{j=0}^\infty
      P(q^{\hp{j}}u^{t^j})\prod_{i=1}^jS(q^{\hp{i}}u^{t^i})\\
      b(u)&=\sum_{j=1}^\infty
      R(q^{\hp{j}}u^{t^j})\prod_{i=1}^{j-1}S(q^{\hp{i}}u^{t^i})
    \end{aligned}
  \end{equation}
  provided that $b(1)\neq 1$.
\end{lemma}

\begin{proof}
  We iterate the functional equation \eqref{eq:general-functional-equation} and
  obtain
  \begin{equation*} 
    f(u)=a_k(u)+b_k(u)f(1)+c_k(u)f(q^{\hp{k}}u^{t^{k}})
  \end{equation*}
  for $k\ge 0$ with
  \begin{align*}
    a_k(u)&=\sum_{j=0}^{k-1}P(q^{\hp{j}}u^{t^j})\prod_{i=1}^jS(q^{\hp{i}}u^{t^i}),\\
    b_k(u)&=\sum_{j=1}^kR(q^{\hp{j}}u^{t^j})\prod_{i=1}^{j-1}S(q^{\hp{i}}u^{t^i}),\\
    c_k(u)&=\prod_{i=1}^{k}S(q^{\hp{i}}u^{t^i}).
  \end{align*}
  The assumption $\abs{q}<1$ implies that $\lim_{k\to\infty}q^{\hp{k}}u^{t^k}=0$
  for $\abs{u}\leq1$. Therefore,
  \begin{equation*}
    \lim_{k\to\infty} a_k(u)=a(u),\qquad
    \lim_{k\to\infty} b_k(u)=b(u),\qquad
    \lim_{k\to\infty} c_k(u)=0,
  \end{equation*}
  for $u\in\calD$ and the functions $a(u)$ and $b(u)$ given in
  \eqref{eq:a_u_b_u_definition}.

  Taking the limit in \eqref{eq:general-functional-equation}, we get
  \begin{equation}\label{eq:functional-equation}
    f(u)=a(u)+b(u)f(1)
  \end{equation}
  for $u\in\calD$. Setting $u=1$ in \eqref{eq:functional-equation} yields~\eqref{eq:general-functional-equation-solution}. 
\end{proof}

\begin{proof}[Proof of Theorem~\ref{theorem:generating-function-height-and-others}]
  The proof of Theorem~\ref{theorem:generating-function-height-and-others}
  follows ideas of Flajolet and Prodinger~\cite{Flajolet-Prodinger:1987:level},
  see also \cite{Elsholtz-Heuberger-Prodinger:2013:huffm}. We first consider
  \ifproc
  \begin{align*}
    H_h(q,u,v) \colonequals& [w^h]H(q,u,v,w) \\
    =& \sum_{\substack{T\in\calF_r\\h(T)=h}} q^{n(T)}u^{m(T)}v^{d(T)}
  \end{align*}
  \else
  \begin{equation*}
    H_h(q,u,v) \colonequals [w^h]H(q,u,v,w)
    = \sum_{\substack{T\in\calF_r\\h(T)=h}} q^{n(T)}u^{m(T)}v^{d(T)}
  \end{equation*}
  \fi
  for some $h\ge 0$.

  A forest $T'$ of height $h+1$ arises from a forest $T$ of height $h$ by replacing
  $j$ of its $m(T)$ leaves on the last level (for some $j$ with $1\le j\le m(T)$) by internal vertices, each with $t$ leaves as its children.  If $j = m(T)$, then
  all old leaves become internal vertices, so that $d(T')=d(T)$. Otherwise,
  i.e., if $j < m(T)$, at least one of them becomes a new leaf, meaning that we
  have a new level that contains one or more leaves, hence $d(T')=d(T)+1$.

  For the generating function~$H_h$, this translates to the recursion
  \begin{equation}\label{eq:H-recursion}
    \begin{aligned}
      H_{h+1}(q,u,v)&=\sum_{\substack{T\in\calF_r\\h(T)=h}}
      \biggl(\sum_{j=1}^{m(T)-1} q^{n(T)+j} u^{jt}v^{d(T)+1}+
      q^{n(T)+m(T)}u^{m(T)t}v^{d(T)}\biggr)\\
      &=\sum_{\substack{T\in\calF_r\\h(T)=h}}q^{n(T)}v^{d(T)}
      \left(qu^tv\frac{1-(qu^t)^{m(T)}}{1-qu^t}+(1-v)(qu^t)^{m(T)}\right)\\
      &= R(q,qu^t,v) H_h(q,1,v)+S(q,qu^t,v)
      H_h(q,qu^t,v),
    \end{aligned}
  \end{equation}
  where we set
  \begin{equation*}
    R(q,u,v)=\frac{uv}{1-u}, \qquad
    S(q,u,v)=\frac{1-v-u}{1-u}.
  \end{equation*}
  Note that the initial value is given by $H_0(q,u,v)=u^rv$.

  Now set
  \begin{equation*}
    \calD_0 \colonequals \{ (q,u,v,w)\in \C^4 \mid 
    \abs{q} < 1/5, \abs{u} \le 1, \abs{v-1} < 1/5, \abs{w} \le 1\}.
  \end{equation*}
  We note that if $(q,u,v,w)\in\calD_0$, we have
  \begin{equation*}
    \abs[normal]{R(q,qu^t,v)} \le   \frac{3}{10},\qquad
    \abs[normal]{S(q,qu^t,v)} \le \frac12.
  \end{equation*}
  This and \eqref{eq:H-recursion} imply that $\abs{H_h(q,u,v)}\le (6/5)(4/5)^h$
  holds for $h\ge 0$ and $(q,u,v,w)\in\calD_0$. Thus
  $H(q,u,v,w)=\sum_{h\ge 0}H_h(q,u,v)w^h$ converges uniformly for
  $(q,u,v,w)\in\calD_0$.

  Multiplying \eqref{eq:H-recursion} by $w^{h+1}$ and summing over all $h\ge 0$
  yields the functional equation
  \begin{equation*}
    H(q,u,v,w)=u^rv+wR(q,qu^t,v)H(q,1,v,w)+wS(q,qu^t,v)H(q,qu^t,v,w).
  \end{equation*}
  Lemma~\ref{lemma:functional-equation} immediately
  yields~\eqref{eq:generating-function-height-and-others}.

  Let now
  \begin{equation*}
    \calD_1=\{(q,u,v,w)\in \C^4 \mid \abs[normal]{qu^{t-1}}<1\}.
  \end{equation*}
  We clearly have $\calD_0\subseteq \calD_1$.
  For $(q,u,v,w)\in\calD_1$, we have
  \begin{equation*}
    \lim_{k\to\infty }q^{\hp{k}}u^{t^k}=\lim_{k\to\infty}q^{-1/(t-1)} \bigl(q^{1/(t-1)}u\bigr)^{t^k}=0. 
  \end{equation*}
  Therefore, $a(q,u,v,w)$ and $b(q,u,v,w)$ are analytic in $\calD_1$.
\end{proof}

In the following lemma, we also state a simplified expression and a functional
equation for $b(q,u,v,w)$ in the case $v=1$, $w=1$.

\begin{lemma}\label{lemma:b-q-u-1-1-formula}
  We have
  \begin{equation*}
    b(q,u,1,1)=\sum_{j=1}^\infty (-1)^{j-1}\prod_{i=1}^j  \frac{q^{\hp{i}}u^{t^i}}{1-q^{\hp{i}}u^{t^i}}=\frac{qu^t}{1-qu^t}(1-b(q,qu^t,1,1)).
  \end{equation*}
  In particular, the coefficient $[u^j]b(q,u,1,1)$ vanishes if $j$ is not a
  multiple of $t$.
\end{lemma}

\begin{proof}
  This is an immediate consequence of \eqref{eq:b-q-u-v-w-general-formula}.
\end{proof}

Next we recall results on the singularities of $H(q,1,1,1)$, see Proposition~10
of~\cite{Elsholtz-Heuberger-Prodinger:2013:huffm}.
We use functions $\varepsilon_j$ for modeling explicit $O$\nobreakdash-constants as
it was mentioned at the end of the introduction.

\begin{lemma}\label{lemma:dominant-singularity}
  The generating function $H(q,1,1,1)$ has exactly one singularity $q=q_0$ with
  $\abs{q} < 1-\frac{0.72}t$. This singularity $q_0$ is a simple pole and is positive. For
  $t\ge 4$, we have \ifproc
  \begin{align*}
    q_0=\frac{1}{2}&+ \frac{1}{2^{t+3}}+\frac{t + 4}{2^{2t+5}} \\
    &+ \frac{3t^{2} + 23  t + 38}{2^{3t+8}}
    + \frac{7 t^{3}}{100 \cdot 2^{4t}}\neweps(t).
  \end{align*}
  \else
  \begin{equation*}
    q_0=\frac{1}{2}+ \frac{1}{2^{t+3}}+\frac{t + 4}{2^{2t+5}}   
    + \frac{3t^{2} + 23  t + 38}{2^{3t+8}} 
    + \frac{7 t^{3}}{100 \cdot 2^{4t}}\neweps(t).
  \end{equation*}
  \fi
  For $t\in\{2,3\}$, the values are given in Table~\ref{tab:constant-Q}.
  Furthermore, let
  \begin{equation*}
    Q=\frac12+\frac{\log 2}{2t}+\frac{0.06}{t^2}
  \end{equation*}
  for $t\ge 6$ and $Q$ be given by Table~\ref{tab:constant-Q} for $2\le t\le
  5$. Then $q_0$ is the only singularity $q$ of $H(q,1,1,1)$ with $\abs{q}\le
  q_0/Q$.

  Setting $U=1-\frac{\log 2}{t^2}$ for $t>2$ and $U=1-\frac{19\log
      2}{80}$ for $t=2$, we have the estimate
  \begin{equation}\label{eq:U-estimates}
    U^{1-t}\max\Big(\frac{q_0}{Q}, \frac{5}{6}\Big)<1.
  \end{equation}
  These results do not depend on the choice of the number of roots~$r$.
\end{lemma}

\begin{table}
  \centering{\small
  \begin{equation*}
    \begin{array}{r|l|l}
      \multicolumn{1}{c|}{t}&
      \multicolumn{1}{c|}{q_0}&
      \multicolumn{1}{c}{Q}\\
      \hline
    2 & 0.5573678720139932 & 0.7131795784312742\\
    3 & 0.5206401166257250 & 0.6307447647757403\\
    4 & 0.5090030531391631 & 0.5930691701039086\\
    5 & 0.5042116835293617 & 0.5720078345052473\\
    6 & 0.5020339464245723 & 0.559428931713329\\
    7 & 0.5009982119507272 & 0.550735002693058\\
    8 & 0.5004941016343997 & 0.544259198784997\\
    9 & 0.500245704703080 & 0.539248917438516\\
    10 & 0.5001224896234884 & 0.535257359027998\\
    \end{array}
  \end{equation*}}
  \caption{Constants $q_0$ and $Q$ for $2\le t\le 10$.
  For the accuracy of these numerical results see the note at the end of
  the introduction.}
  \label{tab:constant-Q}
\end{table}
\begin{proof}
  By \cite[Proposition~10]{Elsholtz-Heuberger-Prodinger:2013:huffm},
  the function $1-b(q,1,1,1)$ has a unique simple zero $q=q_0$ with $\abs{q}\le 1-0.72/t$
  and no further zero for $\abs{q}\le q_0/Q$; the asymptotic estimates for
  $q_0$ and $Q$ follow from the results given in
  \cite{Elsholtz-Heuberger-Prodinger:2013:huffm}.

  At this point, we still have to show that the numerator does not vanish in
  $q_0$. We note that $q_0\le 3/5$. Using
  \cite[Lemma~8]{Elsholtz-Heuberger-Prodinger:2013:huffm}, we obtain
  \begin{equation*}
    \abs{a(q_0,1,1,1)-1}\le q_0^r \frac{q_0}{1-q_0}+\sum_{j=2}^\infty
    q_0^{r\hp j}\prod_{i=1}^j \frac{q_0^{\hp{i}}}{1-q_0^{\hp{i}}}\le
    \frac{9}{10}+\frac{83038203}{903449750}<1.
  \end{equation*}
  Therefore,
  \begin{equation}\label{eq:numerator-bound}
    a(q_0,1,1,1)=\Theta(1)
  \end{equation}
  holds uniformly in $r$.

  For $t\ge 30$, the estimate \eqref{eq:U-estimates} follows from the
  asymptotic expressions.
  For $t\le 30$, it is verified individually.
\end{proof}

Using this result, we will be able to apply singularity analysis to all our
generating functions in the coming sections.
At this point, we restate
Theorem~\ref{theorem:asymptotics} on the number of trees taking the notations
of Theorem~\ref{theorem:generating-function-height-and-others} into account and
extend it to the number of canonical forests with $r$ roots.

\begin{lemma}\label{lemma:number-canonical-forests}
  For $r\ge 1$, let
  \begin{equation}\label{eq:nu-definition}
    \nu(r)=\frac{a(q_0,1,1,1)}{q_0\frac{\partial}{\partial q}b(q,1,1,1) \Bigr\rvert_{q=q_0}}
  \end{equation}
  where $a(q_0,1,1,1)$ is taken in the version with $r$ roots.

  Then
  \begin{equation}\label{eq:nu_r_bounded}
    \nu(r)=\Theta(1)
  \end{equation}
  uniformly in $r\ge 1$ and the number of canonical forests with $r$ roots of size $n$ is
  \begin{equation}\label{eq:number-of-forests}
    \frac{\nu(r)}{q_0^{n}}\bigl(1+O(Q^n)\bigr),
  \end{equation}
  also uniformly in $r\ge 1$.
\end{lemma}
\begin{proof}
  By singularity analysis~\cite{Flajolet-Odlyzko:1990:singul,
    Flajolet-Sedgewick:ta:analy}, Lemma~\ref{lemma:dominant-singularity} and Theorem~\ref{theorem:generating-function-height-and-others}, the number of canonical forests with $r$
  roots of size $n$ is
  \begin{equation}\label{eq:residue-forest}
    -\Res\left(\frac{H(q,1,1,1)}{q^{n+1}},
      q=q_0\right)+O\biggl(\biggl(\frac{Q}{q_0}\biggr)^n\biggr)=
    \frac{\nu(r)}{q_0^{n}}+O\biggl(\biggl(\frac{Q}{q_0}\biggr)^n\biggr).
  \end{equation}
  The $O$-constant can be chosen independently of $r$ as $a(q,1,1,1)$ can be
  bounded independently of $r$ for $\abs{q}=q_0/Q$.

  The estimate \eqref{eq:numerator-bound} immediately yields~\eqref{eq:nu_r_bounded}.
  Combining this with \eqref{eq:residue-forest} yields \eqref{eq:number-of-forests}.
\end{proof}


When analyzing the asymptotic behavior of the height
(Section~\ref{sec:height}), the number of leaves on the last level
(Section~\ref{sec:number-leaves-maximum-depth}) and the path length
(Section~\ref{sec:path-length}), the corresponding formul\ae{} contain the
infinite sum~$\f{b}{q,u,1,w}$ and its derivatives. In order to perform the
calculations to get the asymptotic expressions in~$t$ as well as certifiable
numerical values for particular~$t$, we will work with a truncated sum and
bound the error we make. We define
\begin{equation*}
  \f{b_J}{q,u,1,w} = - \sum_{1\leq j < J} (-1)^j w^j
    \prod_{i=1}^j\frac{q^{\hp{i}}u^{t^{i}}}{1-q^{\hp{i}}u^{t^{i}}}.
\end{equation*}
Note that the variable~$v$ encoding the distinct depths of leaves is
handled separately in Lemmata~\ref{lem:d:den-bounds-pre} and \ref{lem:d:den-bounds}.

The following lemmata provide the estimates we need.


\begin{lemma}\label{lem:den-bounds-pre}
  Let $J\in\N$ and $q$, $u$, $w\in\C$ with $\abs{qu^{t-1}}<1$. Set
  \begin{equation*}
    Q=\abs{w} \frac{\abs{q}^{\hp{J+1}} \abs{u}^{t^{J+1}}}{
      1-\abs{q}^{\hp{J+1}} \abs{u}^{t^{J+1}}},
  \end{equation*}
  and suppose that $Q<1$ holds. Then
  \begin{equation*}
    \abs{\f{b}{q,u,1,w} - \f{b_J}{q,u,1,w}}
    \leq \abs{w}^J \left(
      \prod_{i=1}^J \frac{\abs{q}^{\hp{i}} \abs{u}^{t^i}}{
        \abs{1-q^{\hp{i}} u^{t^i}}} \right)
    \frac{1}{1-Q}.
  \end{equation*}
\end{lemma}

Note that as $\abs{qu^{t-1}}<1$, the error bound stated in the lemma
is decreasing in $J$.


\begin{proof}[Proof of Lemma~\ref{lem:den-bounds-pre}]
  Set
  \begin{equation*}
    R = \f{b}{q,u,1,w} - \f{b_J}{q,u,1,w} =
    - \sum_{j\geq J} (-1)^j w^j
    \prod_{i=1}^j\frac{q^{\hp{i}}u^{t^{i}}}{1-q^{\hp{i}}u^{t^{i}}}.
  \end{equation*}
  As  $\abs[normal]{q^{\hp{i}}u^{t^i}}$ is decreasing in $i$, we have
  \begin{align*}
    \abs{w^j \prod_{i=1}^j \frac{q^{\hp{i}}u^{t^{i}}}{1-q^{\hp{i}}u^{t^{i}}}}
    &\leq \abs{w}^j \left(\frac{\abs{q}^{\hp{J+1}}\abs{u}^{t^{J+1}}}{
        1-\abs{q}^{\hp{J+1}}\abs{u}^{t^{J+1}}}\right)^{j-J} 
    \prod_{i=1}^J \frac{\abs{q}^{\hp{i}}\abs{u}^{t^{i}}}{
      \abs{1-q^{\hp{i}}u^{t^{i}}}} \\
    &= \abs{w}^J Q^{j-J} \prod_{i=1}^J
    \frac{\abs{q}^{\hp{i}}\abs{u}^{t^{i}}}{\abs{1-q^{\hp{i}}u^{t^{i}}}}
  \end{align*}
  for $j\geq J$. 
  This leads to the bound
  \begin{equation*}
    \abs{R} \leq \abs{w}^J \left( \prod_{i=1}^J 
    \frac{\abs{q}^{\hp{i}}\abs{u}^{t^{i}}}{\abs{1-q^{\hp{i}}u^{t^{i}}}} \right)
    \sum_{j\geq J} Q^{j-J}
    = \abs{w}^J \left( \prod_{i=1}^J 
      \frac{\abs{q}^{\hp{i}}\abs{u}^{t^{i}}}{\abs{1-q^{\hp{i}}u^{t^{i}}}} \right)
    \frac{1}{1-Q},
  \end{equation*}
  which we wanted to show.
\end{proof}


We also need to truncate the infinite sums of derivatives of
$\f{b}{q,u,1,w}$. This is done by means of the following lemma.


\begin{lemma}\label{lem:den-bounds}
  Let $J\in\N$ and $\alpha$, $\beta$, $\gamma\in\N_0$, and let $q\in\C$ with
  $\abs{q}\leq\frac23$. Suppose
  \begin{itemize}
  \item either $u=1$, $U=1$ and $\beta=0$,
  \item or $u\in\C$ with $\abs{u} < 1/U - \frac{\log\sqrt{2}}{t^2}$ for $U$
    defined in Lemma~\ref{lemma:dominant-singularity}
  \end{itemize}
  holds. Further, let $w\in\C$ with
  $\abs{w}\leq\frac32$. Set
  \begin{equation*}
    Q=\frac53 \frac{1}{\bigl(\frac65\bigr)^{\hp{J+1}} U^{t^{J+1}} - 1},
  \end{equation*}
  and suppose $J$ was chosen such that $Q<1$ holds. Then
  \begin{multline*}
    \abs{\frac{\partial^{\alpha+\beta+\gamma}}{\partial q^\alpha 
        \partial u^\beta \partial w^\gamma} 
      \bigl(\f{b}{q,u,1,w} - \f{b_J}{q,u,1,w}\bigr)} \\
    \leq \alpha!\, \beta!\, \gamma!\, (t^2/\log\sqrt{2})^\beta 
    6^{\alpha+\gamma} \left(\frac53\right)^J
    \left( \prod_{i=1}^J \frac{1}{\bigl(\frac65\bigr)^{\hp{i}} 
        U^{t^i} - 1} \right)
    \frac{1}{1-Q}.
  \end{multline*}
\end{lemma}


\begin{proof}
  Let $\vartheta\in\C$ with $\abs{\vartheta} < 1/U$ and $\eta\in\C$ with
  $\abs{\eta}\leq\frac53$. Cauchy's integral formula gives
  \begin{equation*}
    \frac{\partial^\alpha}{\partial q^\alpha} 
      \bigl(\f{b}{q,\vartheta,1,\eta} - \f{b_J}{q,\vartheta,1,\eta}\bigr)
    = \frac{\alpha!}{2\pi i} \oint_{\abs{\xi-q}=\frac16}
    \frac{\f{b}{\xi,\vartheta,1,\eta} 
      - \f{b_J}{\xi,\vartheta,1,\eta}}{(\xi-q)^{\alpha+1}} \dd \xi.
  \end{equation*}
  The bound on $q$ implies $\abs{\xi}\leq\frac56$. Using the standard estimate for complex integrals, \eqref{eq:U-estimates} and Lemma~\ref{lem:den-bounds-pre} yield
  \begin{equation*}
    \abs{\frac{\partial^\alpha}{\partial q^\alpha} 
      \bigl(\f{b}{q,\vartheta,1,\eta} - \f{b_J}{q,\vartheta,1,\eta}\bigr)}
    \leq \alpha!\, 6^\alpha \left(\frac53\right)^J
    \left( \prod_{i=1}^J \frac{1}{\bigl(\frac65\bigr)^{\hp{i}}U^{t^i}-1} \right)
    \frac{1}{1-Q}.
  \end{equation*}
  Note that the right hand side is independent of $q$, $\vartheta$ and $\eta$,
  and, as $J$ tends to infinity, this bound is going to zero. Therefore, for
  fixed $\vartheta$ and $\eta$, the series $\frac{\partial^\alpha}{\partial
    q^\alpha} \f{b}{q,\vartheta,1,\eta}$ converges uniformly on the compact set
  $\{q \mid \abs{q}\leq\frac23\}$. Thus, for $\vartheta$ with $\abs{\vartheta} < 1/U$
  and $\eta$ with $\abs{\eta}\leq\frac53$, this function is analytic. Note that
  this result stays true if $\vartheta=1$ and $U=1$.

  We use Cauchy's integral formula again and obtain
  \begin{align*}
    \abs{\frac{\partial^\gamma}{\partial w^\gamma}
      \frac{\partial^\alpha}{\partial q^\alpha} 
      \bigl(\f{b}{q,\vartheta,1,w} - \f{b_J}{q,\vartheta,1,w}\bigr)}
    &= \abs{\frac{\gamma!}{2\pi i} \oint_{\abs{\eta-w}=\frac16}
      \frac{\frac{\partial^\alpha}{\partial q^\alpha} 
        \bigl(\f{b}{q,\vartheta,1,\eta} 
        - \f{b_J}{q,\vartheta,1,\eta}\bigr)}{(\eta-w)^{\gamma+1}} \dd \eta} \\
    &\leq \gamma!\, 6^\gamma \abs{\frac{\partial^\alpha}{\partial q^\alpha} 
        \bigl(\f{b}{q,\vartheta,1,w} - \f{b_J}{q,\vartheta,1,w}\bigr)} \\
    &\leq \alpha!\, \gamma!\, 6^{\alpha+\gamma} \left(\frac53\right)^J
    \left( \prod_{i=1}^J \frac{1}{\bigl(\frac65\bigr)^{\hp{i}}
        U^{t^i} - 1} \right)
    \frac{1}{1-Q}.
  \end{align*}
  Note that $\abs{w}\leq\frac32$ implies $\abs{\eta}\leq\frac53$. Moreover,
  $\frac{\partial^{\alpha+\gamma}}{\partial q^\alpha \partial w^\gamma}
  \f{b}{q,\vartheta,1,w}$ is analytic in $\vartheta$ with $\abs{\vartheta} <
  1/U$. Again, this result stays true if $\vartheta=1$ and $U=1$.

  Using Cauchy's integral formula once more yields
  \begin{multline*}
    \abs{\frac{\partial^\alpha}{\partial q^\alpha}
      \frac{\partial^\beta}{\partial u^\beta}
      \frac{\partial^\gamma}{\partial w^\gamma} 
      \bigl(\f{b}{q,u,1,w} - \f{b_J}{q,u,1,w}\bigr)}\\
    \begin{aligned}
    &= \abs{\frac{\beta!}{2\pi i} \oint_{\abs{\vartheta-u}=\frac{\log\sqrt{2}}{t^2}}
      \frac{\frac{\partial^\alpha}{\partial q^\alpha}
        \frac{\partial^\gamma}{\partial w^\gamma}
        \bigl(\f{b}{q,\vartheta,1,w} - \f{b_J}{q,\vartheta,1,w}\bigr)}{
        (\vartheta-u)^{\beta+1}} \dd \vartheta} \\
    &\leq \beta!\, (t^2 / \log\sqrt{2})^\beta 
    \abs{\frac{\partial^\alpha}{\partial q^\alpha}
      \frac{\partial^\gamma}{\partial w^\gamma}
      \bigl(\f{b}{q,u,1,w} - \f{b_J}{q,u,1,w}\bigr)},
    \end{aligned}
  \end{multline*}
  which is the desired result after inserting the bound from above.
\end{proof}


In Section~\ref{sec:distinct-depths} we analyze the distinct depths of
leaves. Again, we work with infinite sums by replacing them with finite sums
and bounding the error we make. Similar to the estimates above, we define 
\begin{equation*}
  b_J(q,1,v,1) =
  \sum_{1\leq j<J} \frac{v q^{\hp{j}}}{1-q^{\hp{j}}} 
  \prod_{i=1}^{j-1} \frac{1-v-q^{\hp{i}}}{1-q^{\hp{i}}}
\end{equation*}
and have the following two lemmata.


\begin{lemma}\label{lem:d:den-bounds-pre}
  Let $J\in\N$, $q\in\C$ with $\abs{q}<1$ and $v\in\C$. Set
  \begin{equation*}
    Q = \abs{q}^{t^J} \left(1+\frac{\abs{v}}{1-\abs{q}^{\hp{J}}}\right)
  \end{equation*}
  and suppose $Q<1$ holds. Then
  \begin{equation*}
    \abs{\f{b}{q,1,v,1} - \f{b_J}{q,1,v,1}}
    \leq \frac{\abs{v} \abs{q}^{\hp{J}}}{\abs{1-q^{\hp{J}}}}
    \left(\prod_{i=1}^{J-1} 
      \left(1+\frac{\abs{v}}{\abs{1-q^{\hp{i}}}} \right)\right)
    \frac{1}{1-Q}.
  \end{equation*}
\end{lemma}


\begin{proof}
  Set
  \begin{equation*}
    R = \f{b}{q,1,v,1} - \f{b_J}{q,1,v,1} =
    \sum_{j\geq J} \frac{v q^{\hp{j}}}{1-q^{\hp{j}}} 
    \prod_{i=1}^{j-1} \frac{1-v-q^{\hp{i}}}{1-q^{\hp{i}}}.
  \end{equation*}
  Let $j\geq J$. We have
  \begin{equation*}
    \hp{j} = \hp{J} + t^J \hp{j-J} \geq \hp{J} + t^J(j-J).
  \end{equation*}
  Therefore, for $j\geq J$ we obtain
  \begin{equation*}
    \abs{q^{\hp{j}} \prod_{i=1}^{j-1} \frac{1-v-q^{\hp{i}}}{1-q^{\hp{i}}}}
    \leq \abs{q}^{\hp{J}} \abs{q}^{t^J(j-J)}
    \left(1+\frac{\abs{v}}{1-\abs{q}^{\hp{J}}}\right)^{j-J} 
    \prod_{i=1}^{J-1} \left(1+\frac{\abs{v}}{\abs{1-q^{\hp{i}}}}\right).
  \end{equation*}
  This leads to the bound
  \begin{equation*}
    \abs{R} \leq \frac{\abs{v} \abs{q}^{\hp{J}}}{\abs{1-q^{\hp{J}}}}
    \left(\prod_{i=1}^{J-1} 
      \left(1+\frac{\abs{v}}{\abs{1-q^{\hp{i}}}} \right)\right)
    \sum_{j\geq J} Q^{j-J},
  \end{equation*}
  which we wanted to show.
\end{proof}


The result of the previous lemma can be extended to derivatives, see below. The
proof is skipped as it is very similar to the proof of Lemma~\ref{lem:den-bounds}.


\begin{lemma}\label{lem:d:den-bounds}
  Let $J\in\N$, $\alpha\in\N_0$ and $\gamma\in\N_0$. Further, let $q\in\C$ with
  $\abs{q}\leq\frac23$ and $v\in\C$ with $\abs{v}\leq\frac32$. Set
  \begin{equation*}
    Q = \left(\frac56\right)^{t^J} 
    \Biggl(1+\frac{\frac53}{1-\bigl(\frac56\bigr)^{\hp{J}}}\Biggr)
  \end{equation*}
  and suppose $J$ was chosen such that $Q<1$ holds. Then
 \begin{equation*}
    \abs{\frac{\partial^{\alpha+\gamma}}{\partial q^\alpha \partial v^\gamma} 
      \bigl(\f{b}{q,1,v,1} - \f{b_J}{q,1,v,1}\bigr)}
    \leq \alpha!\, \gamma!\, 6^{\alpha+\gamma} 
    \frac{\frac53}{\bigl(\frac65\bigr)^{\hp{J}}-1}
    \left(\prod_{i=1}^{J-1} 
      \left(1+\frac{\frac53}{1-\bigl(\frac56\bigr)^{\hp{i}}} \right)\right)
    \frac{1}{1-Q}.    
  \end{equation*} 
\end{lemma}


\section{The Height}
\label{sec:height}


We start our analysis with the height $h(T)$ of a canonical tree
$T\in\calT$. It turns out that the height is asymptotically (for large sizes $n=n(T)$) normally
distributed, and we will even prove a local limit theorem for it. Moreover, we obtain asymptotic expressions for its
mean and variance. This will be achieved by means of the generating function
$H(q,u,v,w)$ derived in Section~\ref{sec:generating-function}.
 

So let us have a look at the bivariate generating function
\begin{equation*}
  H(q,1,1,w) = \sum_{T\in\calT} q^{n(T)}w^{h(T)}
  = \frac{a(q,1,1,w)}{1-b(q,1,1,w)}
\end{equation*}
for the height. We consider its denominator
\begin{equation*}
  D(q,w) \colonequals 1-b(q,1,1,w) 
  = \sum_{j\geq0} (-1)^j w^j \prod_{i=1}^j \frac{q^{\hp{i}}}{1-q^{\hp{i}}}.
\end{equation*}
From Lemma~\ref{lemma:dominant-singularity} we know that $D(q,1)$ has a simple
dominant zero~$q_0$. We can see the expansion of $D(q,w)$ around $(q_0,1)$ as
perturbation of a meromorphic singularity, cf.\ the book of Flajolet and
Sedgewick~\cite[Section IX.6]{Flajolet-Sedgewick:ta:analy}. This yields a central limit
theorem (normal distribution) for the height without much effort. But we can do
better: we can show a local limit theorem for the height. The precise results
are stated in the following theorem.


\begin{theorem}\label{thm:height}
  For a randomly chosen tree $T\in\calT$ of size $n$ the height $\f{h}{T}$ is
  asymptotically (for $n\to\infty$) normally distributed, and a local limit theorem holds.
Its mean is $\mu_h n + O(1)$ and its variance is $\sigma_h^2 n + O(1)$ with
  \refstepcounter{epsilon}\label{eps:h:m}
  \refstepcounter{epsilon}\label{eps:h:v}
  \begin{align}
    \mu_h &=\frac{\f{\frac{\partial}{\partial w} b}{q_0,1,1,w}\rvert_{w=1}}{
   q_0 \f{\frac{\partial}{\partial q} b}{q,1,1,1}\rvert_{q=q_0}}\label{eq:mu_h_explicit}\\&=\frac{1}{2} + \frac{t-2}{2^{t+3}} + \frac{2t^2+3t-8}{2^{2t+5}}
    + \frac{9t^3+45t^2+2t-88}{2^{3t+8}}
    + \frac{0.55t^4}{2^{4t}} \varepsilon_{\ref{eps:h:m}}(t)\notag
  \intertext{and}
    \sigma_h^2 &= \frac{1}{4} + \frac{-t^2+5t-2}{2^{t+4}}
    + \frac{-4t^3+4t^2+27t-14}{2^{2t+6}}
    + \frac{0.26t^4}{2^{3t}} \varepsilon_{\ref{eps:h:v}}(t)\notag
  \end{align}
  for $t\geq2$.
\end{theorem}


Recall that ``randomly chosen'' here and everywhere else in this article means
``uniformly chosen at random'' and that the error functions
$\f{\varepsilon_j}{\ldots}$ are functions with absolute value bounded by
$1$, see also the last paragraph of the introduction.


We calculated the values of the constants $\mu_h$ and $\sigma_h^2$ numerically
for $2\leq t\leq 30$. Those values can be found in
Table~\ref{tab:special-values-height}.
Figure~\ref{fig:height} shows the result of Theorem~\ref{thm:height}. It
compares the obtained normality with the distribution of the height calculated
for particular values in SageMath.


\tableheight{}


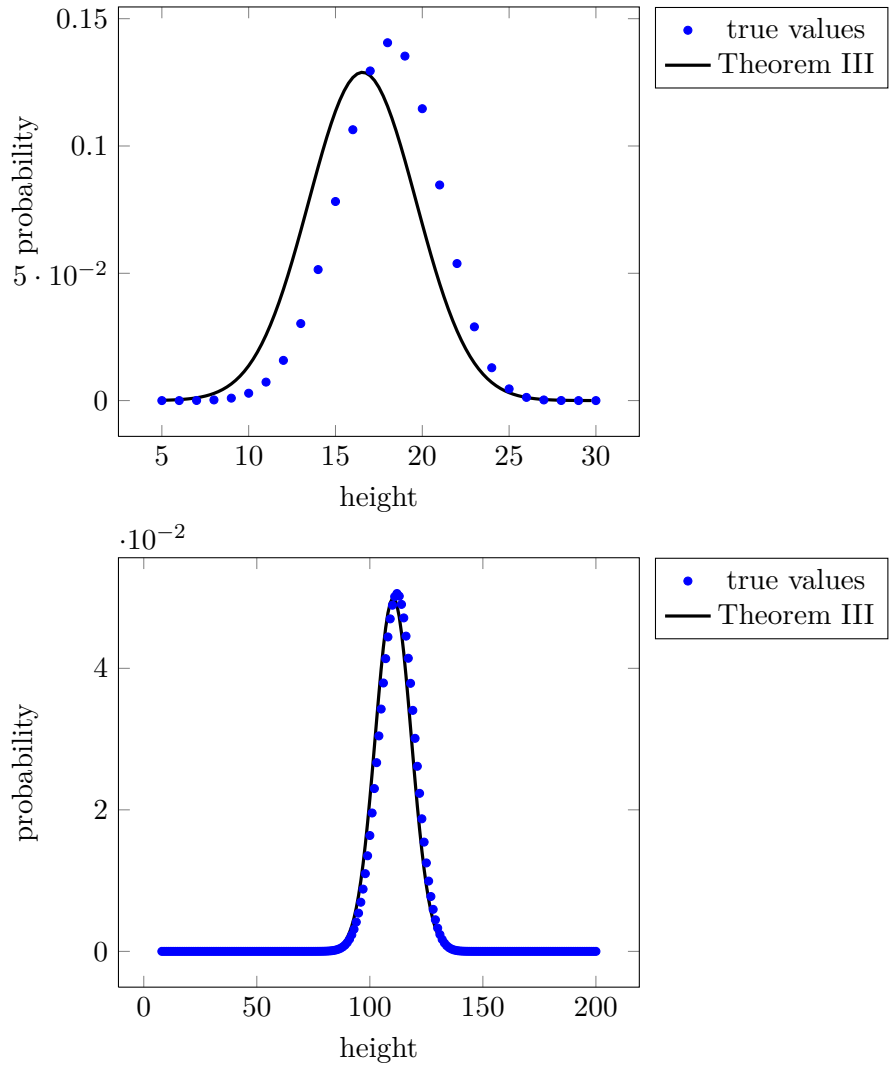
\begin{figure}
  \centering
   \begin{tikzpicture}
    \begin{axis}[
      xlabel=height,
      ylabel=\hspace*{2.5em}probability,
      legend pos=outer north east]
    \addplot[color=blue, mark=*, only marks, mark size=1.5]
         table[x=height30, y=count] {plot_height30_true.txt};
    \addplot[color=black, no marks, very thick]
         table[x=height30, y=count] {plot_height30_asy.txt};
    \legend{true values, Theorem~\ref{thm:height}}
    \end{axis}
  \end{tikzpicture}
   \begin{tikzpicture}
    \begin{axis}[
      xlabel=height,
      ylabel=probability,
      legend pos=outer north east]
    \addplot[color=blue, mark=*, only marks, mark size=1.5]
         table[x=height200, y=count] {plot_height200_true.txt};
    \addplot[color=black, no marks, very thick]
         table[x=height200, y=count] {plot_height200_asy.txt};
    \legend{true values, Theorem~\ref{thm:height}}
    \end{axis}
  \end{tikzpicture}
 
  \caption{Distribution of the height for $t=2$, and $n=30$
    (top figure) and $n=200$ (bottom figure) inner vertices. On
    the one hand, this figure shows the true distribution of all trees of the
    given size and on the other hand the result on the asymptotic normal
    distribution (Theorem~\ref{thm:height} with only main terms of mean and
    variance taken into account).}
  \label{fig:height}
\end{figure}


\begin{remark}\label{rem:numerical-calc}
For the (central and local) limit theorem to hold, it is essential that $\sigma_h^2 \neq 0$,
which is why we need reliable numerical values and estimates for large $t$. As mentioned earlier,
we used interval arithmetic in  SageMath~\cite{Stein-others:2015:sage-mathem-6.5} in all our numerical
  calculations to achieve such results. We used a precision of $53$
  bits (machine precision) for the bounds of the intervals. All values are calculated to such a precision that the
error is at most the magnitude of the last digit that occurs. The reason for the varying
  number of digits after the decimal point (in, for example, Table~\ref{tab:special-values-height})
  are numerical artifacts. In these cases, we could have given an additional digit at the cost of a slightly greater
error (twice the magnitude of the last digit).
\end{remark}


The proof of Theorem~\ref{thm:height} is split up into several parts. At first,
we get asymptotic normality (central limit theorem) and the constants for mean
and variance by using Theorem~IX.9 (meromorphic singularity perturbation) from
the book of Flajolet and Sedgewick~\cite{Flajolet-Sedgewick:ta:analy}. For the
local limit theorem we need to analyze the absolute value of the dominant zero
$\f{q_0}{w}$ of the denominator $\f{D}{q,w}$ of the generating function
$\f{H}{q,1,1,w}$. Going along the unit circle, i.e., taking $w=e^{i\varphi}$,
this value has to have a unique minimum at $\varphi=0$.

From the combinatorial background of the problem (non-negativity of
coefficients) it is clear that $\abs{\f{q_0}{e^{i\varphi}}} \geq
\abs{\f{q_0}{1}}$. The task showing the uniqueness of this minimum at
$\varphi=0$ is again split up: We show that the function
$\abs{\f{q_0}{e^{i\varphi}}}$ is convex in a region around $\varphi=0$ (central
region), see Lemmata~\ref{lem:h:one-root-rouche}
to~\ref{lem:h:qpp-bound-near-1}. For the outer region, where $\varphi$ is not
near $0$, we show that zeros of the denominator are larger there. This is done
in Lemma~\ref{lem:h:outer-region}.

Those lemmata mentioned above showing that the minimum is unique work for all
general $t\geq 30$. For the remaining $t$, precisely, for each $t$ with $2\leq
t\leq 30$, the same ideas are used, but the checking is done algorithmically
using interval arithmetic and the mathematics software system
SageMath~\cite{Stein-others:2015:sage-mathem-6.5}. Details are given in
Remark~\ref{rem:h:small-t}.


So much for the idea of the proof. We start the actual proof by analyzing the
denominator $\f{D}{q,w}$. For our calculations we will truncate this infinite
sum and use the finite sum
\begin{equation*}
  \f{D_J}{q,w} \colonequals 
  \sum_{0 \leq j < J} (-1)^j w^j \prod_{i=1}^j \frac{q^{\hp{i}}}{1-q^{\hp{i}}}
\end{equation*}
instead. Bounds for the tails (difference between the infinite and the finite
sum) are given by Lemma~\ref{lem:den-bounds-pre}. In particular, 
we write down the special case $J=2$ of this lemma, which will be needed a
couple of times in this section. Substituting $1/z$ for $q$, we get
\begin{equation}\label{eq:h:den-bound-J2}
  \abs{\f{D}{1/z,w} - \f{D_2}{1/z,w}}
  \leq \abs{w}^2 \frac{1}{\abs{z-1}} \frac{1}{\abs{z}^{1+t}-1} 
  \frac{1}{1-\abs{w}/(\abs{z}^{1+t+t^2}-1)},
\end{equation}
under the assumption $\abs{w} < \abs{z}^{1+t+t^2}-1$. Derivatives of $\f{D}{q,w}$ are handled by Lemma~\ref{lem:den-bounds}.

As mentioned earlier, the proof of the local limit theorem for the height for general
$t$ consists of two parts: one for $w$ in the central region (around $w=1$)
and one for $w$ in the outer region. The following lemma shows that
everything is fine in the outer region. After that, a couple of lemmata are
needed to prove our result for the central region.


\begin{lemma}\label{lem:h:outer-region}
  Let $w = e^{i\varphi}$, where $\varphi$ is real with $\sqrt{97/96}\,\pi\,
  2^{-t/2} < \abs{\varphi} \le \pi$. Then each zero of $z\mapsto\f{D}{1/z,w}$
  has absolute value smaller than $2-1/2^t$.
\end{lemma}


\begin{proof}
  Suppose that we have a zero $z_0$ of the denominator $\f{D}{1/z,w}$ for a
  given $w$ and that this zero fulfils $\abs{z_0} \geq 2-1/2^t$. We can
  extend the equation $\f{D}{1/z_0,w}=0$ to
  \begin{equation*}
    0 = 1 - \frac{w}{z_0-1} + \f{D}{1/z_0,w} - \f{D_2}{1/z_0,w},
  \end{equation*}
  which can be rewritten as
  \begin{equation*}
    z_0 = 1 + w - (z_0-1) \left(\f{D}{1/z_0,w} - \f{D_2}{1/z_0,w}\right).
  \end{equation*}
  Taking absolute values and using bound~\eqref{eq:h:den-bound-J2} obtained
  from Lemma~\ref{lem:den-bounds} yields
  \begin{equation*}
    \abs{z_0}
    \leq \abs{1+w} + \frac{1}{\abs{z_0^{\hp{2}}-1}} 
    \frac{1}{1-1/\bigl(\abs{z_0}^{\hp{3}}-1\bigr)}.
  \end{equation*}
  We have the lower bounds
  \begin{equation*}
    \abs{z_0^{\hp{2}}-1} \geq \abs{z_0}^{t+1} - 1
    \geq \left(2-\frac{1}{2^t}\right)^{t+1}-1 
    = 2^{t+1} \left(1-\frac{1}{2^{t+1}}\right)^{t+1}-1 \geq 2^t
  \end{equation*}
  and
  \begin{equation*}
    \abs{z_0^{\hp{3}}-1} \geq \abs{z_0}^{t^2+t+1} - 1
    \geq 2^{t^2+t+1}\left(1-\frac{1}{2^{t+1}}\right)^{t^2+t+1}-1
    \geq \frac{807159}{16384} \geq 49,
  \end{equation*}
  which can be found by using monotonicity and the value at $t=2$.  Therefore,
  we obtain
  \begin{equation}\label{eq:h:outer-bound-z0}
    \abs{z_0} \leq \abs{1+w} + \frac{49}{48} \frac{1}{2^t}.
  \end{equation}
  Since we have assumed $\abs{z_0} \geq 2-1/2^t$, we deduce
  \begin{equation*}
    \abs{1+w} \geq 2 - \frac{97}{48}\frac{1}{2^t}.
  \end{equation*}
  On the other hand, using $\abs{\varphi} > \sqrt{97/96}\,\pi\, 2^{-t/2}$
  and the inequality $\abs{\sin(\varphi/4)} \geq \abs{\varphi} / (\sqrt{2}\pi)$
  for $\abs{\varphi}\leq\pi$ (which follows by concavity of the sine on the
  interval $[0, \frac{\pi}{4}]$), we have
  \begin{equation*}
    \abs{1+w} = \sqrt{2+2\cos\varphi} 
    = 2\left(1 - 2 \sin^2 \frac\varphi4\right)
    \leq 2-\frac{2}{\pi^2}\varphi^2
    < 2 - \frac{97}{48}\frac{1}{2^t},
  \end{equation*}
  which yields a contradiction.
\end{proof}


Next, we come to the central region. Looking at the assumptions used in
Lemma~\ref{lem:h:outer-region}, this is when $\abs{\varphi} \leq
\sqrt{97/96}\,\pi\, 2^{-t/2}$. As mentioned in the sketch of the proof, we show
that the function $\abs{\f{q_0}{e^{i\varphi}}}$ is convex. 


We know the location of the dominant and second dominant zero of the
denominator~$\f{D}{q,1}$. As we need those roots for general~$w$ (along the
unit circle), we analyze the difference of $\f{D}{q,w}$ from
$\f{D}{q,1}$. Using Rouch\'e's theorem then yields a bound for the dominant
zero, which is stated precisely in the following lemma.


\begin{lemma}\label{lem:h:one-root-rouche}
  Suppose $t\geq5$ and $\abs{w-1} \leq \frac12 -
  5\bigl(\frac23\bigr)^t$. Then $q\mapsto\f{D}{q,w}$ has exactly one
  root with $\abs{q}<\frac23$ and no root with $\abs{q}=\frac23$.
\end{lemma}


\begin{proof}
  We use Rouch\'e's theorem on the circle $\abs{q}=\frac23$. With
  $\abs{w}\leq\frac32$, $\abs{q}=\frac23$ and the
  bound~\eqref{eq:h:den-bound-J2} (obtained from Lemma~\ref{lem:den-bounds})
  we get
  \begin{equation*}
    \abs{\f{D}{q,w} - \f{D_2}{q,w}}
    \leq \frac{9}{2} \frac{1}{(3/2)^{1+t}-1} 
    \frac{1}{1-(3/2)/((3/2)^{1+t+t^2}-1)} \leq 3.29 \left(\frac23\right)^t = b,
  \end{equation*}
  where we took out the factor $(2/3)^t$ and used monotonicity together with the value for
  $t=5$.
  
  With $\f{D_2}{q,w} = 1 - wq/(1-q)$ we obtain
  \begin{align*}
    \abs{\f{D}{q,w} - \f{D}{q,1}}
    &\leq \abs{\f{D}{q,w} - \f{D_2}{q,w}}
    + \abs{\f{D_2}{q,w} - \f{D_2}{q,1}}
    + \abs{\f{D_2}{q,1} - \f{D}{q,1}} \\
    &\leq 2b + \abs{w-1}\abs{\frac{q}{1-q}} 
    \leq 2b + 2\abs{w-1}
    \leq 1 + 2b - 10\left(\frac23\right)^t
    < 1 - b.
  \end{align*}
  On the other hand, the M\"obius transform $q\mapsto 1-q/(1-q)$ maps the
  circle $\abs{q}=2/3$ to the circle $\abs{z-1/5}=6/5$. Therefore
  $\abs{1-q/(1-q)} \geq 1$, and so we have
  \begin{equation*}
     \abs{\f{D}{q,1}} \geq
     \abs{1-\frac{q}{1-q}} - \abs{\f{D}{q,1} - \f{D_2}{q,1}}
     \geq 1 - b. 
  \end{equation*}
  This proves the lemma by Rouch\'e's theorem and
  Lemma~\ref{lemma:dominant-singularity}.
\end{proof}


The previous lemma gives us exactly one value $\f{q_0}{w}$ for each $w$ in a
region around $1$. We continue by showing that this function $q_0$ is analytic.


\begin{lemma}\label{lem:h:qw-analytic}
  For $t\geq5$ and $\abs{w-1} \leq \frac12 - 5\bigl(\frac23\bigr)^t$, the
  function $\f{q_0}{w}$ given implicitly by $\f{D}{\f{q_0}{w}, w} = 0$,
  $\abs{\f{q_0}{w}}<\frac23$, is analytic.
\end{lemma}


\begin{proof}
  We follow the lines of the proof of the Analytic Inversion Lemma, cf.\
  Flajolet and Sedgewick~\cite{Flajolet-Sedgewick:ta:analy},
  Chapter~IV.7. Consider the function
  \begin{equation*}
    \f{\sigma_1}{w} = \frac{1}{2\pi i} \oint_{\abs{q}=\frac23}
    \frac{\frac{\partial}{\partial q} \f{D}{q,w}}{\f{D}{q,w}}
    q \dd q.
  \end{equation*}
  Since $\f{D}{q,w}\neq 0$ for all $q$ and $w$ allowed by the assumptions, this
  function is continuous. Moreover, using the theorems of Morera and Fubini,
  and Cauchy's integral theorem, the function $\sigma_1$ is analytic. By
  Lemma~\ref{lem:h:one-root-rouche} and by using the residue theorem we get
  that $\f{\sigma_1}{w}$ equals $q$ fulfilling $\f{D}{q,w}=0$ and
  $\abs{q}<\frac23$, i.e., we obtain $\f{\sigma_1}{w} = \f{q_0}{w}$.
\end{proof}


Since we have analyticity of $q_0$ in a region around $1$ by
Lemma~\ref{lem:h:qw-analytic}, we can show that small changes in $w$ do not
matter much, see the following lemma for details. Later, this is used to
estimate the derivative at some point~$w$ by the derivative at~$1$.


\begin{lemma}\label{lem:h:qpp-bound-near-1}
  Let $t\geq 30$ and $w=e^{i\varphi}$, where $\varphi\in\R$ with $\abs{\varphi}
  \leq \sqrt{97/96}\,\pi\, 2^{-t/2}$. We have the inequalities
  \begin{equation*}
    \abs{\f{q_0}{w} - \f{q_0}{1}} \leq \frac{5}{2^{t/2}},
    \quad
    \abs{\f{q_0'}{w} - \f{q_0'}{1}} \leq \frac{17}{2^{t/2}}
    \quad\text{and}\quad
    \abs{\f{q_0''}{w} - \f{q_0''}{1}} \leq \frac{102}{2^{t/2}}.
  \end{equation*}
\end{lemma}


\begin{proof}Set $d = \frac12 - 5\left(\frac23\right)^t$. By
  Lemma~\ref{lem:h:qw-analytic} the function $\f{q_0}{w}$ is analytic for
  $\abs{w-1}\leq d$. Therefore, by Cauchy's integral formula, we get
  \begin{equation*} 
    \f{q_0^{(k)}}{w} - \f{q_0^{(k)}}{1} = 
    \frac{k!}{2\pi i} \oint_{\abs{\zeta-1}=d} 
    \left( \frac{\f{q_0}{\zeta}}{(\zeta-w)^{k+1}}
      -  \frac{\f{q_0}{\zeta}}{(\zeta-1)^{k+1}} \right) \dd \zeta
  \end{equation*}
  for $k\in\N_0$, where $q_0^{(k)}$ denotes the $k$\nobreakdash-th derivative
  of $q_0$. For its absolute value we obtain
  \begin{equation*}
    \abs{\f{q_0^{(k)}}{w} - \f{q_0^{(k)}}{1}}
    \leq k!\, d \max_{\abs{\zeta-1}=d} \abs{\f{q_0}{\zeta}}
    \max_{\abs{\zeta-1}=d} \abs{(\zeta-w)^{-(k+1)} - (\zeta-1)^{-(k+1)}}.
  \end{equation*}
  We have $\abs{\f{q_0}{\zeta}}<\frac23$ by
  Lemma~\ref{lem:h:one-root-rouche}. Further, we get
  \begin{align*}
    \abs{(\zeta-w)^{-(k+1)} - (\zeta-1)^{-(k+1)}}
    &= \abs{\int_1^w \frac{\partial}{\partial \xi} 
      (\zeta-\xi)^{-(k+1)} \dd \xi} \\
    &\leq \abs{w-1} (k+1) \max_{\xi\in[1,w]} \abs{\zeta-\xi}^{-(k+2)}.
  \end{align*}
Since
  \begin{equation*}
    \abs{\xi-1} \leq \abs{w-1} = \big\lvert e^{i\varphi}-1\big\rvert
    \leq \abs{i\int_0^{\varphi} e^{it} \dd t}
    \leq \abs{\varphi},
  \end{equation*}
  we have $\abs{\zeta-\xi} \geq d - \abs{\varphi}$.  Collecting all those
  results, and using $d\leq\frac12$ and the bound given for $\abs{\varphi}$
  results in
  \begin{equation*}
    \abs{\f{q_0^{(k)}}{w} - \f{q_0^{(k)}}{1}}
    \leq \frac{(k+1)!}{3}\abs{w-1} 
    \left(\frac12 - 5\left(\frac23\right)^t 
      - \sqrt{\frac{97}{96}}\,\pi\, 2^{-t/2}\right)^{-(k+2)}.
  \end{equation*}
  Inserting all bounds gives the estimates stated for $k\in\{0,1,2\}$.
\end{proof}


Now we are ready to show that the second derivative of
$\abs{\f{q_0}{e^{i\varphi}}}$ is positive. To do so, we show that this second
derivative is around $\frac18$ for $\varphi=0$ and use the bounds of
Lemma~\ref{lem:h:qpp-bound-near-1} to conclude positivity for $w$ in some
region around~$1$.


\begin{lemma}\label{lem:h:qpp-positive}
  If $t\geq 30$ and $\varphi\in\R$ with $\abs{\varphi} \leq
  \sqrt{97/96}\,\pi\, 2^{-t/2}$, then
  \begin{equation*}
    \frac{\dd^2}{\dd \varphi^2} \abs{\f{q_0}{e^{i\varphi}}}^2 > 0.
  \end{equation*}
\end{lemma}


\begin{proof}
  Write
  \begin{equation*}
    \Delta_w = \left. \frac{\partial}{\partial w} \f{D}{q, w} 
    \right\vert_{q=\f{q_0}{w}}
    \qquad\text{and}\qquad
    \Delta_q = \left. \frac{\partial}{\partial q} \f{D}{q, w} 
    \right\vert_{q=\f{q_0}{w}},
  \end{equation*}
  and analogously $\Delta_{qq}$, $\Delta_{qw}$ and $\Delta_{ww}$ for the
  function $\f{D}{q, w}$ derived twice and then evaluated at $q=\f{q_0}{w}$.
  By inserting the asymptotic expansion of $q_0$, see
  Lemma~\ref{lemma:dominant-singularity}, into the expressions
  \begin{equation}\label{eq:h:der-q0-implicit}
    \f{q_0'}{w} = - \frac{\Delta_w}{\Delta_q}
    \qquad\text{and}\qquad
    \f{q_0''}{w} = \frac{2 \Delta_{qw} \Delta_w \Delta_q 
      - \Delta_{qq} \Delta_w^2 - \Delta_{ww} \Delta_q^2}{\Delta_q^3}
  \end{equation}
  obtained by implicit differentiation, we find
  \begin{equation*}
    \f{q_0'}{1} = -\frac14 + \frac{0.07t}{2^t} \f{\neweps}{t}
    \qquad\text{and}\qquad
    \f{q_0''}{1} = \frac14 + \frac{0.04t^2}{2^t} \f{\neweps}{t}.
  \end{equation*}
  For the calculations themselves, we used the approximation $\f{D_3}{q,w}$ of
  the denominator $\f{D}{q,w}$ together with the bound for the tail given in
  Lemma~\ref{lem:den-bounds}.
  
  Set $w=e^{i\varphi}$. Using the bounds of Lemma~\ref{lem:h:qpp-bound-near-1}
  yields
  \begin{align*}
    \f{q_0}{e^{i\varphi}}&=\frac12+\frac{6}{2^{t/2}}\f{\neweps}{t},\\
    \f{q'_0}{e^{i\varphi}}&=-\frac14+\frac{18}{2^{t/2}}\f{\neweps}{t},\\
    \f{q''_0}{e^{i\varphi}}&=\frac14+\frac{103}{2^{t/2}}\f{\neweps}{t}.    
  \end{align*}

  We define $\f{x}{\varphi}$ and $\f{y}{\varphi}$ to be the real and imaginary
  parts of $\f{q_0}{e^{i\varphi}}$, respectively. Thus 
  \begin{align*}
    \f{x}{\varphi} + i\f{y}{\varphi} &= \f{q_0}{e^{i\varphi}}, \\
    \f{x'}{\varphi} + i\f{y'}{\varphi} &= i e^{i\varphi} \f{q_0'}{e^{i\varphi}}
    \intertext{and}
    \f{x''}{\varphi} + i\f{y''}{\varphi} &= 
    - e^{i\varphi} \f{q_0'}{e^{i\varphi}} - e^{2i\varphi} \f{q_0''}{e^{i\varphi}}.
  \end{align*}
  Then, the estimates above lead to
  \begin{align*}
    \f{x}{\varphi} &= \frac12 + \frac{6}{2^{t/2}} \f{\neweps}{t}, &
    \f{y}{\varphi} &= \frac{6}{2^{t/2}} \f{\neweps}{t}, \\
    \f{x'}{\varphi} &= \frac{19}{2^{t/2}} \f{\neweps}{t}, &
    \f{y'}{\varphi} &= -\frac14 + \frac{19}{2^{t/2}} \f{\neweps}{t}, \\
    \f{x''}{\varphi} &= \frac{124}{2^{t/2}} \f{\neweps}{t}, &
    \f{y''}{\varphi} &= \frac{124}{2^{t/2}} \f{\neweps}{t}.
  \end{align*}
  These in turn together with
  \begin{equation}\label{eq:h:der2-abs-q0}
     \frac{\dd^2}{\dd \varphi^2} \abs{\f{q_0}{e^{i\varphi}}}^2
    = 2(\f{x'}{\varphi}^2 + \f{y'}{\varphi}^2
    + \f{x}{\varphi} \f{x''}{\varphi}
    + \f{y}{\varphi} \f{y''}{\varphi})
  \end{equation}
  give us the second derivative
  \begin{equation*}
    \left. \frac{\dd^2}{\dd \varphi^2} \abs{\f{q_0}{e^{i\varphi}}}^2
      \right\vert_{\varphi=0}
    = \frac18 + \frac{144}{2^{t/2}} \f{\neweps}{t} > 0.1206,
  \end{equation*}
  which is what we wanted to show.
\end{proof}


\begin{remark}\label{rem:h:small-t}
  The ideas in this section presented so far can also be used to show the
  uniqueness of the minimum of $\abs{\f{q_0}{e^{i\varphi}}}$ at $\varphi=0$ for
  a fixed $t$. In particular, this works for $t<30$, where some of the results
  above do not apply. 

  For the calculations the mathematics software system
  SageMath~\cite{Stein-others:2015:sage-mathem-6.5} is used. Further, we use
  interval arithmetic for all operations. The checking for fixed $t$ is done in
  the following way. We start with the interval $[-4,4]$ for $\varphi$. In each
  step, we check if the second derivative (using
  Equations~\eqref{eq:h:der-q0-implicit} and \eqref{eq:h:der2-abs-q0}) is
  positive. If not, then we half each of the bounds of the interval and repeat
  the step above. When this stops, we end up with a region around $0$ that is
  convex. For its complementary, we now use a bisection method to show that
  $\abs{\f{q_0}{e^{i\varphi}}} > \abs{\f{q_0}{1}}$. Note that we can use an
  approximation $\f{D_J}{q,w}$ instead of the denominator $\f{D}{q,w}$, which
  can be compensated taking the bounds obtained in Lemma~\ref{lem:den-bounds}
  into account.

  For $2\leq t\leq 30$, those calculations were done with a positive result,
  i.e., the minimum at $\varphi=0$ is unique.
\end{remark}


Now we have all results together to prove the main theorem of this section.


\begin{proof}[Proof of Theorem~\ref{thm:height}]
  We use Theorem~IX.9 of Flajolet and
  Sedgewick~\cite{Flajolet-Sedgewick:ta:analy} and apply that theorem to the
  function $H(q,1,1,w)$. This gives us the mean and the variance and as a central limit
  asymptotic normality. In particular, we obtain 
  \begin{equation*}
    \E(h(T)) =
    \frac{[q^n]\frac{\partial}{\partial w}H(q, 1, 1, w)\rvert_{w=1}}{
      [q^n]H(q, 1, 1, 1)}.
  \end{equation*}
  By \eqref{eq:generating-function-simplified}, we have
  \begin{equation*}
    \left.\frac{\partial}{\partial w}H(q, 1, 1, w)\right\rvert_{w=1}=
    \frac{\f{a}{q,1,1,1}
      \f{\frac{\partial}{\partial w} b}{q,1,1,w}\rvert_{w=1}}{
      (1-b(q,1,1,1))^2}
    + \frac{\f{\frac{\partial}{\partial w} a}{q,1,1,w}\rvert_{w=1}}{
      1-b(q,1,1,1)}.
  \end{equation*}
  By singularity analysis, we can extract the asymptotics to get the linear
  behavior of this mean and in particular the constant
  \eqref{eq:mu_h_explicit}.
 
  For the local limit, we need a more refined analysis.
  Recall the notation $\f{D}{q,w}$ as the denominator of $\f{H}{q,1,1,w}$ and
  let $\f{q_0}{w}$ be given implicitly by $\f{D}{\f{q_0}{w},w} = 0$,
  $\abs{\f{q_0}{w}}<\frac23$, according to
  Lemma~\ref{lemma:dominant-singularity} and to
  Lemma~\ref{lem:h:one-root-rouche}. Set $q_0 = \f{q_0}{1}$ and
  \begin{equation*}
    c_{\alpha\gamma} = \left. \frac{\partial^{\alpha+\gamma}}{\partial q^\alpha \partial w^\gamma} 
      D(q,w) \right\rvert_{\text{$q=q_0$, $w=1$}}.
  \end{equation*}
  Then we obtain the asymptotic formula $\mu_h n + O(1)$ for the mean, with
  \begin{equation*}
    \mu_h = \frac{c_{01}}{c_{10}q_0},
  \end{equation*}
  and the variance is $\sigma_h^2 n + O(1)$ with
  \begin{equation*}
    \sigma_h^2 = 
    \frac{c_{01}^{2} c_{20} q_0 + c_{01} c_{10}^{2} q_0 - 2 \, c_{01} c_{10}
      c_{11} q_0 + c_{02} c_{10}^{2} q_0 
      + c_{01}^{2} c_{10}}{c_{10}^{3} q_0^{2}}.
  \end{equation*}
  
  To calculate the coefficients $c_{\alpha\gamma}$ we need derivatives of
  $D(q,w)$. In order to avoid working with infinite sums, we use the
  approximations $D_J(q,w)$.
  Lemma~\ref{lem:den-bounds} shows that the error made by using those
  approximations is small. For the calculations themselves,
  SageMath~\cite{Stein-others:2015:sage-mathem-6.5} was used. 

  To show the local limit theorem, we have to show
  \begin{equation*}
    \abs{\f{q_0}{e^{i\varphi}}} > \abs{\f{q_0}{1}}
  \end{equation*}
  for all non-zero $\varphi\in[-\pi,\pi]$, cf.\ Chapter~IX.9
  of~\cite{Flajolet-Sedgewick:ta:analy}.

  Let $t\geq30$. Lemma~\ref{lem:h:qpp-positive} states that
  $\abs{\f{q_0}{e^{i\varphi}}}$ is convex for $\abs{\varphi}\leq
  \sqrt{97/96}\,\pi\, 2^{-t/2}$, therefore the minimum at $\varphi=0$ is unique
  for these $\varphi$.

  For all other $\varphi$, the value of $\abs{\f{q_0}{e^{i\varphi}}}$ is greater
  than $1/(2-1/2^t) > 1/2+1/2^{t+2}$ by Lemma~\ref{lem:h:outer-region}. This
  value itself is greater than $\frac12 + 0.1251/2^t \geq
  \abs{\f{q_0}{1}}$. Therefore the minimum at $\varphi=0$ is unique and the
  local limit thorem follows for $t \geq 30$.

  When $t<30$, we use an algorithmic approach to check that the minimum at
  $\varphi=0$ is unique. The details can be found in
  Remark~\ref{rem:h:small-t}. 
\end{proof}


\section{The Number of Distinct Depths of Leaves}
\label{sec:distinct-depths}


In this section we study the number of distinct depths of leaves $d(T)$ of a
canonical tree $T\in\calT$, motivated by the interpretation as the number of
distinct code word lengths in Huffman codes. This parameter is also
asymptotically normally distributed, and we show a local limit theorem. The
approach is essentially the same as for the height. It is based on the
generating function $H(q,u,v,w)$ from Section~\ref{sec:generating-function}. To
analyse the parameter $d(T)$, we look at the bivariate generating function
\begin{equation*}
  H(q,1,v,1) = \sum_{T\in\calT} q^{n(T)}v^{d(T)}
  = \frac{a(q,1,v,1)}{1-b(q,1,v,1)}
\end{equation*}
for the number of distinct depths of leaves. Again, we consider its denominator
\ifproc
\begin{align*}
  D(q,v) \colonequals& 1-b(q,1,v,1) \\
  &= 1 - \sum_{1\le j} \frac{v q^{\hp{j}}}{1-q^{\hp{j}}} 
  \prod_{i=1}^{j-1} \frac{1-v-q^{\hp{i}}}{1-q^{\hp{i}}}
\end{align*}
\else
\begin{equation*}
  D(q,v) \colonequals 1-b(q,1,v,1) 
  = 1 - \sum_{1\le j} \frac{v q^{\hp{j}}}{1-q^{\hp{j}}} 
  \prod_{i=1}^{j-1} \frac{1-v-q^{\hp{i}}}{1-q^{\hp{i}}}
\end{equation*}
\fi and proceed as in the previous
section. Lemma~\ref{lemma:dominant-singularity} tells us the existence of a
simple dominant zero~$q_0$ of $D(q,1)$. Again, we expand the denominator
$D(q,v)$ around $(q_0,1)$ and use Theorem IX.9 from the book of Flajolet and
Sedgewick~\cite{Flajolet-Sedgewick:ta:analy} to get asymptotic normality. The
local limit theorem follows from considerations of the dominant zero of
$D(q,v)$ with $v$ on the unit circle. This results in the following theorem.


\begin{theorem}\label{thm:distinct-depths}
  For a randomly chosen tree $T\in\calT$ of size $n$ the number of distinct
  depths of leaves $d(T)$ is asymptotically (for $n\to\infty$) normally
  distributed, and a local limit theorem holds. Its mean is $\mu_d n + O(1)$ and its
  variance is $\sigma_d^2 n + O(1)$ with \ifproc
  \begin{multline*}
    \mu_d = \frac{1}{2} + \frac{t-4}{2^{t+3}} + \frac{2t^2-t-14}{2^{2t+5}} \\
    + \frac{9t^3+27t^2-76t-144}{2^{3t+8}} 
    + \frac{0.06t^4}{2^{4t}} \neweps(t)\label{eps:dh:m}
  \end{multline*}
  \else
  \begin{equation*}
    \mu_d = \frac{1}{2} + \frac{t-4}{2^{t+3}} + \frac{2t^2-t-14}{2^{2t+5}}
    + \frac{9t^3+27t^2-76t-144}{2^{3t+8}} 
    + \frac{0.06t^4}{2^{4t}} \neweps(t)\label{eps:dh:m}
  \end{equation*}
  \fi
  and
  \ifproc
  \refstepcounter{epsilon}\label{eps:dh:v}
  \begin{multline*}
    \sigma_d^2 = \frac{1}{4} + \frac{-t^2+9t-14}{2^{t+4}}
    \\+ \frac{-4t^3+20t^2+3t-54}{2^{2t+6}}
    + \frac{0.056t^4}{2^{3t}} \varepsilon_{{\ref{eps:dh:v}}}(t)
  \end{multline*}
  \else
  \begin{equation*}
    \sigma_d^2 = \frac{1}{4} + \frac{-t^2+9t-14}{2^{t+4}}
    + \frac{-4t^3+20t^2+3t-54}{2^{2t+6}}
    + \frac{0.056t^4}{2^{3t}} \neweps(t)\label{eps:dh:v}
  \end{equation*}
  \fi
  for $t\geq2$. 
\end{theorem}


Again, as in the previous section, we calculated the values of the constants
$\mu_d$ and $\sigma_d^2$ numerically for $2\leq t\leq 30$, and they are given
in Table~\ref{tab:special-values-distinct-depths}. 
Figure~\ref{fig:depths} visualizes the result of
Theorem~\ref{thm:distinct-depths} as in the previous section.

\tabledepths{}


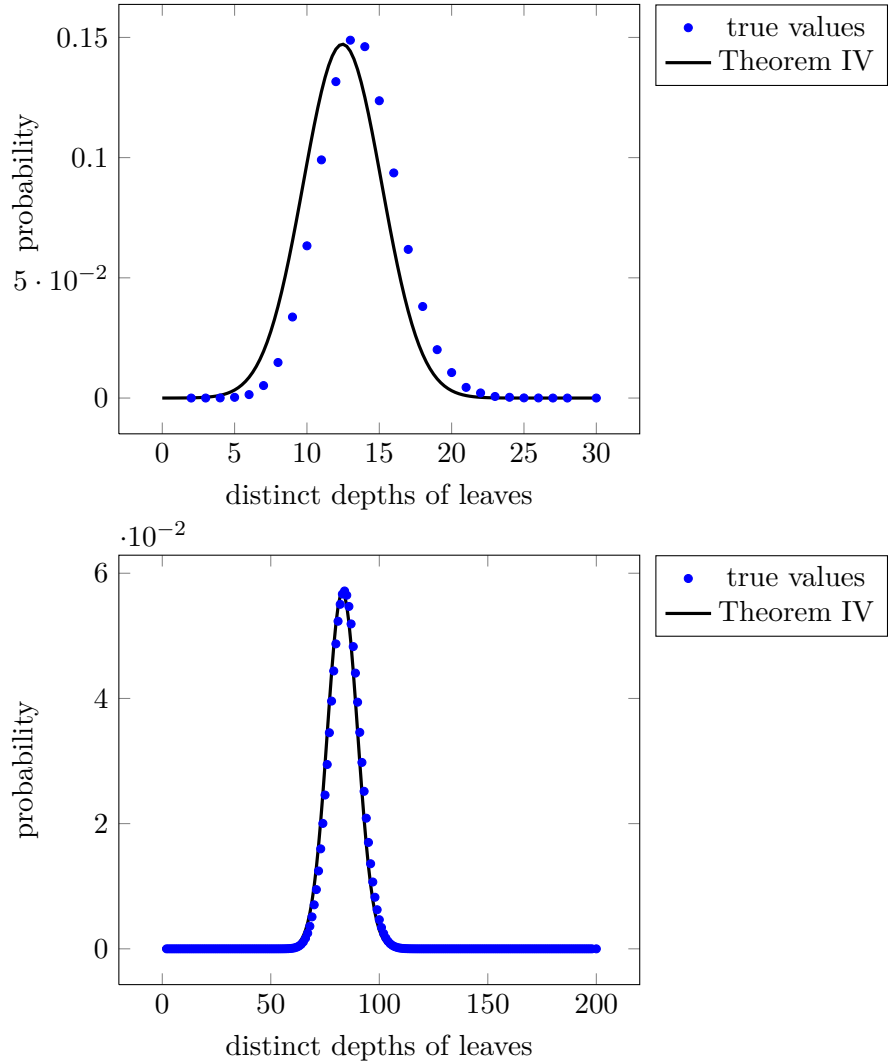
\begin{figure}
  \centering
   \begin{tikzpicture}
    \begin{axis}[
      xlabel=distinct depths of leaves,
      ylabel=\hspace*{2.5em}probability,
      legend pos=outer north east]
    \addplot[color=blue, mark=*, only marks, mark size=1.5]
         table[x=depths30, y=count] {plot_depths30_true.txt};
    \addplot[color=black, no marks, very thick]
         table[x=depths30, y=count] {plot_depths30_asy.txt};
    \legend{true values, Theorem~\ref{thm:distinct-depths}}
    \end{axis}
  \end{tikzpicture}
   \begin{tikzpicture}
    \begin{axis}[
      xlabel=distinct depths of leaves,
      ylabel=probability,
      legend pos=outer north east]
    \addplot[color=blue, mark=*, only marks, mark size=1.5]
         table[x=depths200, y=count] {plot_depths200_true.txt};
    \addplot[color=black, no marks, very thick]
         table[x=depths200, y=count] {plot_depths200_asy.txt};
    \legend{true values, Theorem~\ref{thm:distinct-depths}}
    \end{axis}
  \end{tikzpicture}
 
  \caption{Distribution of the distinct depths of leaves for $t=2$, and $n=30$
    (top figure) and $n=200$ (bottom figure)
    inner vertices. On the one hand, this figure shows the true distribution of
    all trees of the given size and on the other hand the result on the
    asymptotic normal distribution (Theorem~\ref{thm:distinct-depths} with only
    main terms of mean and variance taken into account).}
  \label{fig:depths}
\end{figure}


As mentioned above, the proof of Theorem~\ref{thm:distinct-depths} works
analogously to the proof of Theorem~\ref{thm:height}. It is again spread over
several lemmata. There is a one-to-one correspondence of
Lemmata~\ref{lem:d:outer-region} to~\ref{lem:d:qpp-positive} to
Lemmata~\ref{lem:h:outer-region} to~\ref{lem:h:qpp-positive} in
the section for the height parameter. Due to their similarities, the proofs are
skipped a couple of times and only some differences (for example, the different
constants) are mentioned. The idea of the proof itself is described in the
previous section below Theorem~\ref{thm:height}.


To show Theorem~\ref{thm:distinct-depths}, it is convenient to work with the
finite sum
\begin{equation*}
  D_J(q,v) \colonequals 
  1 - \sum_{1\leq j<J} \frac{v q^{\hp{j}}}{1-q^{\hp{j}}} 
  \prod_{i=1}^{j-1} \frac{1-v-q^{\hp{i}}}{1-q^{\hp{i}}}
\end{equation*}
instead of the denominator $D(q,v) = 1 - \f{b}{q,1,u,1}$. The error made by this approximation
was analyzed at the end of Section~\ref{sec:generating-function}, namely in the
two lemmata~\ref{lem:d:den-bounds-pre} and~\ref{lem:d:den-bounds}.


For the local limit theorem, we split up into the central region around $v=1$
and an outer region. The following lemma covers the latter one.


\begin{lemma}\label{lem:d:outer-region}
  Let $v = e^{i\varphi}$, where $\varphi$ is real with $2\pi\, 2^{-t/2} <
  \abs{\varphi} \le \pi$. Then each zero of $z\mapsto\f{D}{1/z,v}$ has absolute
  value smaller than $2-1/2^t$.
\end{lemma}


The proof goes along the same lines as the proof of
Lemma~\ref{lem:h:outer-region}, but we get the bound
\begin{equation*}
  \abs{z_0} \leq \abs{1+w} + \frac{7}{2^t}
\end{equation*}
instead of \eqref{eq:h:outer-bound-z0}.


Next, we go on to the central region. As a first step, we bound the location of
the dominant zero.


\begin{lemma}\label{lem:d:one-root-rouche}
  Suppose $t\geq4$ and $\abs{v-1} \leq \frac12 -
  5\bigl(\frac23\bigr)^t$, then $q\mapsto\f{D}{q,v}$ has exactly one
  root with $\abs{q}<\frac23$ and no root with $\abs{q}=\frac23$.
\end{lemma}


This lemma is proven analogously to Lemma~\ref{lem:h:one-root-rouche}. The only
difference is the bound
\begin{equation*}
  \abs{\f{D}{q,v} - \f{D_2}{q,v}} \leq 3.09 \left(\frac23\right)^t = b, 
\end{equation*}
which is valid for $t\geq4$.


\begin{lemma}\label{lem:d:qw-analytic}
  For $t\geq4$ and $\abs{v-1} \leq \frac12 - 5\bigl(\frac23\bigr)^t$, the
  function $\f{q_0}{v}$ given implicitly by $\f{D}{\f{q_0}{v}, v} = 0$,
  $\abs{\f{q_0}{v}}<\frac23$, is analytic.
\end{lemma}


The proof of this analyticity result is the same as the one from
Lemma~\ref{lem:h:qw-analytic}, therefore skipped here. 

In the central region around $v=1$, small changes in $v$ do not change the
location of the dominant zero much, which is made explicit in the lemma below.


\begin{lemma}\label{lem:d:qpp-bound-near-1}
  Let $t\geq 30$ and $v=e^{i\varphi}$, where $\varphi\in\R$ with $\abs{\varphi}
  \leq 2\pi\, 2^{-t/2}$, then
  \begin{equation*}
    \abs{\f{q_0}{v} - \f{q_0}{1}} \leq \frac{9}{2^{t/2}},
    \quad
    \abs{\f{q_0'}{v} - \f{q_0'}{1}} \leq \frac{34}{2^{t/2}}
    \quad\text{and}\quad
    \abs{\f{q_0''}{v} - \f{q_0''}{1}} \leq \frac{202}{2^{t/2}}.
  \end{equation*}
\end{lemma}


Again, the proof works analogously to the proof of the corresponding lemma for
the height parameter. 

In order to prove the local limit theorem we show that the second derivative of
$\abs{\f{q_0}{e^{i\varphi}}}$ is positive. This is stated in the following
lemma.


\begin{lemma}\label{lem:d:qpp-positive}
  If $t\geq 30$ and $\varphi\in\R$ with $\abs{\varphi} \leq 2\pi\, 2^{-t/2}$,
  then
  \begin{equation*}
    \frac{\dd^2}{\dd \varphi^2} \abs{\f{q_0}{e^{i\varphi}}}^2 > 0.
  \end{equation*}
\end{lemma}


We use the proof of Lemma~\ref{lem:h:qpp-positive} and update the constants.


For a fixed $t$ we can use the mathematics software system
SageMath~\cite{Stein-others:2015:sage-mathem-6.5} and perform calculations with
interval arithmetic. The details, which are stated for the height in
Remark~\ref{rem:h:small-t}, remain valid. For integers~$t$ fulfilling $2\leq
t\leq 30$ we showed that $\abs{\f{q_0}{e^{i\varphi}}}$ has a unique minimum at
$\varphi=0$.

The proof of Theorem~\ref{thm:distinct-depths} follows by the same
arguments as the proof of Theorem~\ref{thm:height}: We use Theorem IX.9
of Flajolet and Sedgewick~\cite{Flajolet-Sedgewick:ta:analy} applied to the
function $H(q,1,v,1)$ to get mean and variance (and asymptotic normality as a
central limit, too). For the local limit theorem the uniqueness of the minimum
of $\abs{\f{q_0}{e^{i\varphi}}}$ is shown by a two-fold strategy. The central
region with $\abs{\varphi} \leq \sqrt{3}\,\pi\, 2^{-t/2}$ is covered by
Lemma~\ref{lem:d:qpp-positive} (using previous lemmata as
prerequisites). Lemma~\ref{lem:d:outer-region} discusses the outer region. For
$t<30$ the algorithmic approach above is used.


\section{The Width}\label{section:width}


In this section, we consider the width, i.e., the maximum number of
leaves on the same level, for which we have the following theorem.


\begin{theorem}\label{theorem:width}
  For a randomly chosen tree $T\in\calT$ of size $n$, we have
  \begin{equation*}
    \E(w(T))=\mu_w\log n+O(\log\log n)
  \end{equation*}
  for the expectation of the width $w(T)$, where $\mu_w$ is given by
  \begin{equation*}
    \mu_w=\frac1{t-1}\left(\frac1{\log 2}+\frac1{4\cdot 2^t\log^2 2}+\frac{0.2t}{4^t}\neweps(t)\right)
  \end{equation*}
  for $t\ge 10$. For $2\le t\le 9$, the values of $\mu_w$ are given in
  Table~\ref{tab:mu_w_values}.

  Furthermore, we have the concentration property
  \begin{equation}\label{eq:concentration-property-width}
    \P(\abs{w(T)-\mu_w \log n}\ge \sigma \mu_w\log\log n)=O\!\left(\frac1{\log^{\sigma-2} n}\right)
  \end{equation}
  for $\sigma>2$.
\end{theorem}


\begin{table}
  \centering{\small
  \begin{equation*}
  \begin{array}{r|l}
\multicolumn{1}{c|}{t} &
\multicolumn{1}{c}{\mu_w} \\
    \hline
    2 & 1.710776751014961\\
    3 & 0.7660531443158307\\
    4 & 0.4936068552417457\\
    5 & 0.3650919029615249\\
    6 & 0.2902388863790219\\
    7 & 0.2411430286905858\\
    8 & 0.2063933963643483\\
    9 & 0.1804647899046739\\
    10 & 0.1603561167643597
  \end{array}
  \end{equation*}}
  \caption{Numerical values of the constants $\mu_w$ for $2\le t\le 10$, cf.\
    Theorem~\ref{theorem:width}.
  See also Remark~\ref{rem:numerical-calc}.
  For the accuracy of these numerical results see the note at the end of
  the introduction.}
  \label{tab:mu_w_values}
\end{table}


In Figure~\ref{fig:width} one can find the distribution of the leaf-width for a
given parameter set together with the mean found in
Theorem~\ref{theorem:width}. 


\begin{figure}
  \centering
   \begin{tikzpicture}
    \begin{axis}[
      xlabel=leaf-width,
      ylabel=probability,
      legend pos=outer north east]
    \addplot[color=blue, mark=*, only marks, mark size=1.5]
         table[x=width, y=count] {plot_width_true.txt};
    \addplot[color=black, no marks, very thick, dashed]
         table[x=width, y=count] {plot_width_mean.txt};
    \legend{true values, expectation}
    \end{axis}
  \end{tikzpicture}
 
  \caption{Distribution of the leaf-width for $t=2$ and $n=100$ inner
    vertices. On the one hand, this figure shows the true distribution of all
    trees of the given size and on the other hand the result on the expectation
    of this distribution (Theorem~\ref{theorem:width} with only main term of
    mean taken into account).}
  \label{fig:width}
\end{figure}
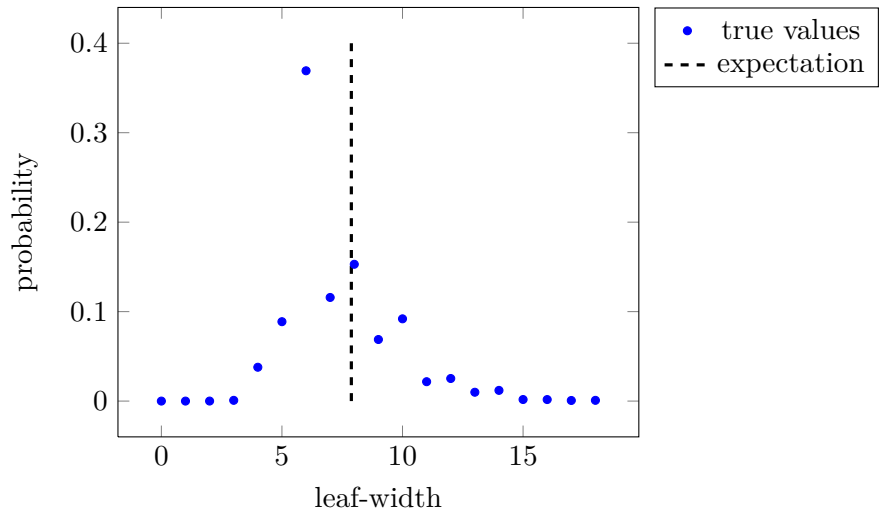


First, we sketch the idea of the proof. We consider trees whose width is
bounded by $K$. The corresponding generating function $W_K(q)$ can be
constructed by a suitable transfer matrix, and we quantify the obvious
convergence of $W_K(q)$ to $H(q,1,1,1)$. The dominant singularity $q_K$ of
$W_K(q)$ is estimated by truncating the infinite positive eigenvector of an
infinite transfer matrix corresponding to $H(q,1,1,1)$ and applying methods
from Perron--Frobenius theory. Then the probability $\P(w(T)\le K)$ can be
extracted from $W_K(q)$ using singularity analysis. Our key estimate states
that the singularity $q_K$ converges exponentially to $q_0$, from which the
main term of the expectation as well as the concentration property are obtained
quite easily. A more precise result on the distribution of the width would
depend on a better understanding of the behaviour of $q_K$ as $K \to \infty$,
which seems to be quite complicated.


The proof of the
theorem depends on the following definitions. Apart from the width $w(T)$,
we also need the ``inner width'' $w^*(T)$ defined to be
\begin{equation*}
  w^*(T) \colonequals \max_{0\le k < h(T)} L_T(k)
\end{equation*}
for a recursive construction. Here, $L_T(k)$ denotes the number of leaves at
level $k$. By definition, the inner width $w^*(T)$ does not take the leaves on
the last level into account.

For $K >  0$, we are interested in the generating function
\begin{equation*}
  W_K(q) \colonequals \sum_{\substack{T\in\calT\\ w(T)\le K}} q^{n(T)}.
\end{equation*}
We represent $W_K(q)$ in terms of the generating functions
\begin{equation*}
  W_{K,r}(q) \colonequals \sum_{\substack{T\in\calT\\ w^*(T)\le K\\m(T)=tr}}q^{n(T)}
\end{equation*}
for $r\ge 0$ so that
\begin{equation*}
  W_K(q)=1+\sum_{r=1}^{\lfloor K/t\rfloor} W_{K,r}(q).
\end{equation*}
Here, the summand $1$ corresponds to the tree of order $1$. For all other
trees, the number $m(T)$ of leaves on the last level is clearly a multiple of
$t$.

Next we set up a recursion for $W_{K,r}$, $1\le r\le N(K)$, where $N(K) \colonequals \lceil K/(t-1)\rceil-1$.
Let us define the column
vector
\begin{equation*}
  \bfW_K(q) \colonequals (W_{K,1}(q),\ldots, W_{K,N(K)})^T,
\end{equation*}
and the ``transfer matrix'' 
\begin{equation*}
  M_K(q) \colonequals \left(q^r\left[  \frac rt \le s \le
       \frac{r+K}{t} \right] \right)_{\substack{1\le
      r\le N(K)\\1\le s\le N(K)}},
\end{equation*}
where the Iversonian notation\footnote{Keep in mind that we also use square
  brackets for extracting coefficients: $[q^n]Q(q)$ gives the $n$-th
  coefficient of the power series $Q$.}
\begin{equation*}
  [\mathit{expr}]=
  \begin{cases}
    1&\text{ if $\mathit{expr}$ is true},\\
    0&\text{ if $\mathit{expr}$ is false}
  \end{cases}
\end{equation*}
popularised by Graham, Knuth and Patashnik~\cite{Graham-Knuth-Patashnik:1994} has been used.

We now express $\bfW_K(q)$ in terms of $M_K(q)$:

\begin{lemma}
  For $K\ge t$, we have
  \begin{equation}\label{eq:width-functional-equation-K}
    \bfW_K(q)=(I-M_K(q))^{-1}\begin{pmatrix}q\\0\\\vdots\\0\end{pmatrix}.
  \end{equation}
\end{lemma}
\begin{proof}
  As in the proof of
  Theorem~\ref{theorem:generating-function-height-and-others}, a tree $T'$ of
  height $h+1\ge 2$, inner width at most $K$ and $m(T')=rt$ arises from a tree
  $T$ of height $h$, inner width at most $K$ and $m(T)=st$ by replacing $r$ of
  the $st$ leaves of $T$ on the last level by internal vertices with $t$
  succeeding leaves each. We obviously have $r\le st$. In order to ensure that
  $w^*(T')\le K$, we have to ensure that $st-r\le K$. We rewrite these two
  inequalities as
  \begin{equation}\label{eq:width-recursion-inequality}
     \frac rt \le s \le \frac{r+K}{t}.
  \end{equation}
  If $r\le N(K)$, we have $r < K/(t-1)$ and therefore $s < K/(t-1)$ by
  \eqref{eq:width-recursion-inequality}, i.e., $s\le N(K)$. This justifies our
  choice of $N(K)$. The construction above yields $s$ new internal vertices
  in~$T'$. There is only one tree $T'$ of height $< 2$, namely the star of order $t+1$, which has one internal vertex
  (the root). In this case, $r=1$.

  Translating these considerations into the language of generating functions
  yields
  \begin{equation*}
    W_{K,r}(q)= q [r=1]+\sum_{s=1}^{N(K)}q^r\left[  \frac rt \le s \le
       \frac{r+K}{t} \right] W_{K,s}(q).
  \end{equation*}
  Rewriting this in vector form yields \eqref{eq:width-functional-equation-K}.
\end{proof}

We will obtain asymptotic expressions for the coefficients of $\bfW_K$ by
singularity analysis. To this end, we have to find the singularities of
$(I-M_K(q))^{-1}$ as a meromorphic function in $q$.  In order to do so, we have
to consider the zeros of the determinant $\det(I-M_K(q))$. Note that $q_K$ is a zero of
$\det(I-M_K(q))$ if and only if $1$ is an eigenvalue of $M_K(q_K)$. In the next
lemma, we collect a few results connecting $M_K(q)$ with Perron--Frobenius
theory.

\begin{lemma}\label{lemma:Perron-Frobenius-implications}
  Let $K\ge t$ and $q > 0$. Then
  \begin{enumerate}
  \item the matrix $M_K(q)$ is a non-negative, irreducible,
    primitive matrix;
  \item the function $q\mapsto \lambdamax(M_K(q))$ mapping $q$ to
    the spectral radius of $M_K(q)$ is a strictly increasing function from
    $(0,\infty)$ to $(0,\infty)$;
  \item if $M_K(q)x\le x$ or $M_K(q)x\ge x$ holds componentwise for some positive vector $x$,
    then $\lambdamax(M_K(q))\le 1$ or $\lambdamax(M_K(q))\ge 1$, respectively.
  \end{enumerate}
\end{lemma}
\begin{proof}
  We prove each statement separately.
  \begin{enumerate}
  \item The matrix $M_K(q)$ is non-negative by definition. We note that
    $\frac rt \le r-1$ holds for all $r\ge 2$ and $r+1\le \frac{r+K}t$ holds
    for all $r < N(K)$. This implies that all subdiagonal, diagonal and
    superdiagonal elements of $M_K(q)$ are positive. Thus $M_K(q)$ is
    irreducible. As all diagonal elements are positive, it is also primitive.
  \item This is an immediate consequence of \cite[Theorem~8.8.1(b)]{Godsil-Royle:2001:alggraphtheory}.
  \item Assume that $M_K(q)x\le x$ for some positive $x$. Let $y^T>0$ be a left
    eigenvector of $M_K(q)$ to the eigenvalue $\rho(M_K(q))$. Then
    \begin{equation*}
      \rho(M_K(q))y^Tx=y^T M_K(q)x \le y^Tx.
    \end{equation*}
    The result follows upon division by $y^Tx>0$. The case $M_K(q)x\ge x$ is
    analogous.  \qedhere
  \end{enumerate}
\end{proof}

We consider the infinite matrix
\begin{equation*}
  M_{\infty}(q) \colonequals \left(q^r\left[  \frac rt \le s \right] \right)_{\substack{1\le
      r\\1\le s}}
\end{equation*}
and the infinite determinant $\det (I-M_\infty(q))$ which is defined to be the
limit of the principal minors $\det ([r=s]-q^r\left[  \frac rt \le s
\right])_{\substack{1\le r\le N\\1\le s\le N}}$ when $N$ tends to $\infty$,
cf.\ Eaves~\cite{Eaves:1970}. For $\abs{q} < 1$, this infinite determinant
converges by Eaves' sufficient condition.

We now show that the infinite determinant is indeed the denominator of the
generating function $H(q,1,1,1)$.

\begin{lemma}\label{lemma:infinite-determinant}
   We have 
   \begin{equation*}
     \det(I-M_{\infty}(q))=1-b(q,1,1,1)
   \end{equation*}
   where $b(q,1,1,1)$ is given in Lemma~\ref{lemma:b-q-u-1-1-formula}.
\end{lemma}

\begin{proof}
  When expanding the infinite determinant, we take the $1$ on the diagonal in
  almost all rows and some other entry in rows $a_1 < a_2 < \cdots < a_k$ for
  some $k$. These other entries have to come from $-M_\infty(q)$. Extracting
  the sign for these rows yields
  \begin{align*}
    \det(I-M_{\infty}(q))
    &=\sum_{k\ge 0}(-1)^k\sum_{1\le a_1 < a_2 < \cdots < a_k} \det( q^{a_i}[a_i\le ta_j])_{1\le i,j\le k}\\
    &=\sum_{k\ge 0}(-1)^k\sum_{1\le a_1 < a_2 < \cdots < a_k} q^{a_1+\cdots +a_k}\det([a_i\le ta_j])_{1\le
    i,j\le k} .
  \end{align*}
  We trivially have $a_i\le ta_j$ for $j\ge i$, so all entries on the diagonal
  of $([a_i\le ta_j])_{1\le i,j\le k}$ and above this diagonal are $1$. If
  $a_2\le ta_1$, the first and the second row of $([a_i\le ta_j])_{1\le i,j\le
    k}$ are identical, so the determinant vanishes. Therefore, we only have to
  consider summands with $a_2 > ta_1$. In this case, we clearly have $a_i >
  ta_1$ for all $i\ge 2$, i.e., the first column of $([a_i\le ta_j])_{1\le
    i,j\le k}$ is $(1,0,\ldots,0)^T$. Repeating this argument, we see that only
  summands with $a_{j+1} > t a_j$ for $1\le j < k$ contribute to the
  determinant. For those summands, the matrix $([a_i\le ta_j])_{1\le i,j\le k}$
  equals $([j\ge i])_{1\le i,j\le k}$ and thus has determinant $1$.

  Therefore, we obtain the representation
  \begin{equation*}
    \det(I-M_{\infty}(q))=\sum_{k\ge 0}(-1)^k \sum_{\substack{a_1,\ldots, a_k\\
      \forall j\colon a_{j+1} > ta_j}}q^{a_1+\cdots +a_k}.
  \end{equation*}
  With the change of variables $a_1 \equalscolon b_k$ and $a_{j+1}-ta_{j}
  \equalscolon b_{k-j}$ for $1\le j < k$, we obtain
  \begin{align*}
    \det(I-M_{\infty}(q))&=\sum_{k\ge 0}(-1)^k \sum_{b_1,\ldots,
        b_k\ge 1}q^{b_1\hp{1}+\cdots +b_k\hp{k}}\\
      &=\sum_{k\ge 0}(-1)^k\prod_{j=1}^k
      \Bigl(\sum_{b_j\ge 1} (q^{\hp{j}})^{b_j}\Bigr)=1-b(q,1,1,1).
  \end{align*}
\end{proof}

If $K$ tends to infinity, $W_K(q)$ tends to $H(q,1,1,1)$, as the restriction on
the width becomes meaningless. For our purposes, we will need a slightly
stronger result: we also need convergence of the numerator and the denominator
of $W_K(q)$ given by \eqref{eq:width-functional-equation-K} and Cramer's rule
to the numerator $a(q,1,1,1)$ and the denominator $1-b(q,1,1,1)$ of
$H(q,1,1,1)$, respectively.  We prove this in two steps. The first one is to
prove that the numerator and the denominator of $W_K(q)$ tend to the
corresponding infinite determinants. This is stated in the following lemma.

\begin{lemma}\label{lemma:convergence-determinant}
  For $\abs{q} \le 0.6$, we have
  \begin{equation*}
    \det(I-M_K(q))=\det(I-M_{\infty}(q))+O(q^{K/(2t)}).
  \end{equation*}
  The same conclusion holds when the $s$-th column of both $I-M_K(q)$ and
  $I-M_{\infty}(q)$ are replaced by the vector $(q,0,\ldots)^T$ with $K-1$ and
  infinitely many zeroes, respectively. The estimate still holds for the first derivatives with respect to $q$.
\end{lemma}
\begin{proof}
  The infinite determinant $\det(I-M_\infty(q))$ consists of summands
  \begin{equation*}
    \pm \prod_{s\in S}q^{\pi(s)}=\pm q^{\sum_{s\in S}\pi(s)} = \pm q^{\sum_{s\in S} s}
  \end{equation*}
  where $\pi\colon\N\to\N$ is a bijection such that there are only finitely many
  non-fixed points of $\pi$ and $S$ is a finite subset of $\N$ containing
  all non-fixed points of $\pi$. Note that the complement of $S$ corresponds to
  those columns where $1$ has been chosen on the diagonal in the expansion of
  the determinant. Not all $(\pi,S)$ will actually occur due to the Iversonian
  expression in the definition of $M_\infty(q)$.

  For every $k\in\N$, there is a bijection from the set
  \begin{multline*}
    \Bigl\{(\pi,S) \Bigm| \pi\colon\N\to\N\text{ bijective, } S\subseteq\N \text{ finite
      such that }\{s\in\N\mid \pi(s)\neq s\}\subseteq S \\\text{ and }\sum_{s\in S}\pi(s)=k \Bigr\} 
  \end{multline*}
  to the set 
  \begin{equation*}
    \Bigl\{ (x_1,\ldots, x_j)\in\N^j \Bigm| j\in\N,\ \sum_{i=1}^j x_i=k \text{ with
      pairwise distinct }x_i \Bigr\}
  \end{equation*}
  of compositions of $k$ with distinct parts: the set $S$ can be recovered as
  the set of summands in the composition, the permutation $\pi$ can be
  recovered from the order of the summands.

  As there are at most $\exp(2\sqrt{k}\log k)$ compositions of $k$ with
  distinct parts by a result of Richmond and
  Knopfmacher~\cite{Richmond-Knopfmacher:1995:compos}, there are at most that
  many summands $\pm q^k$ in the infinite determinant $\det(I-M_\infty(q))$.

  The difference between $\det(I-M_\infty(q))$ and $\det(I-M_K(q))$ consists of
  those summands which do not choose the $1$ on the diagonal in some row $>
  N(K)$ or which choose an entry in some column $s$ and in some row $r$ with $s
  > (r+K)/t$. In the latter case, the $1$ on the diagonal cannot be chosen in
  row $s$, so that the exponent of $q$ in this summand is at least $r+s >
  K/t$. So all summands in the difference are of the form $\pm q^k$ for some $k
  \ge K/t$. By the triangle inequality and the above estimates, we obtain
  \begin{equation*}
    \abs{\det(I-M_\infty(q)) - \det(I-M_K(q))} \le \sum_{k\ge K/t}\exp(2\sqrt k
    \log k)q^k=O(q^{K/(2t)}).
  \end{equation*}
  The argument does not change if the $s$-th column of both matrices is replaced
  by the column vector $(q,0,\ldots,0)^T$.

  Differentiating the determinant can be done term by term. The error term does
  not change as the bound $O(q^{K/(2t)})$ is weak enough.
\end{proof}

The second step in the proof of the convergence of the numerator and the denominator of
$W_K(q)$ consists of the following simple lemma.

\begin{lemma}\label{lemma:denominator-numerator-convergence}
  Let $\abs{q}\le 0.6$. Then the denominator $\det(I-M_K(q))$ of $W_K(q)$ converges to $1-b(q,1,1,1)$ with
  error $O(q^{K/(2t)})$. The numerator $\det(I-M_K(q))W_K(q)$ of $W_K(q)$
  converges to $a(q,1,1,1)$ with the same error. The same is true for
  the first derivatives with respect to $q$.
\end{lemma}

\begin{proof}
  The first statement is simply the combination of
  Lemmata~\ref{lemma:convergence-determinant} and
  \ref{lemma:infinite-determinant}. 

  As a formal power series, $W_K(q)$ converges to $H(q,1,1,1)$ as
  $[q^n]W_K(q)=[q^n]H(q,1,1,1)$ holds for $n\le (K-1)/(t-1)$, because a
  canonical tree with $n$ internal vertices has $1+n(t-1)$ leaves and therefore
  width at most $1+n(t-1)$.

  As $1-b(q,1,1,1)$ has no root with $\abs{q}< 1/2$ by Lemma~\ref{lemma:dominant-singularity}, $W_K(q)$ converges to
  $H(q,1,1,1)$ for $\abs{q} < 1/2$. As the denominator is already known to
  converge to the denominator $1-b(q,1,1,1)$ of $H(q,1,1,1)$, we conclude that
  the numerators (which are already known to converge to some infinite
  determinant) actually have to converge to $a(q,1,1,1)$.

  Taking derivatives with respect to $q$ does not change the argument by
  Lemma~\ref{lemma:convergence-determinant}.
\end{proof}

In order to obtain information on the roots of $\det(I-M_K(q))$ and therefore
the singularities of $\bfW_K(q)$, we approximate the Perron--Frobenius
eigenvector of $M_K(q)$ by the one of the infinite matrix $M_\infty(q)$. The following lemma gives this eigenvector explicitly---as we will see in the next section, it has a natural combinatorial interpretation.

\begin{lemma}
  For $r\ge 1$, we have
  \begin{equation}\label{eq:infinite-eigenvector-preparation}
    q^r\Bigl(1-\sum_{j=1}^{\lceil r/t\rceil-1} [u^{jt}]b(q,u,1,1)\Bigr)=[u^{rt}]b(q,u,1,1).
  \end{equation}
  In particular, if we set $p_r = [u^{rt}]b(q_0,u,1,1)$, then $(p_r)_{r\ge 1}$ is a right eigenvector of $M_\infty(q_0)$ to the eigenvalue $1$, i.e.,
  \begin{equation}\label{eq:eigenvector-infinite-equation}
    M_\infty(q_0)\cdot (p_r)_{r\ge 1}=(p_r)_{r\ge 1}.
  \end{equation}
\end{lemma}


  \newlength{\wwwoolen}
  \newcommand{\wwwool}[2]{\settowidth{\wwwoolen}{$#1$}{#1}\hspace*{-\wwwoolen}\hphantom{#2}}
  \newcommand{\wwwoor}[2]{\settowidth{\wwwoolen}{$#1$}\hphantom{#2}\hspace*{-\wwwoolen}{#1}}


\begin{proof}
  Multiplying the left hand side of \eqref{eq:infinite-eigenvector-preparation}
  with $u^{rt}$ and summing over $r\ge 1$ yields
  \begin{align*}
    \frac{qu^t}{1-qu^t}-\sum_{\substack{r\ge 1\\\wwwoor{j}{r}\ge 1\\ \wwwoor{jt}{r} < \wwwool{r}{1}}}(qu^t)^r[u^{jt}]b(q,u,1,1)
    &=\frac{qu^t}{1-qu^t}-\sum_{j=1}^{\infty}[u^{jt}]b(q,u,1,1)\sum_{r=jt+1}^{\infty}(qu^t)^r\\
    &=\frac{qu^t}{1-qu^t}-\frac{qu^t}{1-qu^t}\sum_{j=1}^{\infty}(qu^t)^{jt}[u^{jt}]b(q,u,1,1)\\
    &=\frac{qu^t}{1-qu^t}(1-b(q,qu^t,1,1))=b(q,u,1,1),
  \end{align*}
  where the last equality comes from Lemma~\ref{lemma:b-q-u-1-1-formula}.
  This concludes the proof of \eqref{eq:infinite-eigenvector-preparation}.

  Setting $q=q_0$ in \eqref{eq:infinite-eigenvector-preparation} and noting that $1=b(q_0,1,1,1)=\sum_{r\ge 1}p_r$ yields 
  \eqref{eq:eigenvector-infinite-equation}.
\end{proof}

We now use the fact that $(p_r)_{r\ge 1}$ is an eigenvector of $M_\infty(q)$ to derive bounds for its entries.

\begin{proposition}\label{proposition:p_r-bounds}
  All constants $p_r$, $r\ge 1$, are positive, and we have $p_r = \Omega(q_*^r/r)$ and $p_r = O(r^2 q_*^r)$, where
  \begin{equation*}
    q_*=q_0^{1+\frac{1}{t-1}}.
  \end{equation*}
\end{proposition}
\begin{proof}
  As we will see later in the proof of Theorem~\ref{theorem:distribution-m}, equation~\eqref{eq:probability_convergence}, the $p_r$ are limits of
  probabilities and therefore a priori non-negative. In fact, this is a consequence of Lemma~\ref{lemma:number-canonical-forests}. Moreover, they sum to $1$ as mentioned earlier, and in view of the eigenvalue equation and the fact that $M_{\infty}(q)$ is an irreducible matrix, we even know that they must be strictly positive.

  By the eigenvalue equation \eqref{eq:eigenvector-infinite-equation}, we have
  \begin{equation*}
    p_r\ge q_0^r p_{\lceil r/t \rceil}
  \end{equation*}
  for all $r\ge 1$. Iterating this yields, with $p_{\mathrm{min}} = \min_{s < t} p_s$,
  \begin{align*}
    p_r&\ge q_0^{\sum_{j=0}^{\lfloor\log_t r\rfloor-1} \lceil r/t^j \rceil} p_{\lceil r/t^{\lfloor \log_t r\rfloor}\rceil}
    \geq p_{\mathrm{min}} q_0^{\sum_{j=0}^{\lfloor\log_t r\rfloor-1}  (1+r/t^j) }\\&\ge p_{\mathrm{min}}q_0^{\log_t r+\sum_{j=0}^{\infty}  r/t^j }=p_{\mathrm{min}} r^{\log_t q_0}q_0^{r(1+1/(t-1))}.
  \end{align*}
  As $q_0\ge 1/t$ by Lemma~\ref{lemma:dominant-singularity}, we have $\log_t q_0\ge -1$ and the lower bound follows.

  To prove the upper bound, we proceed in two steps. In a first step, we note that the eigenvalue equation \eqref{eq:eigenvector-infinite-equation} together with the fact that $\sum_{r\ge 1}p_r=1$ yields the weaker upper bound
  \begin{equation*}
    p_r= q_0^r\sum_{s\ge \lceil r/t\rceil }p_s\le q_0^r \sum_{s\ge 1}p_s=q_0^r.
  \end{equation*}
  In a second step, we use induction on $r$ and assume that $p_s \le  c s^2 q_*^{s}$ for $s <  r$ for some constant $c$ depending on $t$.
  Then the eigenvalue equation \eqref{eq:eigenvector-infinite-equation} yields
  \begin{align*}
    p_r&\le q_0^r \sum_{s\ge \lceil r/t\rceil} p_s\le 
    cq_0^r \sum_{\lceil r/t\rceil\le s < r} s^2 q_*^{s}+q_0^r \sum_{r\le s}q_0^s
    \le cq_0^r \sum_{\lceil r/t\rceil\le s} s^2 q_*^{s}+\frac1{1-q_0}q_0^{2r}\\
    &=
    c q_0^r \left( \frac{\lceil r/t \rceil^2}{1-q_*}
      +\frac{2q_*\lceil r/t \rceil}{(1-q_*)^2}
      +\frac{q_*(1+q_*)}{(1-q_*)^3}
    \right)q_*^{\lceil r/t\rceil}+\frac{1}{1-q_0}q_0^{2r}\\
    &\le c q_0^r \left( \frac{(r+t)^2}{t^2(1-q_*)}
      +\frac{2q_*(r+t)}{t(1-q_*)^2} +\frac{q_*(1+q_*)}{(1-q_*)^3} \right)q_*^{
      r/t}+\frac{1}{1-q_0}q_0^{2r}.
  \end{align*}
  As $t^2(1-q_*)> 1$ for $t\ge 2$ (cf.\ Lemma~\ref{lemma:dominant-singularity}), we obtain
  \begin{equation*}
    p_r \le c r^2 q_0^{r}q_*^{r/t}=
    cr^2 q_0^{r\left(1+\frac 1 t \left(1+\frac1{t-1}\right)\right)}
    =c r^2 q_*^r
  \end{equation*}
  for sufficiently large $r$.
\end{proof}

\begin{lemma}\label{lemma:q_K-bounds}
  The generating function $W_K(q)$ has a unique singularity $q_K$ with
  $\abs{q_K}\le 0.6$ for $K\ge \newc\label{c_7}$, where $c_{\ref{c_7}}$ is a suitable positive
  constant depending on $t$. It is a simple pole and a zero of
  $\det(I-M_K(q))$. Furthermore
  \newinvisiblec\label{c_5}\newinvisiblec\label{c_6}
  \begin{equation*}
    q_0+c_{\ref{c_5}} \frac1K q_0^{K/(t-1)}\le q_K \le q_0+c_{\ref{c_6}} K^2 q_0^{K/(t-1)}
  \end{equation*}
  for suitable positive constants $c_{\ref{c_5}}$, $c_{\ref{c_6}}$ depending on $t$.
\end{lemma}
\begin{proof}
  In the following, \addtocounter{constant}{1}$c_{\arabic{constant}}$, \addtocounter{constant}{1}$c_{\arabic{constant}}$, \addtocounter{constant}{-2}\ldots{}
  denote suitable positive constants depending on $t$.

  As $H(q,1,1,1)$ has a unique pole $q$ with $\abs{q}\le 0.6$ by
  Lemma~\ref{lemma:dominant-singularity} and numerator and denominator of
  $W_K(q)$ tend to the numerator and denominator of $H(q,1,1,1)$ respectively
  by Lemma~\ref{lemma:denominator-numerator-convergence}, $W_K(q)$ also has a
  unique pole with $\abs{q}\le 0.6$ for sufficiently large $K$.

  We set $x_K=(p_1,\ldots,p_{N(K)})^T$. If we find a $q>0$ such that
  $M_K(q)x_K\ge x_K$, then Lemma~\ref{lemma:Perron-Frobenius-implications}
  implies that $\lambdamax(M_K(q))\ge 1$ and $q_K < q$.

  We therefore consider the $r$-th row of $M_K(q)x_K$ for some $1\le r\le N(K)$. We have
  \begin{align*}
    (M_K(q)x_K)_r&=q^{r}\sum_{\frac rt\le s \le \frac{r+K}t} p_s\ge
    q^{r}\sum_{\frac rt\le s <  \frac{r+K}t}
    p_s=q^r\left(\frac{p_r}{q_0^r}-\frac{p_{r+K}}{q_0^{r+K}}\right)\\
    &=
    p_r\left(\frac{q}{q_0}\right)^r  \left(1-\frac{p_{r+K}}{p_rq_0^K}\right)
  \end{align*}
  by the eigenvalue equation~\eqref{eq:eigenvector-infinite-equation}.
  By Proposition~\ref{proposition:p_r-bounds}, we have\newinvisiblec\label{c_1}
  \begin{equation*}
    \frac{p_{r+K}}{p_rq_0^K}\le c_{\ref{c_1}} r(r+K)^2
    \frac{q_*^{r+K}}{q_*^{r}q_0^K}=
    c_{\ref{c_1}}r(r+K)^2q_0^{K/(t-1)}\le \newc K^3 q_0^{K/(t-1)}.
  \end{equation*}
  Therefore, we have
  \begin{equation*}
    \sqrt[r]{1-\frac{p_{r+K}}{p_rq_0^K}}=\frac1{\left(1-\frac{p_{r+K}}{p_rq_0^K}\right)^{-1/r}}
    \ge \frac1{1+\frac{2p_{r+K}}{rp_rq_0^K}}\ge \frac1{1+\newc K^2 q_0^{K/(t-1)}}.
  \end{equation*}
  This means that for $q=q_0+\newc K^2 q_0^{K/(t-1)}$, we have $M_K(q)x_K\ge
  x_K$, as desired.

  The proof of the lower bound runs along the same lines.
\end{proof}

\begin{proof}[Proof of Theorem~\ref{theorem:width}]
  We choose $K$ large enough so that $W_K(q)$ has a unique
  singularity $q_K$ with $\abs{q_K}\le 0.6$ and such that
  $q_K/0.6<0.99$.
  By singularity analysis and Lemma~\ref{lemma:denominator-numerator-convergence} we have
  \begin{equation*}
    \P(w(T)\le K)=\frac{[q^n]W_K(q)}{[q^n]H(q,1,1,1)}=(1+O(0.6^{K/2t}))\left(\frac{q_K}{q_0}\right)^{-n-1}(1+O(0.99^n))
  \end{equation*}
  for $K\ge \newc$.

  We now estimate
  \begin{equation}\label{eq:mean-width-reformulation}
    \E(w(T))=\sum_{K\ge 0}(1-\P(w(T)\le K)).
  \end{equation}
  We use the abbreviation $S \colonequals 1/q_0^{t-1}>1$.

  First, we consider the summands of \eqref{eq:mean-width-reformulation} with
  $S^K\le n/\log^2n$. By Lemma~\ref{lemma:q_K-bounds}, we have \newinvisiblec\label{c_9}\newinvisiblec\label{c_10}
  \begin{equation*}
    \left(\frac{q_K}{q_0}\right)^n\ge\left( 1+c_{\ref{c_9}}\frac{1}{S^K\log_S
        n}\right)^n\ge \left(1+c_{\ref{c_10}}\frac{\log n}{n}\right)^n\ge
    c_{\ref{c_10}}\log n.
  \end{equation*}
  We conclude that these summands of \eqref{eq:mean-width-reformulation}
  contribute $\log_S n+O(\log\log n)$. Similar estimates imply that
  \begin{equation}\label{eq:concentration-property-lower-bound}
    \P(w(T)-\log_S n\le -\sigma \log_S\log n)=O\left(\frac{1}{\log^{\sigma-1} n}\right)
  \end{equation}
  for $\sigma>1$.

  Now, we consider the summands of \eqref{eq:mean-width-reformulation} with
  $n/\log^2n< S^K< n\log^3 n$. These are $O(\log\log n)$ summands with each
  trivially contributing at most $1$, so the total contribution is $O(\log\log
  n)$.

  Next, we consider the summands of \eqref{eq:mean-width-reformulation} with
  $n\log^3 n\leq S^K\le n^{4t\log S}$. We now have \newinvisiblec\label{c_11}
  \begin{equation*}
    \frac{q_K}{q_0}\le 1+c_{\ref{c_11}} \frac{\log^2 n}{S^K}\le
    1+c_{\ref{c_11}}\frac1{n\log n}
  \end{equation*}
  and therefore
  \begin{equation*}
    \P(w(T)\le K)\ge(1+O(n^{-\abs{\log_S 0.6}/(2t)}))\exp\left(-(n+1)\log\left(\frac{q_K}{q_0}\right)\right)\ge
    1-\newc\frac1{\log n}.
  \end{equation*}
  The total contribution of these summands is therefore $O(1)$. Similar estimates imply that
  \begin{equation}\label{eq:concentration-property-upper-bound}
    \P(w(T)-\log_S n\ge \sigma \log_S\log n)=O\left(\frac{1}{\log^{\sigma-2} n}\right)
  \end{equation}
  for $\sigma>2$.

  Next, we consider the summands of \eqref{eq:mean-width-reformulation} with
  $n^{4t\log S}< S^K\le S^{tn}$. This time, we have
  \begin{equation*}
    \frac{q_K}{q_0}\le 1+\newc \frac{n^2}{n^4}
  \end{equation*}
  and therefore
  \begin{equation*}
    \P(w(T)\le K)=(1+O(n^{-2\abs{\log 0.6}}))\exp\left(-(n+1)\log\left(\frac{q_K}{q_0}\right)\right)\ge
    1-\newc\frac1{n}.
  \end{equation*}
  The total contribution of these summands is therefore $O(1)$.

  Finally, we note that all summands with $K>tn$ vanish: any tree with $n$
  internal nodes has at most width $tn$.

  Collecting all terms, we obtain
  \begin{equation*}
    \E(w(T))=\log_S n+O(\log\log n)=\frac{\log n}{-(t-1)\log q_0}+O(\log\log n).
  \end{equation*}

  Combining \eqref{eq:concentration-property-lower-bound} and
  \eqref{eq:concentration-property-upper-bound} immediately yields the
  concentration property \eqref{eq:concentration-property-width}.
\end{proof}


\section{The Number of Leaves on the Last Level}
\label{sec:number-leaves-maximum-depth}


Analysing the parameter $m(T)$ counting the number of leaves of maximum depth
(labelled by the variable $u$ in the generating function $H(q,u,v,w)$) is the
topic of this section. Here, $T$ is a canonical forest in $\calF_r$ for
some number of roots~$r$. We note that for fixed $\abs{u}\le 1$, the dominant
simple pole $q_0$ of $H(q,1,1,1)$ is also the dominant singularity of
$H(q,u,1,1)$ and is still a simple pole. Therefore, $m(T)$ tends to a discrete
limiting distribution; we refer to Section~IX.2 of the book of Flajolet and
Sedgewick~\cite{Flajolet-Sedgewick:ta:analy}. Note that the number $m(T)$ is
divisible by $t$ unless $T$ has height $0$. The result presented in this
section is a very useful tool in proving the central limit theorem for the path
length in the following section.


\begin{theorem}\label{theorem:distribution-m}
  Let $q_0$, $Q$ and $U$ be as described in Lemma~\ref{lemma:dominant-singularity}
  and $q_*$ as defined in Proposition~\ref{proposition:p_r-bounds}.
  For $m\ge 1$ such that $mt\in\Z$, we set $p_m=[u^{mt}]b(q_0,u,1,1)$ as in Lemma~\ref{eq:infinite-eigenvector-preparation}.
  Then, for a randomly chosen forest $T\in\calF_r$ of size $n$, we have
  \begin{equation}\label{eq:probability-distribution-m}
    \P(m(T)=mt)=p_m+O(Q^nU^{mt})
    =p_m\biggl(1+O\biggl(Q^n m \biggl(\frac{U^t}{q_*}\biggr)^m\biggr)\biggr)
    =O(U^{mt})
  \end{equation}
  uniformly in $r$.

  Furthermore, we have $\E(m(T))=\mu_m+O(Q^n)$ and
  $\Var(m(T))=\sigma_m^2+O(Q^n)$ uniformly in~$r$. Here,
  \begin{align*}
    \mu_m&=2  t - \frac{t^{2} - t}{2^{t+1}} - \frac{t^{3} + 6  t^{2} - 5
      t}{2^{2t+3}} - \frac{3  t^{4} + 32  t^{3} + 61  t^{2} - 56  t}{2^{3t+8}}
    -\frac{t^5}{3\cdot 2^{4t+4}}
    +\frac{1.3t^4}{2^{4t}}\neweps(t)
    \intertext{and}
    \sigma_m^2&=
    2 t^{2}
    -\frac{t^{4}-3t^{2}}{2^{t+1}}
    -\frac{t^{5}+13t^{4}-3t^{3}-17t^{2}}{4^{2t+3}}\\
    &\alignpheqop{-}\frac{3t^{6}+59t^{5}+215t^{4}-89t^{3}-208t^{2}}{2^{3t+6}}
    +\frac{2t^{7}}{2^{4t}}\neweps(t)
  \end{align*}
  for $t\ge 4$. For $t\in\{2,3\}$, the values of $\mu_m$ and $\sigma_m^2$ are
  given in Table \ref{tab:mu_t-sigma_t-values}.
\end{theorem}

\begin{table}[htbp]
  \centering{\small
  \begin{equation*}
    \begin{array}{r|r|r}
      t&\multicolumn{1}{c|}{\mu_m}&\multicolumn{1}{c}{\sigma_m^2}\\
      \hline
2 & 3.3008907135661046 & 3.4340283494347781 \\
3 & 5.4223250580971105 & 10.9926467981432752 \\
4 & 7.5391743055684431 & 23.0048877906448059 \\
5 & 9.6531072700455410 & 39.9382006717564049 \\
6 & 11.7525465927985450 & 61.9509728363450114 \\
7 & 13.8311837210749625 & 88.8290211521323761 \\
8 & 15.8889617566427750 & 120.2125697911546141 \\
9 & 17.9291240142580452 & 155.7621950801096596 \\
10 & 19.9558689242933884 & 195.2366537978909468 \\
    \end{array}
  \end{equation*}}
  \caption{Numerical values of the constants in mean and variance of the number of
    leaves on the last level for $t\in\{2,3\}$, cf.\
    Theorem~\ref{theorem:distribution-m}.
  See also Remark~\ref{rem:numerical-calc}.
  For the accuracy of these numerical results see the note at the end of
  the introduction.}
  \label{tab:mu_t-sigma_t-values}
\end{table}
Note that by Lemma~\ref{lemma:b-q-u-1-1-formula}, $p_m=0$ for non-integer $m$.


Again, we visualize the distribution of the leaves on the last level for a
given parameter set, see Figure~\ref{fig:leaves}. This is compared with the
mean of Theorem~\ref{theorem:distribution-m}.


\begin{figure}
  \centering
   \begin{tikzpicture}
    \begin{axis}[
      xlabel=leaves on the last level,
      ylabel=probability,
      legend pos=outer north east]
    \addplot[color=blue, mark=*, only marks, mark size=1.5]
         table[x=leaves, y=count] {plot_leaves_true.txt};
    \addplot[color=black, mark=x, only marks, mark size=4.5]
         table[x=leaves, y=count] {plot_leaves_pm.txt};
    \addplot[color=black, no marks, very thick, dashed]
         table[x=leaves, y=count] {plot_leaves_mean.txt};
    \legend{true values, $p_m$, expectation}
    \end{axis}
  \end{tikzpicture}
 
  \caption{Distribution of the leaves on the last level for $t=2$ and $n=30$
    inner vertices. On the one hand, this figure shows the true distribution of
    all trees of the given size and on the other hand the result on the
    expectation of this distribution (Theorem~\ref{theorem:distribution-m} with
    only main term of mean taken into account).}
  \label{fig:leaves}
\end{figure}
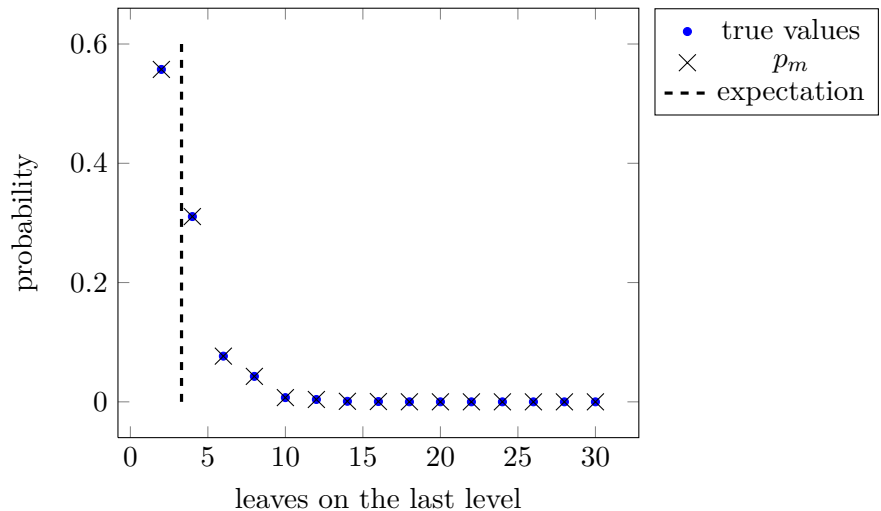


\begin{proof} 
  As the variables $v$ and $w$ do not play any role,
  we write $H(q,u)$, $a(q,u)$ and $b(q,u)$ instead of
  $H(q,u,1,1)$, $a(q,u,1,1)$ and $b(q,u,1,1)$, respectively.

  By \eqref{eq:U-estimates}, we have $U^{1-t} q_0/Q <1$, i.e., $a(q,u)$ and $b(q,u)$ are analytic for
  $\abs{q}\le q_0/Q$ and $\abs{u}\le 1/U$ by Theorem~\ref{theorem:generating-function-height-and-others}. By
  \eqref{eq:generating-function-height-and-others} and
  Lemma~\ref{lemma:dominant-singularity}, the meromorphic function $q\mapsto
  H(q,u)$ for fixed $u$ with $\abs{u}\le 1/U$ has a unique singularity in $\{ q
  \mid \abs{q}\le q_0/Q \}$, namely $q_0$, independently of $u$.

  We use Cauchy's formula, the residue theorem and
  the fact that $a(q,u)$ does not contribute to the residue at $q=q_0$ to obtain
  \begin{align*}
    [q^n]H(q,u)&=\frac1{2\pi
      i}\oint_{\abs{q}=1/2}\frac{H(q,u)}{q^{n+1}}\dd q\\
    &= -\Res\left(\frac{H(q, u)}{q^{n+1}}, q=q_0\right)+\frac1{2\pi
      i}\oint_{\abs{q}=q_0/Q}\frac{H(q,u)}{q^{n+1}}\dd q\\
    &=\frac{b(q_0,u)\nu(r)}{q_0^{n}}+\frac1{2\pi
      i}\oint_{\abs{q}=q_0/Q}\frac{H(q,u)}{q^{n+1}}\dd q
  \end{align*}
  where $\nu(r)$ has been defined in \eqref{eq:nu-definition}.

  By Lemma~\ref{lemma:number-canonical-forests}, the probability generating function
  $P_n(u)$ of $m(T)$ is given by
  \begin{equation}\label{eq:probability_convergence}
    P_n(u)=b(q_0,u)+O(Q^n),
  \end{equation}
  uniformly for $\abs{u}\le 1/U$ and uniformly in the number of roots $r$ (it
  suffices to bound the numerator and the denominator of $H(q,u)$ separately in
  order to get a uniform bound in $r$). We remark that this proves non-negativity of the constants $p_m$,
which we required in the proof of Proposition~\ref{proposition:p_r-bounds}.

  Expectation and variance follow upon differentiating $b(q_0,u)$ with respect
  to $u$ and inserting the asymptotic expression for $q_0$. Here, we use the
  bounds derived in Lemma~\ref{lem:den-bounds}.

  In order to compute $\P(m(T)=mt)$, we consider
  \begin{equation*}
    [u^{mt}][q^n]H(q,u)=p_m \frac{\nu(r)}{q_0^{n}}
    +\frac1{(2\pi i)^2}\oint_{\abs{u}=1/U}\oint_{\abs{q}=q_0/Q}
    \frac{H(q,u)}{q^{n+1}u^{mt+1}} \dd q \dd u.
  \end{equation*}
  Bounding $H(q,u)$ uniformly in $r$ and using
  Lemma~\ref{lemma:number-canonical-forests} and
  Proposition~\ref{proposition:p_r-bounds} yields
  \eqref{eq:probability-distribution-m},
  taking into account that
  \begin{equation*}
    \frac{U^t}{q_*}=2+\frac{2}{t^2}\neweps(t)
  \end{equation*}
  for $t\ge 30$ and that $U^t/q_*>1$ remains true for all $t\ge 2$.
\end{proof}


\section{The Path Length}\label{sec:path-length}


This section is devoted to the analysis of the path length. While the external
path length is most natural in the setting of Huffman codes, it is more
convenient to work with the total and the internal path length,
respectively. As it was pointed out in the introduction, all three are
essentially equivalent, since they are (deterministically) related by simple
linear equations.


\begin{theorem}\label{thm:pathlength}
  For a randomly chosen tree $T\in\calT$ of size $n$ the total path length (as
  well as the internal and the external path length) is asymptotically (for
  $n\to\infty$) normally distributed. Its mean is asymptotically
  $\mu_{\mathit{tpl}} n^2 + O(n)$ and its variance is asymptotically
  $\sigma_{\mathit{tpl}}^2 n^3 + O(n^2)$ with
  \begin{equation*}
    \mu_{\mathit{tpl}} = \frac{t}{2} \mu_h = 
    \frac{t}{4} + \frac{t^2-2t}{2^{t+4}} + \frac{2t^3+3t^2-8t}{2^{2t+6}}
    + \frac{9t^4+45t^3+2t^2-88t}{2^{3t+9}} 
    + \frac{0.048 t^5}{2^{4t}} \neweps(t)\label{eps:tpl:m}
  \end{equation*}
  and
  \begin{multline*}
    \sigma_{\mathit{tpl}}^2 = \frac{t^2}{12}
    + \frac{-t^4+5t^3-2t^2}{3\cdot 2^{t+4}}
    + \frac{-6t^5+6t^4+27t^3-14t^2}{3\cdot 2^{2t+6}} \\
    + \frac{-27t^6-72t^5+237t^4+302t^3-232t^2 }{3\cdot 2^{3t+9}}
    + \frac{0.078 t^7}{2^{4t}} \neweps(t)\label{eps:tpl:v}
  \end{multline*}
  for $t\geq30$. 
\end{theorem}


We determined numerical values of these constants as in the previous
sections. They are given in Table~\ref{tab:special-values-pathlength}.
Figure~\ref{fig:total} shows the result of Theorem~\ref{thm:pathlength} for
particular values. It compares the obtained normality with the distribution of
the total path length found by a simulation in SageMath.


\tablepathlength{}


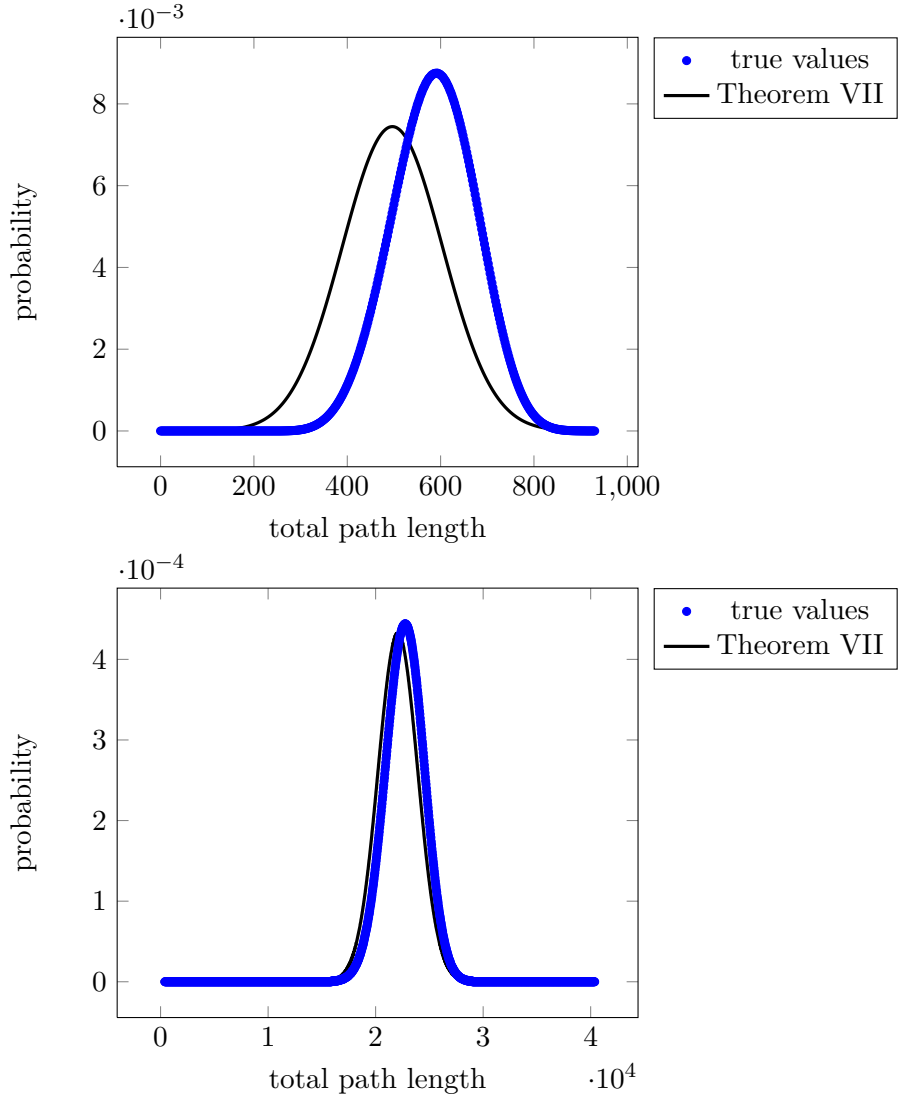
\begin{figure}
  \centering
   \begin{tikzpicture}
    \begin{axis}[
      xlabel=total path length,
      ylabel=probability,
      legend pos=outer north east]
    \addplot[color=blue, mark=*, only marks, mark size=1.5]
         table[x=total30, y=count] {plot_total30_true.txt};
    \addplot[color=black, no marks, very thick]
         table[x=total30, y=count] {plot_total30_asy.txt};
    \legend{true values, Theorem~\ref{thm:pathlength}}
    \end{axis}
  \end{tikzpicture}
   \begin{tikzpicture}
    \begin{axis}[
      xlabel=total path length,
      ylabel=probability,
      legend pos=outer north east]
    \addplot[color=blue, mark=*, only marks, mark size=1.5]
         table[x=total200, y=count] {plot_total200_true.txt};
    \addplot[color=black, no marks, very thick]
         table[x=total200, y=count] {plot_total200_asy.txt};
    \legend{true values, Theorem~\ref{thm:pathlength}}
    \end{axis}
  \end{tikzpicture}
 
  \caption{Distribution of the total path length for $t=2$, and $n=30$
    (top figure) and $n=200$ (bottom figure) inner
    vertices. On the one hand, this figure shows the true distribution of all
    trees of the given size and on the other hand the result on the asymptotic
    normal distribution (Theorem~\ref{thm:pathlength} with only main terms of
    mean and variance taken into account). In order to take into account that
    the total path length is always even, we rescale the limit distribution.}
  \label{fig:total}
\end{figure}


We first use a generating functions approach to determine the asymptotic
behaviour of the mean and variance. Let us define
\begin{equation*}
L_r(q,u,w) \colonequals \sum_{T \in \calT} \ell(T)^r q^{n(T)}u^{m(T)}w^{h(T)}
\end{equation*}
for the $r$-th moment of the total path length. Note that
\begin{equation*}
 L_0(q,u,w) = H(q,u,1,w) =  a_0(q,u,w) + b(q,u,w)\frac{a_0(q,1,w)}{1-b(q,1,w)}
\end{equation*}
in the notation of
Theorem~\ref{theorem:generating-function-height-and-others}, but writing $a_0$
instead of $a$ and leaving out the parameter $v$.

We are
specifically interested in $L_1$ and $L_2$. In analogy to the approach we used
to determine a formula for $H(q,u,v,w)$ in the proof of Theorem~\ref{theorem:generating-function-height-and-others}, we obtain a functional equation for $L_r(q,u,w)$ by first introducing
\begin{equation*}
L_{r,h} (q,u) = [w^h] L_r(q,u,w) = \sum_{\substack{T \in \calT \\ h(T) = h}} \ell(T)^r q^{n(T)}u^{m(T)}.
\end{equation*}


Define, for the sake of convenience, the linear operators $\Phi_u = u
\frac{\partial}{\partial u}$, $\Phi_w = w \frac{\partial}{\partial w}$ and
$\Phi_q = q \frac{\partial}{\partial q}$ acting on our generating functions. We
get the following result for the generating function of the first moment.


\begin{lemma}\label{lemma:L_1-functional-equation}
  We have
  \begin{equation*}
    L_1(q,u,w) = a_1(q,u,w) + b(q,u,w) \frac{a_1(q,1,w)}{1-b(q,1,w)},
  \end{equation*}
  with
  \begin{equation*}
    a_1(q,u,w) = \sum_{j=0}^{\infty} (-1)^j w^j 
    (\Phi_u\Phi_w L_0)(q,q^{\hp{j}}u^{t^j},w) 
    \prod_{i=1}^j \frac{q^{\hp{i}}u^{t^i}}{1-q^{\hp{i}}u^{t^i}}.
  \end{equation*}
\end{lemma}


\begin{proof}
Replacing $j$ leaves of depth $h$ by internal vertices, thus creating $jt$ new leaves of depth $h+1$, increases the total path length by $jt(h+1)$.
Thus we get
\ifproc
\begin{align*}
  L_{1,h+1}(q,u) &= \sum_{\substack{T \in \calT \\ h(T) = h+1}} (h+1)m(T)
  q^{n(T)}u^{m(T)} \\ 
  &\phantom{=}+ \sum_{\substack{T \in \calT \\ h(T) = h}} 
  \sum_{j=1}^{m(T)} \ell(T) q^{n(T)+j}u^{jt} \\
  &= (h+1)u \frac{\partial}{\partial u} L_{0,h+1}(q,u) \\ 
  &\phantom{=}+ \frac{qu^t}{1-qu^t} \left( L_{1,h}(q,1) -  L_{1,h}(q,qu^t) \right).
\end{align*}
\else
\begin{align*}
  L_{1,h+1}(q,u) &=  \sum_{\substack{T \in \calT \\ h(T) = h}} 
  \sum_{j=1}^{m(T)} \ell(T) q^{n(T)+j}u^{jt} 
  + \sum_{\substack{T \in \calT \\ h(T) = h+1}} (h+1)m(T)
  q^{n(T)}u^{m(T)} \\
  &= \frac{qu^t}{1-qu^t} \big( L_{1,h}(q,1) -  L_{1,h}(q,qu^t) \big)
  + (h+1)u \frac{\partial}{\partial u} L_{0,h+1}(q,u) 
\end{align*}
\fi and $L_{1,0}(q,u) = 0$.  Then, by multiplying by $w^{h+1}$ and
summing over all $h$, we obtain \ifproc
\begin{align*}
L_1(q,u,w) &= \Phi_u\Phi_w L_0(q,u,w) \\ 
&\phantom{=}+ \frac{qu^tw}{1-qu^t} \left( L_1(q,1,w) - L_1(q,qu^t,w) \right).
\end{align*}
\else
\begin{equation*}
  L_1(q,u,w) = \frac{qu^tw}{1-qu^t} \left( L_1(q,1,w) - L_1(q,qu^t,w) \right)
  + \Phi_u\Phi_w L_0(q,u,w).
\end{equation*}
\fi
Lemma~\ref{lemma:functional-equation} yields the desired formula 
for $L_1(q,u,w)$.
\end{proof}


Next, we derive a formula for the generating function of the second
moment.


\begin{lemma}\label{lem:tpl:L2}
  We have
  \begin{equation*}
    L_2(q,u,w) = a_2(q,u,w) + b(q,u,w) \frac{a_2(q,1,w)}{1-b(q,1,w)}
  \end{equation*}
  with
  \begin{equation*}
    a_2(q,u,w) = \sum_{j=0}^{\infty} (-1)^j w^j 
    \left( 2(\Phi_u\Phi_w L_1)(q,q^{\hp{j}}u^{t^j},w) 
      - (\Phi_u^2\Phi_w^2 L_0)(q,q^{\hp{j}}u^{t^j},w) \right) 
    \prod_{i=1}^j \frac{q^{\hp{i}}u^{t^i}}{1-q^{\hp{i}}u^{t^i}}.
  \end{equation*}
\end{lemma}


\begin{proof}
As in Lemma~\ref{lemma:L_1-functional-equation}, we derive a functional equation for $L_2(q,u,w)$. Starting with a tree
$T$ of height~$h$ and creating $jt$ new leaves of depth $h+1$ changes the
square of the total
path length from $\ell(T)^2$ to $\left(\ell(T) + jt(h+1)\right)^2$. This
translates to
\begin{align*}
  L_{2,h+1}(q,u) &= \sum_{\substack{T \in \calT \\ h(T) = h}} 
  \sum_{j=1}^{m(T)} \ell(T)^2 q^{n(T)+j} u^{jt}
  + \sum_{\substack{T \in \calT \\ h(T) = h+1}} (h+1)^2 m(T)^2
  q^{n(T)} u^{m(T)} \\
  &\alignpheqop{+} 2 \sum_{\substack{T \in \calT \\ h(T) = h+1}} (h+1)
  \big(\ell(T) - m(T)(h+1)\big)m(T) q^{n(T)} u^{m(T)}  \\
  &= \frac{qu^t}{1-qu^t} \left( L_{2,h}(q,1) -  L_{2,h}(q,qu^t) \right) \\
  &\alignpheqop{+} 2(h+1) \Phi_u L_{1,h+1}(q,u)
  - (h+1)^2 \Phi_u^2 L_{0,h+1}(q,u).
\end{align*}
Note that we have $L_{2,0}(q,u) = 0$. Encoding the height by $w^h$ leads to the
functional equation for the generating function
\begin{equation*}
  L_2(q,u,w) = \frac{qu^tw}{1-qu^t} \big( L_2(q,1,w) - L_2(q,qu^t,w) \big)
  + 2\Phi_u\Phi_w L_1(q,u,w) - \Phi_u^2\Phi_w^2 L_0(q,u,w).
\end{equation*}
Again, Lemma~\ref{lemma:functional-equation} finishes this proof.
\end{proof}


In order to determine the asymptotic behaviour of mean and variance, one only
needs to find the expansion around the dominating singularity $q_0$ and apply
singularity analysis. The main term of the mean is easy to guess: assuming that
the vertices are essentially uniformly distributed along the entire height, it
is natural to conjecture that $\ell(T)$ is typically around $tn(T)h(T)/2$ and
thus of quadratic order. This is indeed true, and the variance turns out to be
of cubic order (terms of degree 4 cancel, as one would expect). The following
lemma substantiates these claims for the mean.


\begin{proposition}\label{pro:mu_tpl:mu_h}
  The mean of the total path length is $\mu_{\mathit{tpl}} n^2 + O(n)$ with
  \begin{equation*}
    \mu_{\mathit{tpl}} = \frac{t}{2} \mu_h.
  \end{equation*}
\end{proposition}

\begin{proof}
  By substituting $L_0$ into the functional equation of
  Lemma~\ref{lemma:L_1-functional-equation}, we get an explicit expression for
  $L_1(q,1,w)$, namely
  \begin{align*}
    L_1(q,1,w) &= \frac{a_0(q,1,w)(\Phi_w b)(q,1,w)}{(1-b(q,1,w))^3}
    \sum_{j=0}^{\infty}  (-1)^{j}w^j  (\Phi_u b)(q,q^{\hp{j}},w)
    \prod_{i=1}^j \frac{q^{\hp{i}}}{1-q^{\hp{i}}} \\
    &\alignpheqop{+} \frac{a_0(q,1,w)}{(1-b(q,1,w))^2} \sum_{j=0}^{\infty}
    (-1)^{j}w^j  (\Phi_u\Phi_w b)(q,q^{\hp{j}},w) 
    \prod_{i=1}^j \frac{q^{\hp{i}}}{1-q^{\hp{i}}}\\
    &\alignpheqop{+} \frac{(\Phi_w a_0)(q,1,w)}{(1-b(q,1,w))^2}
    \sum_{j=0}^{\infty}  (-1)^{j}w^j  (\Phi_u b)(q,q^{\hp{j}},w)
    \prod_{i=1}^j \frac{q^{\hp{i}}}{1-q^{\hp{i}}}\\
    &\alignpheqop{+} \frac{1}{1-b(q,1,w)} \sum_{j=0}^{\infty}  (-1)^{j}w^j
    (\Phi_u \Phi_w a_0)(q,q^{\hp{j}},w)
    \prod_{i=1}^j \frac{q^{\hp{i}}}{1-q^{\hp{i}}}.
  \end{align*}
  The dominant term in this sum is the first one, with a triple pole at the
  dominant singularity $q_0$. The second and third term, however, are also
  relevant in the calculation of the variance, where one further term in the
  asymptotic expansion is needed in view of the inevitable cancellation in the
  main term.  Singularity analysis immediately yields the asymptotic behaviour
  of the mean: since the pole is of cubic order, the order of the mean is
  quadratic, i.e., it is asymptotically equal to $\mu_{\mathit{tpl}} n^2$,
  where the constant $\mu_{\mathit{tpl}}$ is given by
  \begin{equation}\label{eq:mu_tpl_prelim}
    \mu_{\mathit{tpl}} = \frac{(\Phi_w b)(q_0,1,1)}{2(\Phi_q b)(q_0,1,1)^2}
    \sum_{j=0}^{\infty}  (-1)^{j} (\Phi_u b)(q_0,q_0^{\hp{j}},1)
    \prod_{i=1}^j \frac{q_0^{\hp{i}}}{1-q_0^{\hp{i}}} .
  \end{equation}
  Plugging in the definition of $b$ as a sum, it is possible to simplify this
  further. One has
  \begin{equation*}
    (\Phi_u b)(q,u,1) = \sum_{k=1}^{\infty} (-1)^{k-1} 
    \biggl(\prod_{h=1}^k \frac{q^{\hp{h}}u^{t^h}}{1-q^{\hp{h}}u^{t^h}}\biggr)
    \sum_{h=1}^k \frac{t^h}{1-q^{\hp{h}}u^{t^h}}
  \end{equation*}
  by logarithmic differentiation and thus
  \begin{align*}
    (\Phi_u b)(q,q^{\hp{j}},1) &= \sum_{k=1}^{\infty} (-1)^{k-1}
    \biggl(\prod_{h=1}^k
    \frac{q^{\hp{h}+t^h\hp{j}}}{1-q^{\hp{h}+t^h\hp{j}}}\biggr)
    \sum_{h=1}^k \frac{t^h}{1-q^{\hp{h}+t^h\hp{j}}} \\
    &= \sum_{k=1}^{\infty} (-1)^{k-1}
    \biggl(\prod_{i=j+1}^{j+k} \frac{q^{\hp{i}}}{1-q^{\hp{i}}}\biggr)
    \sum_{h=1}^k \frac{t^h}{1-q^{\hp{h+j}}}
  \end{align*}
  since $\hp{h}+t^h\hp{j} = \hp{h+j}$ by definition. Plugging this into~\eqref{eq:mu_tpl_prelim}, we find
  \begin{equation*}
    \mu_{\mathit{tpl}} = \frac{(\Phi_w b)(q_0,1,1)}{2(\Phi_q b)(q_0,1,1)^2}
    \sum_{j=0}^{\infty} \sum_{k=1}^{\infty} (-1)^{j+k-1}
    \biggl(\prod_{i=1}^{j+k} \frac{q_0^{\hp{i}}}{1-q_0^{\hp{i}}}\biggr)
    \sum_{h=1}^k \frac{t^h}{1-q_0^{\hp{h+j}}}.
  \end{equation*}
  Substituting $\ell = j+k$ and interchanging the order of summation, we arrive
  at
  \begin{align*}
    \mu_{\mathit{tpl}} &= \frac{(\Phi_w b)(q_0,1,1)}{2(\Phi_q b)(q_0,1,1)^2}
    \sum_{\ell=1}^{\infty}  (-1)^{\ell-1}
    \biggl(\prod_{i=1}^{\ell} \frac{q_0^{\hp{i}}}{1-q_0^{\hp{i}}}\biggr)
    \sum_{k=1}^{\ell} \sum_{h=1}^k \frac{t^h}{1-q_0^{\hp{h+\ell-k}}} \\
    &= \frac{(\Phi_w b)(q_0,1,1)}{2(\Phi_q b)(q_0,1,1)^2} 
    \sum_{\ell=1}^{\infty}  (-1)^{\ell-1}
    \biggl(\prod_{i=1}^{\ell} \frac{q_0^{\hp{i}}}{1-q_0^{\hp{i}}}\biggr)
    \sum_{r=1}^{\ell} \sum_{h=1}^r \frac{t^h}{1-q_0^{\hp{r}}} \\
    &= \frac{(\Phi_w b)(q_0,1,1)}{2(\Phi_q b)(q_0,1,1)^2}
    \sum_{\ell=1}^{\infty}  (-1)^{\ell-1}
    \biggl(\prod_{i=1}^{\ell} \frac{q_0^{\hp{i}}}{1-q_0^{\hp{i}}}\biggr)
    \sum_{r=1}^{\ell} \frac{t \hp{r}}{1-q_0^{\hp{r}}}.
  \end{align*}
  Noting now that
  \begin{equation*}
    (\Phi_qb)(q,1,1) = \sum_{\ell=1}^{\infty}  (-1)^{\ell-1}
    \biggl(\prod_{i=1}^{\ell} \frac{q^{\hp{i}}}{1-q^{\hp{i}}}\biggr)
    \sum_{r=1}^{\ell} \frac{\hp{r}}{1-q^{\hp{r}}},
  \end{equation*}
  which can be seen by another logarithmic differentiation, we can replace the
  sum in the expression for $\mu_{\mathit{tpl}}$ above by $t \cdot
  (\Phi_qb)(q_0,1,1)$, which finally yields
  \begin{equation*}
    \mu_{\mathit{tpl}} = \frac{t}{2} \cdot 
    \frac{(\Phi_w b)(q_0,1,1)}{(\Phi_q b)(q_0,1,1)},
  \end{equation*}
  and the second fraction is precisely $\mu_h$, cf.\
  Equation~\eqref{eq:mu_h_explicit}.
\end{proof}


Our next goal is to obtain the asymptotics of the variance, which will again
follow by applying the tools from singularity analysis together with the result
for the mean shown above.

Let us use the abbreviation
\begin{equation*}
  \f{\Sigma}{q, M, \Phi}
  = \sum_{j=0}^\infty (-1)^j \f{M}{j}
  \f{(\Phi b)}{q,q^{\hp{j}},1}
  \biggl( \prod_{i=1}^j \frac{q^{\hp{i}}}{1-q^{\hp{i}}} \biggr),
\end{equation*}
where $M$ is a function in the variable~$j$ and $\Phi$ an
operator, to simplify the expressions in the following lemma.


\begin{lemma}\label{lem:sigma_tpl}
  The variance of the total path length is $\sigma_{\mathit{tpl}}^2 n^3 +
  O(n^2)$, where
  \begin{align*}
    \sigma^2_{\mathit{tpl}}
    &=\frac{(\Phi_q^2 b)\left(q_{0}, 1, 1\right)\,
      (\Phi_w b)\left(q_{0}, 1, 1\right)^{2}}%
    {(\Phi_q b)\left(q_{0}, 1, 1\right)^{5}}
    \f{\Sigma}{q_{0}, j\mapsto 1, \Phi_u}^{2} \\
    &\alignpheqop{-}\frac{(\Phi_q \Phi_w b)\left(q_{0}, 1, 1\right)\,
      (\Phi_w b)\left(q_{0}, 1, 1\right)}%
    {(\Phi_q b)\left(q_{0}, 1, 1\right)^{4}}
    \f{\Sigma}{q_{0}, j\mapsto 1, \Phi_u}^{2} \\
    &\alignpheqop{-}\frac{(\Phi_w b)\left(q_{0}, 1, 1\right)^{2}}%
    {(\Phi_q b)\left(q_{0}, 1, 1\right)^{4}}
    \f{\Sigma}{q_{0}, j\mapsto 1, \Phi_u}\bigg(
    \f{\Sigma}{q_{0}, j\mapsto 1, \Phi_q \Phi_u}
    +\f[big]{\Sigma}{q_{0}, j\mapsto \hp{j}, \Phi_u^2} \\
    &\qquad\qquad+\f[bigg]{\Sigma}{q_{0},
      j\mapsto \sum_{i=1}^j \frac{\hp{i}}{1-q_0^{\hp{i}}}, \Phi_u}\bigg) \\
    &\alignpheqop{+} \frac{(\Phi_w b)(q_0,1,1)^2}{3(\Phi_qb)(q_0,1,1)^3}
    \bigg( \f[big]{\Sigma}{q_0, j\mapsto 2\hp{j+1}-1, \Phi_u^2}
    + \f[bigg]{\Sigma}{q_0,
      j\mapsto 2t\sum_{i=1}^j \frac{\hp{i}}{1-q_0^{\hp{i}}}, \Phi_u}\bigg) \\
    &\alignpheqop{+} \frac{(\Phi_w^2 b)(q_0,1,1)}{3(\Phi_qb)(q_0,1,1)^3}
    \f[big]{\Sigma}{q_0, j\mapsto 1, \Phi_u}^2 \\
    &\alignpheqop{+} \frac{(\Phi_w b)(q_0,1,1)}{3(\Phi_qb)(q_0,1,1)^3}
    \f[big]{\Sigma}{q_0, j\mapsto 1, \Phi_u}\Bigl(
    \f[big]{\Sigma}{q_0, j\mapsto 1, \Phi_u\Phi_w}+
    \f[big]{\Sigma}{q_0, j\mapsto j, \Phi_u}\Bigr).
  \end{align*}
\end{lemma}


\begin{proof}
  In order to calculate the variance, one needs, besides the result of
  Proposition~\ref{pro:mu_tpl:mu_h}, the asymptotic behaviour of $L_2(q,1,1)$ at the
  dominant singularity. Only the terms of pole order $4$ and $5$ (i.e., highest
  and second-highest) are needed. More details on the computation can be found in
  the appendix. By Lemma~\ref{lem:tpl:L2} we obtain
  \begin{align*}
    L_2(q,1,1) &= \frac{6a_0(q,1,1)(\Phi_w b)(q,1,1)^2}{(1-b(q,1,1))^5}
    \f[big]{\Sigma}{q, j\mapsto 1, \Phi_u}^2 \\
    &\alignpheqop{+} \frac{4a_0(q,1,1)(\Phi_w b)(q,1,1)^2}{(1-b(q,1,1))^4}
    \left( \f[big]{\Sigma}{q, j\mapsto \hp{j+1}, \Phi_u^2}
      + \f[bigg]{\Sigma}{q,
        j\mapsto \sum_{i=1}^j \frac{t\hp{i}}{1-q^{\hp{i}}}, \Phi_u}\right) \\
    &\alignpheqop{+} \frac{8a_0(q,1,1)(\Phi_w b)(q,1,1)}{(1-b(q,1,1))^4}
    \f[big]{\Sigma}{q, j\mapsto 1, \Phi_u}
    \f[big]{\Sigma}{q, j\mapsto 1, \Phi_u\Phi_w} \\
    &\alignpheqop{+} \frac{6(\Phi_w a_0)(q,1,1)(\Phi_w b)(q,1,1)}{(1-b(q,1,1))^4}
    \f[big]{\Sigma}{q, j\mapsto 1, \Phi_u}^2 \\
    &\alignpheqop{+} \frac{2a_0(q,1,1)(\Phi_w^2 b)(q,1,1)}{(1-b(q,1,1))^4}
    \f[big]{\Sigma}{q, j\mapsto 1, \Phi_u}^2 \\
    &\alignpheqop{+} \frac{2a_0(q,1,1)(\Phi_w b)(q,1,1)}{(1-b(q,1,1))^4}
    \f[big]{\Sigma}{q, j\mapsto 1, \Phi_u}
    \f[big]{\Sigma}{q, j\mapsto j, \Phi_u} \\
    &\alignpheqop{-} \frac{2a_0(q,1,1)(\Phi_w b)(q,1,1)^2}{(1-b(q,1,1))^4}
    \f[big]{\Sigma}{q, j\mapsto 1, \Phi_u^2} \\
    &\alignpheqop{+} O\!\left(\frac{1}{(1-b(q,1,1))^3}\right)
  \end{align*}
  as $q$ tends to $q_0$.

  Applying singularity analysis to the highest- and second-highest order terms
  of both $L_1$ and $L_2$ yields the variance. The terms of order $n^4$ cancel
  (as one would expect), and one finds that the main term of the variance is
  asymptotically $\sigma^2_{\mathit{tpl}} n^3$.
\end{proof}


In order to obtain expressions (either the asymptotics in $t$ or the values for
particular given~$t$) of $\mu_{\mathit{tpl}}$ and $\sigma_{\mathit{tpl}}^2$ we
insert the dominant singularity~$q_0$ (see
Lemma~\ref{lemma:dominant-singularity}) into the formul\ae{} obtained in
Proposition~\ref{pro:mu_tpl:mu_h} and Lemma~\ref{lem:sigma_tpl}. We remind the
reader again that it is important to establish that $\sigma_{\mathit{tpl}}^2 \neq 0$, so
numerical values and estimates for large $t$ are needed again. A couple of
technical difficulties arise due to the infinite sums. These are discussed in the following remark.


\begin{remark}\label{rem:get-asy-tpl}
  We use the SageMath~\cite{Stein-others:2015:sage-mathem-6.5} mathematics
  software system for our calculations. In order to get the asymptotic
  expression and values for $\sigma_{\mathit{tpl}}^2$ in
  Theorem~\ref{thm:pathlength} (note that we have $\mu_{\mathit{tpl}}$ already
  due to Proposition~\ref{pro:mu_tpl:mu_h} and the results of
  Section~\ref{sec:height}), we have to evaluate infinite sums and 
  insert the dominant singularity~$q_0$.

  We will explain step by step how this is done.

  \begin{enumerate}[(a)]

  \item We start with the expression for $\sigma_{\mathit{tpl}}^2$ found
    in Lemma~\ref{lem:sigma_tpl}.

  \item First, let us consider the infinite sums $\f{\Sigma}{q_0, M, \Phi}$.
    For a suitable $J_\Sigma$ depending on $t$, we
    calculate the first $J_\Sigma$ summands directly and use a bound for the
    tails. More precisely, we use
    \begin{multline*}
      \sum_{j=J_\Sigma}^\infty (-1)^j \f{M}{j}
        \f{(\Phi b)}{q_0,q_0^{\hp{j}},1}
        \biggl( \prod_{i=1}^j \frac{q_0^{\hp{i}}}{1-q_0^{\hp{i}}} \biggr) \\
      \in I \f{(\Phi b)}{q_0, I q_0^{\hp{J_\Sigma}},1}
      \biggl( \prod_{i=1}^{J_\Sigma} \frac{q_0^{\hp{i}}}{1-q_0^{\hp{i}}} \biggr)
      \sum_{j=J_\Sigma}^\infty M_j Q^{j-J_\Sigma}
    \end{multline*}
    with the interval $I=[-1,1]$, $\f{M}{j} \leq M_j$ for $j\geq J_\Sigma$ and
    $Q=q_0^{\hp{J+1}} / (1-q_0^{\hp{J+1}})$.

    Let us consider the bound $M_j$. If $\f{M}{j} = \hp{j}$, we set $M_j = t^j
    / (t-1)$ and analogously for $\f{M}{j} = \hp{j+1}$. If $\f{M}{j} =
    \sum_{i=1}^j t\hp{i} / (1-q_0^{\hp{i}})$, we use $M_j = 2t^{j+1} /
    (1-q_0)$. Otherwise ($\f{M}{j} = 1$ and $\f{M}{j} = j$), we simply take
    $M_j = \f{M}{j}$. These choices allow us to find a closed form for
    $\sum_{j=J_\Sigma}^\infty M_j Q^{j-J_\Sigma}$.

    Proceeding as described above gives an expression consisting of finitely
    many summands containing 
    functions~$b$, which will be handled in the following step.

  \item Let us deal with the function~$b(q,u,w)$ and its derivatives,
    which all are infinite sums. As above, we calculate the first $J_b$
    summands directly for a suitable $J_b$ chosen depending on $t$. Then we add the bound provided by
    Lemmata~\ref{lem:den-bounds-pre} and~\ref{lem:den-bounds} to take care
    of the tails. 

    At this point, we end up with a symbolic expression not containing any
    (visible or hidden) infinite sums; only the variables $t$, $q_0$, $U$ and
    the interval~$I$ occur. Thus, we are almost ready to insert the asymptotic
    expressions or values for these parameters.

  \item Now, we are ready to insert the dominant singularity~$q_0$. On the one
    hand, this can be the asymptotic expansion of $q_0$ as $t\to\infty$ (in our
    case valid for $t\geq30$), cf.\@ Lemma~\ref{lemma:dominant-singularity}. We
    choose $J_\Sigma=J_b=3$. The result will then again be an asymptotic
    expression for~$\sigma_{\mathit{tpl}}^2$.

    On the other hand, we can use a particular value for $q_0$ for given $t$
    (which for us means, more precisely, an interval containing $q_0$). In
    these cases, we choose $J_\Sigma=J_b=4$ for $8\leq t\leq30$ and higher
    values for $t<8$ (up to $J_\Sigma=J_b=14$ for $t=2$). The
    resulting~$\sigma_{\mathit{tpl}}^2$ is then computed using interval
    arithmetic.
  \end{enumerate}

\end{remark}


In order to prove asymptotic normality of the total path length, a different, more probabilistic approach is needed. Standard
theorems from analytic combinatorics no longer apply since the path length
grows faster than, for example, the height, so that mean and variance no longer
have linear order.

We number the internal vertices of a random canonical $t$-ary tree of size $n$
from $1$ to $n$ in a natural top-to-bottom, left-to-right way, starting at the
root. Let $X_{k,n}$ denote the depth of the $k$-th internal vertex~$v_k$ in a
random tree $T \in \mathcal T$ of order $n$. Moreover, set $Y_{k,n} = X_{k+1,n}
- X_{k,n} \in \{0,1\}$. In words, $Y_{k,n}$ is $1$ if the $(k+1)$-th internal
vertex has greater distance from the root than the $k$-th, and $0$
otherwise. It is clear that the height can be expressed as
\begin{equation*}
h(T) = 1 + \max_k X_{k,n} = 1 + X_{n,n} = 1 + \sum_{k=1}^{n-1} Y_{k,n},
\end{equation*}
which would indeed be an alternative approach to the central limit theorem for
the height. More importantly, though, the internal path length can also be
expressed in terms of the random variables $Y_{k,n}$, namely by
\ifproc
\begin{align*}
  \ell_{\mathit{internal}}(T) &= \sum_{k=1}^n X_{k,n} 
  = \sum_{k=1}^n \sum_{j=1}^{k-1} Y_{j,n} \\
  &= \sum_{j=1}^{n-1} (n-j)Y_{j,n}.
\end{align*}
\else
\begin{equation*}
\ell_{\mathit{internal}}(T) = \sum_{k=1}^n X_{k,n} = \sum_{k=1}^n \sum_{j=1}^{k-1} Y_{j,n} = \sum_{j=1}^{n-1} (n-j)Y_{j,n}.
\end{equation*}
\fi
Now
\begin{equation*}
n^{-1} \ell_{\mathit{internal}}(T) = \sum_{j=1}^{n-1} \frac{n-j}{n} Y_{j,n}
\end{equation*}
can be seen as a sum of $n-1$ bounded random variables $Z_{j,n} = \frac{n-j}{n}
Y_{j,n}$. An advantage of this decomposition over other possible decompositions
(e.g., by counting the number of vertices at different depths) is that the
number of variables is not random. Another important point is that the $Z_{j,n}$ are bounded after rescaling, so that they also have bounded moments.


Unfortunately, the $Z_{j,n}$ are neither identically distributed (which is not a major issue) nor independent, which makes standard versions of the Central Limit Theorem for sums of random variables inapplicable. However, they are almost independent in that they satisfy a so-called ``strong mixing condition'' (Inequality~\eqref{eq:mixing}
of the following lemma).

\begin{lemma}\label{lemma:mixing-condition}
Let $\mathcal{F}_{s_1}$ be the
$\sigma$-algebra induced by the random variables
$Z_{1,n},Z_{2,n},\ldots,Z_{s_1,n}$, and let $\mathcal{G}_{s_2}$ be the
$\sigma$-algebra induced by the random variables
$Z_{s_2,n},Z_{s_2+1,n},\ldots,Z_{n-1,n}$. There exist constants $\kappa$ and $\lambda$ (depending only on $t$) such that
\begin{equation}\label{eq:mixing}
\abs{\P(A \cap B) - \P(A) \P(B)} \leq \kappa e^{-\lambda (s_2-s_1)}
\end{equation}
for all $1 \leq s_1 < s_2 \leq n$ and all events $A \in \mathcal{F}_{s_1}$ and
$B \in \mathcal{G}_{s_2}$.
\end{lemma}

The main idea of the proof of the strong mixing condition is simple: events $A \in
\mathcal{F}_{s_1}$ describe the shape of the random tree $T$ up to the $s_1$-th
internal vertex $v_{s_1}$, while events $B \in \mathcal{G}_{s_2}$ describe the shape of
the random tree $T$ from the $s_2$-th internal vertex  $v_{s_2}$ on. The probabilities of
such events can be calculated by means of Lemma~\ref{lemma:number-canonical-forests}
and Theorem~\ref{theorem:distribution-m}, and the exponential error
terms that one obtains through this approach yield the estimate~\eqref{eq:mixing}
above. 


\begin{proof}[Proof of Lemma~\ref{lemma:mixing-condition}]
For a canonical tree $T$, let $F_\lambda(T)$ and $F_\rho(T)$ be the
number of internal vertices on the same level as $v_{s_1}$, left and right of
$v_{s_1}$, respectively. Similarly, let $G_\lambda(T)$ and $G_\rho(T)$ be the number
of internal vertices on the same level as $v_{s_2}$, left and right of $v_{s_2}$,
respectively. For fixed $s_1$, $f_\lambda$, $f_\rho$, $s_2$, $g_\lambda$ and
$g_\rho$,
there is a bijection between the following:
\begin{itemize}
\item the set of canonical trees $T$ with $F_\lambda(T)=f_\lambda$,
  $F_\rho(T)=f_\rho$, $G_\lambda(T)=g_\lambda$, $G_\rho(T)=g_\rho$ and such that
  $v_{s_1}$ and $v_{s_2}$ are on different levels, and
\item the set of tuples $(T_1, T_2, T_3)$
  where $T_j$ is a canonical forest with $r_j$ roots, $n_j$ internal vertices
  and $m_jt$ leaves at the last level, where the values of $r_j$ and $n_j$ are
  given in Table~\ref{tab:values-r_j-n_j-m_j}, $T_j$ has no isolated
  roots\footnote{We define an \emph{isolated root} to be a root without children.} and
  $m_jt \ge r_{j+1}$ holds for $j\in\{1,2\}$.
\end{itemize}
An illustration can be found in Figure~\ref{fig:comp-trees}.


\begin{figure}
  \centering

  \begin{tikzpicture}[scale=0.5,
    lines/.style={line width=3pt, draw=lightgray}]

    \def\levelf{-3}
    \def\levelg{-6}
    \def\levelm{-9}
    \def\fl{3}
    \def\fr{2}
    \def\fs{2}
    \def\gl{2}
    \def\gr{3}
    \def\gs{3}

    \coordinate (root1) at (0.5, 0);
    \coordinate [label=above:$s_1$] (s1) at (\fs, \levelf);
    \coordinate [label=above:$s_2$] (s2) at (\gs, \levelg);

    \draw [lines]
    (root1) -- (-\fl-2+\fs, \levelf) -- (\fr+\fs, \levelf) -- cycle;
    \draw [lines]
    (-\fl+\fs, \levelf) -- (\fr+\fs, \levelf)
    -- (\gr+\gs, \levelg) -- (-\gl+\gs-4, \levelg) -- cycle;
    \draw [lines]
    (-\gl+\gs, \levelg) -- (\gr+\gs, \levelg)
    -- (7, \levelm) -- (0, \levelm) -- cycle;

    \def\drawpoint{\fill [black] (\point) circle (6pt)}
    \foreach \point in {root1, s1, s2}
    \drawpoint;
    \foreach \i in {1,...,\fl}
    \def\point{-\i+\fs, \levelf}
    \drawpoint;
    \foreach \i in {1,...,\fr}
    \def\point{\i+\fs, \levelf}
    \drawpoint;
    \foreach \i in {1,...,\gl}
    \def\point{-\i+\gs, \levelg}
    \drawpoint;
    \foreach \i in {1,...,\gr}
    \def\point{\i+\gs, \levelg}
    \drawpoint;

    \node at ($ (root1) + (0, \levelf/2) $) {\large$T_1$};
    \node at (\fs-\fl/2+\fr/2, \levelf/2+\levelg/2) {\large$T_2$};
    \node at (\gs-\gl/2+\gr/2, \levelg/2+\levelm/2) {\large$T_3$};

    \draw [decoration={brace, mirror, raise=2mm, amplitude=4pt}, decorate]
    (\fs-\fl, \levelf) -- node [yshift=-4ex] {\small$f_\lambda$} (\fs-1, \levelf);
    \draw [decoration={brace, mirror, raise=2mm, amplitude=4pt}, decorate]
    (\fs+1, \levelf) -- node [yshift=-4ex] {\small$f_\rho$} (\fs+\fr, \levelf);
    \draw [decoration={brace, mirror, raise=2mm, amplitude=4pt}, decorate]
    (\gs-\gl, \levelg) -- node [yshift=-4ex] {\small$g_\lambda$} (\gs-1, \levelg);
    \draw [decoration={brace, mirror, raise=2mm, amplitude=4pt}, decorate]
    (\gs+1, \levelg) -- node [yshift=-4ex] {\small$g_\rho$} (\gs+\gr, \levelg);

    \node at ($ (root1) + (2, 0) $) [anchor=west]
    {$r_1=1$};
    \node at ($ (s1) + (\fr+2, 0) $) [anchor=west]
    {$r_2=6$, $f_\lambda=3$, $f_\rho=2$};
    \node at ($ (s2) + (\gr+2, 0) $) [anchor=west]
    {$r_3=6$, $g_\lambda=2$, $g_\rho=3$};
  \end{tikzpicture}

  \caption{Decomposition of canonical trees. This decomposition into $T_1$, $T_2$
    and $T_3$ is used in the proof of Lemma~\ref{lemma:mixing-condition}.}
  \label{fig:comp-trees}
\end{figure}
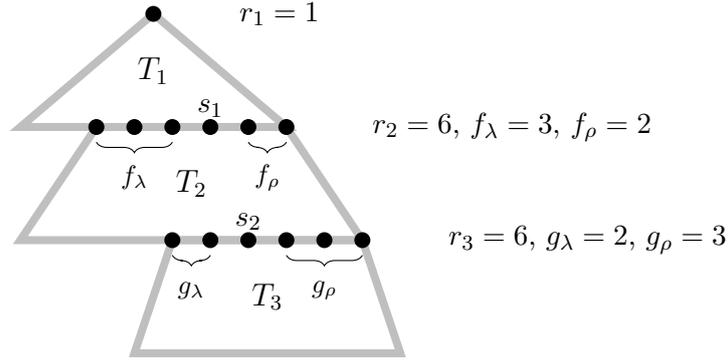


\begin{table}
  \centering{\small
  \begin{equation*}
    \begin{array}{c|c|c}
      j&r_j&n_j\\\hline
      1& 1& s_1-1-f_\lambda\\
      2& f_\lambda+1+f_\rho&s_2-1-g_\lambda-(s_1-1-f_\lambda)\\
      3& g_\lambda+1+g_\rho&n-(s_2-1-g_\lambda)
    \end{array}
  \end{equation*}}
  \caption{Values of $r_j$ and $n_j$ for the decomposition of a random tree.}
  \label{tab:values-r_j-n_j-m_j}
\end{table}


Here, $T_1$ consists of the first levels of $T$ up to and including
the level of $v_{s_1}$, $T_2$ consists of the levels of $T$ from and including the
level of $v_{s_1}$ up to and including the level of $v_{s_2}$, and $T_3$ consists of
the levels of $T$ from and including the level of $v_{s_2}$. Note that the internal
vertices of $T$ are partitioned into those of $T_1$, $T_2$ and $T_3$ as the last
level of a forest does not have any internal vertices by definition.

Note that the definition of a canonical forest does allow an arbitrary number
of isolated roots; by definition, those are leaves and not internal
vertices. In order to use Lemma~\ref{lemma:number-canonical-forests} and
Theorem~\ref{theorem:distribution-m} for our cases, we use the simple bijection
between forests with $n$ internal vertices and $r$ roots all of which are
non-isolated and forests with $n-r$ internal vertices and $rt$ roots realised
by omitting all $r$ roots.

With $Q=\frac12 + (\log 2)/(2t) + 0.06/t^2$
(Lemma~\ref{lemma:dominant-singularity}), $q_*=q_0^{1+1/(t-1)}$
(Proposition~\ref{proposition:p_r-bounds}) and $U=1-(\log 2)/t^2$
(Theorem~\ref{theorem:distribution-m}), we fix $0<\delta<1/4$ such that
\begin{equation}\label{eq:delta-definition}
  \delta j\biggl(\frac{U}{q_*^{1/t}}\biggr)^{\delta j}<Q^{-j/4}
\end{equation}
holds for all $j\ge 1$. We first compute the probability to have at
least $m_1t \ge \delta(s_2-s_1)$ vertices at the level of $v_{s_1}$. To do so, we
use the decomposition as described above with the following modification: we do
not use the full decomposition into $(T_1, T_2, T_3)$, but join the latter
two to have a decomposition $(T_1, T_{23})$ in the obvious way. By
Lemma~\ref{lemma:number-canonical-forests} and
Theorem~\ref{theorem:distribution-m} we have
\begin{equation*}
  \nu(1) q_0^{-n_1} O(U^{m_1t})=O(q_0^{-n_1} U^{m_1t})
\end{equation*}
canonical trees~$T_1$ with $m_1t$ leaves, and there are
\begin{equation*}
  \nu(tr_2)q_0^{-n_2-n_3+r_2}(1+O(Q^{n_2+n_3-r_2}))=O(q_0^{-n_2-n_3})
\end{equation*}
canonical forests~$T_{23}$. Note that we used $\nu(tr_2)=\Theta(1)$
(see Lemma~\ref{lemma:number-canonical-forests}) and
$q_0^{r_2}\leq1$. Therefore, using $U<1$, we find the
desired probability to be
\begin{align*}
  \frac{1}{\nu(1)q_0^{-n}(1+O(Q^n))}&\sum_{m_1\ge \delta(s_2-s_1)/t
    }\sum_{f_\lambda}O(q_0^{-n_1}U^{m_1t}q_0^{-n_2-n_3})\\
  &=\sum_{m_1\ge
    \delta(s_2-s_1)/t}\sum_{f_\lambda}O(U^{m_1t})\\
  &= \sum_{m_1\ge
    \delta(s_2-s_1)/t}O(m_1t\,U^{m_1t})\\
  &= \f[big]{O}{\delta(s_2-s_1)U^{\delta(s_2-s_1)}}
  =\f[big]{O}{(s_2-s_1)U^{\delta(s_2-s_1)}}.
\end{align*}

Analogously, the probability that there are at least $\delta(s_2-s_1)$ vertices
at the level of~$v_{s_2}$ is also $\f[big]{O}{(s_2-s_1)U^{\delta (s_2-s_1)}}$. In
particular, the probability that $v_{s_1}$ and $v_{s_2}$ are on the same level
is bounded by $\f[big]{O}{(s_2-s_1)U^{\delta(s_2-s_1)}}$. From now on, we consider the
event $W$ that $v_{s_1}$ and $v_{s_2}$ are on different levels and that there
at most $\delta(s_2-s_1)$ vertices at each of the levels of $v_{s_1}$ and
$v_{s_2}$, respectively. The previous discussion shows that
\begin{equation}\label{eq:W-probability}
  \P(W)\ge 1-O((s_2-s_1)U^{\delta(s_2-s_1)}).
\end{equation}

Let now two events $A \in \mathcal{F}_{s_1}$ in the $\sigma$-algebra generated
by $Z_{1,n},\ldots,Z_{s_1,n}$ and $B \in \mathcal{G}_{s_2}$ in the
$\sigma$-algebra generated by $Z_{s_2,n},Z_{s_2+1,n},\ldots,Z_{n-1,n}$ be
given. The event $A$ consists of a collection of possible shapes of the random
tree $T$ up to the $s_1$-th vertex $v_{s_1}$, and likewise $B$ consists of a
collection of possible shapes of the random tree $T$ from the $s_2$-th vertex
$v_{s_2}$ onwards. For ease of presentation, we assume that the events $A$ and
$B$ consist of only one such shape up to $s_1$ and from $s_2$ on, respectively;
the general case follows upon summation over all shapes in $A$ and $B$. The
shapes $A$ and $B$ uniquely determine $F_\lambda(T) \equalscolon f_\lambda$ and
$G_\rho(T) \equalscolon g_\rho$,
respectively. On the other hand, $F_\rho(T)$ and $G_\lambda(T)$ will be somewhat restricted
by the shapes in $A$ and $B$, respectively.

Using Lemma~\ref{lemma:number-canonical-forests}, Theorem~\ref{theorem:distribution-m} and
the bijection into a tree and forests described above yields the following estimates for the probabilities we are interested in.
There, the error term $O(Q^n)$ in the denominator will always be absorbed by
the error term in the numerator because $Q^n\le Q^{s_2-s-1}$.
We obtain
\begin{align*}
  \P(A\cap W)&=\frac1{\nu(1)q_0^{-n}(1+O(Q^n))}
  \sum_{f_\rho} \nu(tr_2)q_0^{-n_2-n_3+r_2}(1+O(Q^{n_2+n_3-r_2}))\\
  &=\sum_{f_\rho}\frac{\nu(tr_2)}{\nu(1)}q_0^{n_1+r_2}(1+O(Q^{(1-\delta)(s_2-s_1)}))
\end{align*}
using the inequalities $Q<1$, $r_2\leq \delta(s_2-s_1)$ (since we are in the situation that event $W$ occurs) and $n_2+n_3 \geq
s_2-s_1$. We also get
\begin{align*}
  \P(B\cap W)&=\frac1{\nu(1)q_0^{-n}(1+O(Q^n))}
  \sum_{g_\lambda}\sum_{r_3/t\le m_2\le \delta(s_2-s_1)/t}
  \nu(1)q_0^{-n_1-n_2}\\
  &\qquad\qquad\times p_{m_2}\Biggl(1+O\Biggr(Q^{n_1+n_2}m_2\biggl(\frac{U^t}{q_*}\biggr)^{m_2}\Biggr)\Biggr)\\
  &=\sum_{g_\lambda}\sum_{r_3/t\le m_2\le \delta(s_2-s_1)/t}
  q_0^{n_3}p_{m_2}\Bigl(1+\f[big]{O}{Q^{(3/4-\delta)(s_2-s_1)}}\Bigr)
\end{align*}
by \eqref{eq:delta-definition} with $j=s_2-s_1$ and the inequalities $g_\lambda\le
\delta(s_2-s_1)$ (again because $W$ occurs) and $n_1+n_2\ge n_2\ge (1-\delta)(s_2-s_1)$.
Similarly, we calculate the probability that all three events $A$, $B$
and $W$ occur simultaneously as
\begin{align*}
  \P(A\cap B\cap W)&=\frac1{\nu(1)q_0^{-n}(1+O(Q^n))}\sum_{f_\rho,
  g_\lambda}\sum_{r_3/t\le m_2\le\delta(s_2-s_1)/t}\nu(tr_2)q_0^{-n_2+r_2}\\
&\qquad\qquad\times p_{m_2}\Biggl(1+O\Biggl(Q^{n_2-r_2}m_2\biggl(\frac{U^t}{q_*}\biggr)^{m_2}\Biggr)\Biggr)\\
&=\sum_{f_\rho,
  g_\lambda}\,\sum_{r_3/t\le m_2\le \delta(s_2-s_1)/t}\frac{\nu(tr_2)}{\nu(1)}q_0^{n_1+n_3+r_2}p_{m_2}(1+O(Q^{(3/4-2\delta)(s_2-s_1)})),
\end{align*}
where we additionally used $r_2\leq \delta(s_2-s_1)$. We conclude that
\begin{align*}
  \abs{\P(A\cap  W)\P(B\cap W)-\P(A\cap B\cap W)}
  &\le \sum_{f_\rho,
  g_\lambda}\sum_{r_3/t\le m\le
  \delta(s_2-s_1)}\frac{\nu(tr_2)}{\nu(1)}q_0^{n_1+n_3+r_2}p_m\\
&\qquad\qquad\times
O(Q^{(3/4-2\delta)(s_2-s_1)}))\\
&= O(Q^{(3/4-2\delta)(s_2-s_1)})\P(A\cap B\cap W).
\end{align*}
Combining this with~\eqref{eq:W-probability} yields the strong mixing property
\eqref{eq:mixing}.
\end{proof}

Now we are able to apply the following result of Sunklodas.

\begin{lemma}[{Sunklodas~\cite{sunklodas1984rate}}]\label{lemma:sunklodas}
Let $d$, $s\in(2, 3]$, $\kappa$, $\lambda$, $c_0$ be fixed positive constants. Then
there exists a constant $K$ such that for all positive integers $n$ and random
variables $X_1$, $X_2$, $\ldots$, $X_n$ the following holds:

If
\begin{enumerate}
\item $\E(X_j) = 0$ for all $j$,
\item $\max_{1 \leq j \leq n} \E(|X_j|^s)\le d$,
\item the strong mixing condition
  \begin{equation*}
    \sup_{\substack{A \in \mathcal{F}_{t},\,
        B \in \mathcal{G}_{t+\tau} \\ 1 \leq t \leq n-\tau}}
    |P(A \cap B) - P(A)P(B)| \leq \kappa e^{-\lambda \tau}
  \end{equation*}
  holds for all $\tau$ (where $\mathcal{F}_{t}$ and $\mathcal{G}_{t+\tau}$ are the
  $\sigma$-algebras generated by $X_1$, \ldots, $X_t$ and by $X_{t+\tau}$,
  \ldots,  $X_{n}$, respectively) and
\item the inequality
  \begin{equation*}
    B_n^2 = \Var \biggl( \sum_{j=1}^n X_j \biggr) \geq c_0 n
  \end{equation*}
  holds,
\end{enumerate}
then we have 
\begin{equation*}
\sup_x \biggl\lvert \P\biggl(\frac{1}{B_n}\sum_{j=1}^n X_j < x \biggr)  - \Phi(x)\biggr\rvert \leq \frac{K (\log(B_n/\sqrt{c_0}))^{s-1}}{B_n^{s-2}},
\end{equation*}
where $\Phi(x)=(2\pi)^{-1/2} \int_{-\infty}^x e^{-u^2/2}\,du$ denotes the
distribution function of a standard normal distribution.
\end{lemma}

\begin{remark}
Actually, Sunklodas gives a stronger statement where $\lambda$ is not necessarily constant, but we will only need this version. Moreover, he technically considers an infinite sequence $X_1,X_2,\ldots$ of random variables and assumes that the conditions above hold for all $n$. However, the statement gives an explicit inequality for each fixed $n$, and the proof of this inequality given in~\cite{sunklodas1984rate} only makes use of the conditions for the same fixed $n$. This is important for us, since we are not considering an infinite sequence, but rather a finite sequence of $n$ random variables that all depend on $n$.
\end{remark}

\begin{proof}[Proof of Theorem~\ref{thm:pathlength}]
  The qualitative behavior of the asymptotics of mean and variance follows from
  the moment generating functions $L_1$ and $L_2$ (see
  Lemmata~\ref{lemma:L_1-functional-equation} and~\ref{lem:tpl:L2}) by using
  the standard tools from singularity
  analysis~\cite{Flajolet-Sedgewick:ta:analy}, as explained earlier. We get the constants
  $\mu_{\mathit{tpl}}$ and $\sigma_{\mathit{tpl}}$ from
  Proposition~\ref{pro:mu_tpl:mu_h} and Lemma~\ref{lem:sigma_tpl}, respectively, by
  inserting either the asymptotic expansion of $q_0$, cf.\
  Lemma~\ref{lemma:dominant-singularity}, or the values of $q_0$ for given $t$
  (see also Remark~\ref{rem:get-asy-tpl}).

Asymptotic normality follows from Sunklodas's result (Lemma~\ref{lemma:sunklodas}) applied to the sequence $X_j = Z_{j,n} - \E (Z_{j,n}) = \frac{n-j}{n} (Y_{j,n} - \E(Y_{j,n}))$, where $Y_{j,n},Z_{j,n}$ are defined as explained earlier in this section. Since $|X_j|$ is bounded by $1$, the first condition of Lemma~\ref{lemma:sunklodas} is trivially satisfied (for any $s$). The second condition (strong mixing property) is exactly Lemma~\ref{lemma:mixing-condition}, and finally we already know that the variance of the sum $n^{-1} \sum_{j=1}^n X_j$, which is equal to the variance of  $n^{-1} \ell_{\mathit{internal}}(T)$, is of linear order, because the variance of $\ell_{\mathit{internal}}$ is of cubic order.

Since the upper bound for $\Delta_n$ in Lemma~\ref{lemma:sunklodas} goes to $0$ as $n \to \infty$, it follows that the distribution of $\ell_{\mathit{internal}}$ (suitably renormalised) converges weakly to a Gaussian distribution. We can even conclude that the speed of convergence is $O(n^{-1/2}\log n)$.
\end{proof}

\bibliographystyle{amsplainurl}
\bibliography{cheub,stephan}

\providecommand{\Submitted}{Submitted} \providecommand{\availableat}{ available
  at } \providecommand{\alsoavailableat}{ also available at }
  \providecommand{\evavailableat}{earlier version available at }
  \providecommand{\toappearin}{To appear in } \providecommand{\toappear}{to
  appear} \providecommand{\inpreparation}{in preparation}
  \providecommand{\doi}[1]{\href{http://dx.doi.org/#1}{\path{doi:#1}}}
  \providecommand{\etc}{\emph{etc.}}\def\cprime{$'$}
  \def\polhk#1{\setbox0=\hbox{#1}{\ooalign{\hidewidth
  \lower1.5ex\hbox{`}\hidewidth\crcr\unhbox0}}}
\providecommand{\bysame}{\leavevmode\hbox to3em{\hrulefill}\thinspace}
\providecommand{\MR}{\relax\ifhmode\unskip\space\fi MR }
\providecommand{\MRhref}[2]{%
  \href{http://www.ams.org/mathscinet-getitem?mr=#1}{#2}
}
\providecommand{\href}[2]{#2}
\begin{thebibliography}{10}

\bibitem{bender2005locally}
Edward~A. Bender and E.~Rodney Canfield,
  \href{http://www.combinatorics.org/Volume_12/Abstracts/v12i1r57.html}{\emph{Locally
  restricted compositions. {I}. {R}estricted adjacent differences}}, Electron.
  J. Combin. \textbf{12} (2005), Research Paper 57, 27 pp. \MR{2180794
  (2007b:05019)}

\bibitem{bender2009locally}
\bysame,
  \href{http://www.combinatorics.org/Volume_16/Abstracts/v16i1r108.html}{\emph{Locally
  restricted compositions. {II}. {G}eneral restrictions and infinite
  matrices}}, Electron. J. Combin. \textbf{16} (2009), no.~1, Research Paper
  108, 35 pp. \MR{2539350 (2010j:05021)}

\bibitem{bender2010locally}
\bysame,
  \href{http://www.combinatorics.org/Volume_17/Abstracts/v17i1r145.html}{\emph{Locally
  restricted compositions {III}. {A}djacent-part periodic inequalities}},
  Electron. J. Combin. \textbf{17} (2010), no.~1, Research Paper 145, 9 pp.
  \MR{2745698 (2011k:05017)}

\bibitem{bender2012locally}
Edward~A. Bender, E.~Rodney Canfield, and Zhicheng Gao, \emph{Locally
  restricted compositions {IV}. {N}early free large parts and gap-freeness},
  Electron. J. Combin. \textbf{19} (2012), no.~4, Paper 14, 29. \MR{3001651}

\bibitem{Boyd:1975}
David~W. Boyd, \emph{The asymptotic number of solutions of a diophantine
  equation from coding theory}, J. Combinatorial Theory Ser. A \textbf{18}
  (1975), 210--215. \MR{0360437 (50 \#12887)}

\bibitem{drmota2009height}
Michael Drmota, \href{http://dx.doi.org/10.1007/s00026-009-0009-x}{\emph{The
  height of increasing trees}}, Ann. Comb. \textbf{12} (2009), no.~4, 373--402.
  \MR{2496124 (2011b:05049)}

\bibitem{drmota2009random}
\bysame, \href{http://dx.doi.org/10.1007/978-3-211-75357-6}{\emph{Random
  trees}}, SpringerWienNewYork, Vienna, 2009. \MR{2484382 (2010i:05003)}

\bibitem{drmota1997profile}
Michael Drmota and Bernhard Gittenberger,
  \href{http://dx.doi.org/10.1002/(SICI)1098-2418(199707)10:4<421::AID-RSA2>3.3.CO;2-P}{\emph{On
  the profile of random trees}}, Random Structures Algorithms \textbf{10}
  (1997), no.~4, 421--451. \MR{1608230 (99c:05176)}

\bibitem{drmota2005profiles}
Michael Drmota and Hsien-Kuei Hwang,
  \href{http://dx.doi.org/10.1239/aap/1118858628}{\emph{Profiles of random
  trees: correlation and width of random recursive trees and binary search
  trees}}, Adv. in Appl. Probab. \textbf{37} (2005), no.~2, 321--341.
  \MR{2144556 (2006a:60012)}

\bibitem{Eaves:1970}
Reuben~E. Eaves, \emph{A sufficient condition for the convergence of an
  infinite determinant.}, SIAM J. Appl. Math. \textbf{18} (1970), 652--657.
  \MR{0258853 (41 \#3498)}

\bibitem{Elsholtz-Heuberger-Prodinger:2013:huffm}
Christian Elsholtz, Clemens Heuberger, and Helmut Prodinger,
  \href{http://dx.doi.org/10.1109/TIT.2012.2226560}{\emph{The number of
  {H}uffman codes, compact trees, and sums of unit fractions}}, IEEE Trans.
  Inf. Theory \textbf{59} (2013), 1065--1075. \MR{3015716}

\bibitem{flajolet1993distribution}
Philippe Flajolet, Zhicheng Gao, Andrew Odlyzko, and Bruce Richmond,
  \href{http://dx.doi.org/10.1017/S0963548300000560}{\emph{The distribution of
  heights of binary trees and other simple trees}}, Combin. Probab. Comput.
  \textbf{2} (1993), no.~2, 145--156. \MR{1249127 (94k:05061)}

\bibitem{Flajolet-Odlyzko:1990:singul}
Philippe Flajolet and Andrew Odlyzko,
  \href{http://dx.doi.org/10.1137/0403019}{\emph{Singularity analysis of
  generating functions}}, SIAM J. Discrete Math. \textbf{3} (1990), 216--240.
  \MR{MR1039294 (90m:05012)}

\bibitem{Flajolet-Prodinger:1987:level}
Philippe Flajolet and Helmut Prodinger,
  \href{http://dx.doi.org/10.1016/0012-365X(87)90137-3}{\emph{Level number
  sequences for trees}}, Discrete Math. \textbf{65} (1987), no.~2, 149--156.
  \MR{893076 (88e:05030)}

\bibitem{Flajolet-Sedgewick:ta:analy}
Philippe Flajolet and Robert Sedgewick,
  \href{http://dx.doi.org/10.1017/CBO9780511801655}{\emph{Analytic
  combinatorics}}, Cambridge University Press, Cambridge, 2009.

\bibitem{Godsil-Royle:2001:alggraphtheory}
Chris~D. Godsil and Gordon Royle, \emph{Algebraic graph theory}, Graduate texts
  in mathematics, vol. 207, Springer Verlag (New York), 2001.

\bibitem{Graham-Knuth-Patashnik:1994}
Ronald~L. Graham, Donald~E. Knuth, and Oren Patashnik, \emph{Concrete
  mathematics. {A} foundation for computer science}, second ed.,
  Addison-Wesley, 1994.

\bibitem{Heuberger-Krenn-Wagner:2013:analy-param}
Clemens Heuberger, Daniel Krenn, and Stephan Wagner,
  \href{http://knowledgecenter.siam.org/0238-000001/}{\emph{Analysis of
  parameters of trees corresponding to {H}uffman codes and sums of unit
  fractions}}, Proceedings of the Meeting on Analytic Algorithmics \&
  Combinatorics (ANALCO), New Orleans, Louisiana, USA, January 6, 2013, SIAM,
  Philadelphia PA, 2013, pp.~33--42.

\bibitem{Komlos-Moser-Nemetz:1984}
J.~Koml{\'o}s, W.~Moser, and T.~Nemetz, \emph{On the asymptotic number of
  prefix codes}, Mitt. Math. Sem. Giessen (1984), no.~165, 35--48. \MR{745868
  (86a:94009)}

\bibitem{Kraft:1949:thesis}
Leon~G. Kraft, \emph{A device for quantizing, grouping, and coding amplitude
  modulated pulses}, Master's thesis, Massachusetts Institute of Technology,
  1949.

\bibitem{mahmoud1991limiting}
Hosam~M. Mahmoud,
  \href{http://dx.doi.org/10.1017/S0269964800001881}{\emph{Limiting
  distributions for path lengths in recursive trees}}, Probab. Engrg. Inform.
  Sci. \textbf{5} (1991), no.~1, 53--59. \MR{1183165 (93g:60019)}

\bibitem{McMillan:1956:inequalities}
Brockway McMillan, \href{http://dx.doi.org/10.1109/TIT.1956.1056818}{\emph{Two
  inequalities implied by unique decipherability}}, Information Theory, IRE
  Transactions on \textbf{2} (1956), no.~4, 115--116.

\bibitem{Richmond-Knopfmacher:1995:compos}
Bruce Richmond and Arnold Knopfmacher,
  \href{http://dx.doi.org/10.1007/BF01827930}{\emph{Compositions with distinct
  parts}}, Aequationes Math. \textbf{49} (1995), no.~1--2, 86--97. \MR{1309295
  (95k:11133)}

\bibitem{Stein-others:2015:sage-mathem-6.5}
William~A. Stein et~al., \emph{{S}age {M}athematics {S}oftware ({V}ersion
  6.5)}, The Sage Development Team, 2015, \url{http://www.sagemath.org}.

\bibitem{sunklodas1984rate}
Jonas~Kazys Sunklodas, \emph{The rate of convergence in the central limit
  theorem for strongly mixing random variables}, Litovsk. Mat. Sb. \textbf{24}
  (1984), no.~2, 174--185. \MR{773606 (86k:60041)}

\bibitem{takacs1992total}
Lajos Tak{\'a}cs, \emph{On the total heights of random rooted trees}, J. Appl.
  Probab. \textbf{29} (1992), no.~3, 543--556. \MR{1174430 (93g:05127)}

\bibitem{takacs1993limit}
\bysame, \href{http://dx.doi.org/10.1155/S1048953393000176}{\emph{Limit
  distributions for queues and random rooted trees}}, J. Appl. Math. Stochastic
  Anal. \textbf{6} (1993), no.~3, 189--216. \MR{1238599 (94m:60194)}

\bibitem{tangora1991level}
Martin~C. Tangora,
  \href{http://dx.doi.org/10.1016/S0195-6698(13)80019-4}{\emph{Level number
  sequences of trees and the lambda algebra}}, European J. Combin. \textbf{12}
  (1991), no.~5, 433--443. \MR{1129814 (93d:05012)}

\end{thebibliography}
\appendix
\newpage
\section{Details on the Variance of the Total Path Length}

In this appendix, more details of the proof of Lemma~\ref{lem:sigma_tpl} are
given. Some of these calculations were performed with computer assistance using
SageMath~\cite{Stein-others:2015:sage-mathem-6.5}.

For any variable $z$, we write $\Phi_z=z\frac{\partial}{\partial z}$. Such an
operator satisfies the following properties.

\begin{lemma}\label{lem:op-Phi-prop}
For any expressions $a$, $b$ and any variable $z$, we have
\begin{equation*}
  \Phi_z(a+b)=\Phi_z(a)+\Phi_z(b),\qquad
  \Phi_z(ab)=\Phi_z(a)b+a\Phi_z(b).
\end{equation*}
Moreover, for a function~$f$ we have
\begin{equation*}
  \Phi_z^2 f = \Phi_z \biggl(z \frac{\partial f}{\partial z}\biggr)=
  z \frac{\partial f}{\partial z} + z^2 \frac{\partial^2f }{\partial z^2} =
  \Phi_z f + z^2\frac{\partial^2f}{\partial z^2}.
\end{equation*}
If $(z_1,\ldots, z_k)\mapsto f(z_1, \ldots, z_k)$ is a $k$-ary function and
$a_1$, \ldots, $a_k$ are expressions, we have
\begin{equation*}
  \Phi_z(f(a_1, \ldots, a_k))=\sum_{j=1}^{k}\frac{\partial f}{\partial
    z_j}(a_1,\ldots, a_k)\Phi_{z}(a_j).
\end{equation*}
\end{lemma}

In view of the functions occurring in Section~\ref{sec:path-length}, we have
the following, more specific properties.

\begin{lemma}
Let $(q, u, w)\mapsto f(q, u, w)$ be a function and $j$ be a non-negative integer. We have
\begin{align*}
\Phi_q (f(q, q^{\hp{j}}u^{t^j}, w)) &= (\Phi_q f)(q, q^{\hp{j}}u^{t^j}, w) +
  \frac{\partial f}{\partial u} (q, q^{\hp{j}}u^{t^j}, w) \hp{j}
  q^{\hp{j}}u^{t^j} \\
  &= (\Phi_q f)(q, q^{\hp{j}}u^{t^j}, w)+ \hp{j}(\Phi_u f)(q, q^{\hp{j}}u^{t^j}, w),\\
  \Phi_u(f(q, q^{\hp{j}}u^{t^j}, w))&=\frac{\partial f}{\partial u}(q,
  q^{\hp{j}}u^{t^j}, w) t^j q^{\hp{j}}u^{t^j}= t^j (\Phi_u f)(q, q^{\hp{j}}u^{t^j}, w),\\
  \Phi_w(f(q, q^{\hp{j}}u^{t^j}, w))&=(\Phi_w f)(q, q^{\hp{j}}u^{t^j}, w).
\end{align*}
\end{lemma}

We also need derivatives of the products appearing throughout this article.

\begin{lemma}
Let
\begin{equation*}
  P_j(q, u) = \prod_{i=1}^j \frac{q^{\hp{i}}u^{t^i}}{1-q^{\hp{i}}u^{t^i}}.
\end{equation*}
Then we have
\begin{align*}
  \Phi_u P_j(q, u)&=p_j(q,u)P_j(q, u),\\
  \Phi_q P_j(q, u)&=r_j(q, u)P_j(q, u),
\end{align*}
with
\begin{equation*}
  p_j(q, u)=\sum_{i=1}^j\frac{t^i}{1-q^{\hp{i}}u^{t^i}}.
\end{equation*}
\begin{equation*}
  r_j(q, u)=\sum_{i=1}^j\frac{\hp{i}}{1-q^{\hp{i}}u^{t^i}}.
\end{equation*}
\end{lemma}

\begin{proof}
These results follow since $(\Phi_z f) = f(z)/(1-z)$ for $f(z)=z/(1-z)$.
\end{proof}

Next, we consider the infinite sum
\begin{equation*}
  S(q, u, w, M, f):=\sum_{j\ge 0}(-1)^j w^j M(j, q, u, w) f(q,
  q^{\hp{j}}u^{t^j}, w) P_j(q, u),
\end{equation*}
where the function~$M$ depends on $j$, $q$, $u$ and $w$, and the function~$f$ on
$q$, $u$ and $w$.
Note that this notion is slightly more general than $\f{\Sigma}{q, M, \Phi}$ of
Section~\ref{sec:path-length}. The relationship between these two is
\begin{equation*}
  \f{\Sigma}{q, M, \Phi}
  = \left.S(q, u, w, M, \Phi b)\right\vert_{u=1, w=1},
\end{equation*}
where $b$ is defined in
Theorem~\ref{theorem:generating-function-height-and-others} of this article.

Taking derivatives yields the following results.

\begin{lemma}
We have
\begin{align*}
  \Phi_q S(q, u, w, M, f) &= S(q, u, w, \Phi_q M, f) +
  S(q, u, w, M, \Phi_q f) +
  S(q, u, w, M \hp{j}, \Phi_u f)\\&\qquad +
  S(q, u, w, M r_j(q, u), f),\\
  \Phi_u S(q, u, w, M, f) &= S(q, u, w, \Phi_u M, f) + S(q, u,
  w, M t^j, \Phi_u f)+
  S(q, u, w, M p_j(q, u), f),\\
  \Phi_w S(q, u, w, M, f) &= S(q, u, w, M j, f) + S(q, u, w, \Phi_w M, f) + S(q,
  u, w, M, \Phi_w f),
\end{align*}
where $M \f{g}{j, q, u, w}$ is short for $(j, q, u, w) \mapsto
\f{M}{j, q, u, w} \f{g}{j, q, u, w}$.
\end{lemma}

We are now on our way to derive an expression for $L_2(q, 1, 1)$ suitable for
doing singularity analysis (cf.\ the proof of Lemma~\ref{lem:sigma_tpl}). As a
first step, using the properties above we obtain the following
expression for $L_2(q, 1, 1)$ (only leading terms):
\begin{align*}
&\frac{6 \, (\Phi_w a)\left(q, 1, 1\right) (\Phi_w b)\left(q, 1, 1\right) S\left(q, 1, 1, 1, (\Phi_u b)\left(q, 1, 1\right)\right)^{2}}{(1 - \f{b}{q, 1, 1})^{4}} \\
&+\frac{8 \, (\Phi_w b)\left(q, 1, 1\right) S\left(q, 1, 1, 1, (\Phi_u \Phi_w b)\left(q, 1, 1\right)\right) S\left(q, 1, 1, 1, (\Phi_u b)\left(q, 1, 1\right)\right) a\left(q, 1, 1\right)}{(1 - \f{b}{q, 1, 1})^{4}} \\
&+\frac{6 \, (\Phi_w b)\left(q, 1, 1\right)^{2} S\left(q, 1, 1, 1, (\Phi_u b)\left(q, 1, 1\right)\right)^{2} a\left(q, 1, 1\right)}{(1 - \f{b}{q, 1, 1})^{5}} \\
&+\frac{2 \, (\Phi_w^2 b)\left(q, 1, 1\right) S\left(q, 1, 1, 1, (\Phi_u b)\left(q, 1, 1\right)\right)^{2} a\left(q, 1, 1\right)}{(1 - \f{b}{q, 1, 1})^{4}} \\
&- \frac{2 \, (\Phi_w b)\left(q, 1, 1\right)^{2} S\left(q, 1, 1, 1, (\Phi_u^2 b)\left(q, 1, 1\right)\right) a\left(q, 1, 1\right)}{(1 - \f{b}{q, 1, 1})^{4}} \\
&+\frac{4 \, (\Phi_w b)\left(q, 1, 1\right)^{2} S\left(q, 1, 1, 1, S\left(q, 1, 1, t^{j}, (\Phi_u^2 b)\left(q, 1, 1\right)\right)\right) a\left(q, 1, 1\right)}{(1 - \f{b}{q, 1, 1})^{4}} \\
&+\frac{4 \, (\Phi_w b)\left(q, 1, 1\right)^{2} S\left(q, 1, 1, 1, S\left(q, 1, 1, \f{p_j}{q, 1}, (\Phi_u b)\left(q, 1, 1\right)\right)\right) a\left(q, 1, 1\right)}{(1 - \f{b}{q, 1, 1})^{4}} \\
&+\frac{2 \, (\Phi_w b)\left(q, 1, 1\right) S\left(q, 1, 1, 1, (\Phi_u b)\left(q, 1, 1\right)\right) S\left(q, 1, 1, j, (\Phi_u b)\left(q, 1, 1\right)\right) a\left(q, 1, 1\right)}{(1 - \f{b}{q, 1, 1})^{4}}.
\end{align*}
For readability, we have not written the $\mapsto$ formally needed in the
formula above; for example, the $S$-function in the first summand should read
as
\begin{equation*}
  S\left(q, 1, 1, (j, q, u, w) \mapsto 1,
    (q, u, w) \mapsto (\Phi_u b)\left(q, 1, 1\right)\right)
\end{equation*}
Compared to the formula found in the proof of
Lemma~\ref{lem:sigma_tpl}, nested $S$-functions appear. As a next step, we
simplify these nested $S$-functions by means of the following lemma.

\begin{lemma}\label{lem:nested-S}
We have
\begin{multline*}
  S\left(q, 1, 1, 1, (q, u, w) \mapsto S\left(q, u, 1, (j, q, u, w) \mapsto \f{p_j}{q, u}, f\right)\right) \\
  =S\biggl(q, 1, 1, (j, q, u, w) \mapsto t\sum_{i=1}^{j}\frac{\hp{i}}{1-q^{\hp{i}}}, f\biggr)
\end{multline*}
and
\begin{multline*}
  S\left(q, 1, 1, 1, (q, u, w) \mapsto S\left(q, u, 1, (j, q, u, w) \mapsto t^j, f\right)\right) \\
  =S\biggl(q, 1, 1, (j, q, u, w) \mapsto \hp{j+1}, f\biggr).
\end{multline*}
\end{lemma}

\begin{proof}
We have
\begin{align*}
  S\bigl(q, &1, 1, 1, (q, u, w) \mapsto S\left(q, u, 1, (j, q, u, w) \mapsto M\left(j, q, u, w\right), f\right)\bigr)\\
  &=\sum_{k\ge 0}(-1)^k S(q, q^{\hp{k}}, 1, (j, q, u, w) \mapsto M(j, q, u, w), f) P_k(q, 1)\\
  &=\sum_{k\ge 0}(-1)^k P_k(q, 1) \sum_{j\ge 0}(-1)^j M(j, q, q^{\hp{k}}, w)f(q, q^{\hp{j}}q^{\hp{k}t^j}, 1)P_j(q, q^{\hp{k}})\\
  &=\sum_{j, k\ge 0}(-1)^{k+j}   M(j, q, q^{\hp{k}}, w)f(q, q^{\hp{j+k}}, 1) \prod_{i=1}^k \frac{q^{\hp{i}}}{1-q^{\hp{i}}} \prod_{i=1}^j\frac{q^{\hp{i}}q^{\hp{k}t^i}}{1-q^{\hp{i}}q^{\hp{k}t^i}}\\
  &=\sum_{j, k\ge 0}(-1)^{k+j}   M(j, q, q^{\hp{k}}, w)f(q, q^{\hp{j+k}}, 1) \prod_{i=1}^k \frac{q^{\hp{i}}}{1-q^{\hp{i}}} \prod_{i=1}^j\frac{q^{\hp{i+k}}}{1-q^{\hp{i + k}}}\\
  &=\sum_{j, k\ge 0}(-1)^{k+j}   M(j, q, q^{\hp{k}}, w)f(q, q^{\hp{j+k}}, 1) P_{j+k}(q, 1).
\end{align*}
With the substitution $\ell=j+k$, this equals
\begin{multline*}
  \sum_{\ell\ge 0} (-1)^\ell f(q, q^{\hp{\ell}}, 1) P_{\ell}(q, 1) \sum_{j=0}^{\ell} M(j, q, q^{\hp{\ell-j}}, w)\\
  =S\biggl(q, 1, 1, (j, q, u, w) \mapsto \sum_{k=0}^j M(k, q, q^{\hp{j-k}}, w), f\biggr).
\end{multline*}
We now compute the inner sums occurring in the simplified expressions for the nested $S$-functions.
The second one is simply $\sum_{k=0}^j t^k=\hp{j+1}$. The first one is
\begin{equation*}
  \sum_{k=0}^j p_k(q, q^{\hp{j-k}}) = \sum_{k=0}^j \sum_{i=1}^k\frac{t^i}{1-q^{\hp{i}}q^{\hp{j-k}t^i}}
  =\sum_{1\le i\le k\le j}\frac{t^i}{1-q^{\hp{i+j-k}}}.
\end{equation*}
With the substitution $i+j-k=\ell$, this equals
\begin{equation*}
  \sum_{1\le i\le i+j-\ell\le j}\frac{t^i}{1-q^{\hp{\ell}}}
  =\sum_{1\le i\le \ell\le j} \frac{t^i}{1-q^{\hp{\ell}}}
  =\sum_{\ell=1}^{j}\frac{1}{1-q^{\hp{\ell}}} \sum_{i=1}^{\ell}t^i
  =\sum_{\ell=1}^{j}\frac{t\hp{\ell}}{1-q^{\hp{\ell}}}.
\end{equation*}
The result now follows.
\end{proof}

We continue to rewrite $L_2(q, 1, 1)$. Using the previous lemma, we have
\begin{equation*}
  L_2(q, 1, 1) = \frac{V_5(q)}{(1-b(q, 1, 1))^5} + \frac{V_4(q)}{(1 - b(q, 1, 1))^4} + O((1-b(q, 1, 1))^{-3})
\end{equation*}
for suitable $V_5(q)$ and $V_4(q)$.
Using the fact that $b(q_0, 1, 1) = 1$ and the expression for $\Phi_q^2 f$ of Lemma~\ref{lem:op-Phi-prop},
we get
\begin{align*}
  1-b(q, 1, 1)&=1 - \biggl(1 + (q-q_0)\frac{\partial b}{\partial q}(q_0, 1, 1) + \frac{(q-q_0)^2}{2} \frac{\partial^2 b}{\partial q^2}(q_0, 1, 1) + O\bigl((q-q_0)^3\bigr)\biggr)\\
  &=\biggl(1-\frac{q}{q_0}\biggr)(\Phi_q b)(q_0, 1, 1) - \biggl(1-\frac{q}{q_0}\biggr)^2 \frac{\Phi_q^2b-\Phi_q b}{2}(q_0, 1, 1) + O((q-q_0)^3)\\
  &=\biggl(1-\frac{q}{q_0}\biggr)(\Phi_q b)(q_0, 1, 1)\left(1 - \biggl(1-\frac{q}{q_0}\biggr) \frac{\Phi_q^2b-\Phi_q b}{2\Phi_q b}(q_0, 1, 1) + O((q-q_0)^2)\right).
\end{align*}
We also have
\begin{equation*}
  V_5(q)=V_5(q_0)+(q-q_0)\frac{\partial V_5}{\partial q}(q_0)+O((q-q_0)^2)
  =V_5(q_0)- \biggl(1 - \frac{q}{q_0}\biggr)(\Phi_q V_5)(q_0)+O((q-q_0)^2).
\end{equation*}
Therefore, we obtain
\begin{multline*}
  L_2(q, 1, 1) = \frac{V_5(q_0)}{((\Phi_q b)(q_0, 1, 1))^5}\left(1-\frac{q}{q_0}\right)^{-5}\\
  +
\left(-\frac{(\Phi_q V_5)(q_0)}{((\Phi_q b)(q_0, 1, 1))^5} + \frac{5V_5(q_0) (\Phi_q^2b - \Phi_q b)(q_0, 1, 1)}{2((\Phi_q b)(q_0, 1, 1))^6} +\frac{V_4(q)}{((\Phi_q b)(q_0, 1, 1))^4}\right)\left(1-\frac{q}{q_0}\right)^{-4}\\
+ O\biggl(\biggl(1-\frac{q}{q_0}\biggr)^{-3}\biggr),
\end{multline*}
an expression which is suitable for singularity analysis.

\begin{lemma}
We have
\begin{equation*}
  [q^n]\left(1-\frac{q}{q_0}\right)^{-5} = \frac{n^4q_0^{-n}}{24}\left(1+\frac{10}{n} + O\biggl(\frac{1}{n^2}\biggr)\right) = q_0^{-n}\biggl(
  \frac{n^4}{24}+\frac{5n^3}{12}+O(n^2)\biggr).
\end{equation*}
\end{lemma}
The previous lemma follows directly by expanding into a binomial series. We can use it to
extract coefficients of $L_2(q, 1, 1)$ and obtain
\begin{multline*}
  [q^n]L_2(q, 1, 1)=q_0^{-n}\biggl(
  \frac{V_5(q_0)}{24((\Phi_q b)(q_0, 1, 1))^5 }n^4
  \\+\biggl(\frac{5V_5(q_0)}{12((\Phi_q b)(q_0, 1, 1))^5}
-\frac{(\Phi_q V_5)(q_0)}{6((\Phi_q b)(q_0, 1, 1))^5} + \frac{5V_5(q_0) (\Phi_q^2b - \Phi_q b)(q_0, 1, 1)}{12((\Phi_q b)(q_0, 1, 1))^6} \\
+\frac{V_4(q)}{6((\Phi_q b)(q_0, 1, 1))^4}
\biggr)n^3
  +O(n^2)
\biggr).
\end{multline*}
We conclude that the second moment of the total path length is
\begin{multline*}
  \frac{V_5(q_0)}{24((\Phi_q b)(q_0, 1, 1))^4 a_0(q_0, 1, 1)}n^4\\
  +\biggl(\frac{5V_5(q_0)}{12((\Phi_q b)(q_0, 1, 1))^4}
-\frac{(\Phi_q V_5)(q_0)}{6((\Phi_q b)(q_0, 1, 1))^4} + \frac{5V_5(q_0) (\Phi_q^2b - \Phi_q b)(q_0, 1, 1)}{12((\Phi_q b)(q_0, 1, 1))^5} \\
+\frac{V_4(q)}{6((\Phi_q b)(q_0, 1, 1))^3}
\biggr)\frac{n^3}{a_0(q_0, 1, 1)} + O(n^2).
\end{multline*}
Similarly, writing
\begin{equation*}
  L_1(q, 1, 1) = \frac{E_3(q)}{(1-b(q, 1, 1))^3} + \frac{E_2(q)}{(1 - b(q, 1, 1))^2} + O((1-b(q, 1, 1))^{-1})
\end{equation*}
and performing singularity analysis shows that the expectation is
\begin{multline*}
  \frac{E_3(q_0)}{2((\Phi_q b)(q_0, 1, 1))^2 a_0(q_0, 1, 1)}n^2\\
  +\biggl(\frac{3E_3(q_0)}{2((\Phi_q b)(q_0, 1, 1))^2}
-\frac{(\Phi_q E_3)(q_0)}{((\Phi_q b)(q_0, 1, 1))^2} + \frac{3E_3(q_0) (\Phi_q^2b - \Phi_q b)(q_0, 1, 1)}{2((\Phi_q b)(q_0, 1, 1))^3} \\
+\frac{E_2(q)}{(\Phi_q b)(q_0, 1, 1)}
\biggr)\frac{n}{a_0(q_0, 1, 1)} + O(1).
\end{multline*}
From the results above an expression for the constant $\sigma_{\mathit{tpl}}^2$ that occurs in the asymptotic formula for the
variance follows. Using Lemma~\ref{lem:nested-S} to rewrite the nested
S-functions gives the result that was stated in Lemma~\ref{lem:sigma_tpl}.
\end{document}